\documentclass[11pt, reqno]{amsart}
\usepackage{color}
\RequirePackage[colorlinks,citecolor=blue,urlcolor=blue]{hyperref}
\usepackage{amsfonts,amsmath, amssymb, amsthm, amscd}
\usepackage{bm}
\usepackage{mathrsfs}
\usepackage{yfonts}
\usepackage{verbatim}
\usepackage{graphicx}
\usepackage[utf8]{inputenc}
\usepackage{latexsym}
\usepackage{lscape}
\usepackage{epsfig}
\usepackage{subfigure}
\usepackage{dsfont}
\usepackage{mathdots}
\usepackage{mathtools}
\usepackage{upgreek}
\usepackage{tikz}
\usetikzlibrary{shapes.multipart}
\usetikzlibrary{patterns}
\usetikzlibrary{shapes.multipart}
\usetikzlibrary{arrows}
\usetikzlibrary{decorations.markings}
\usepackage{fullpage}

\newcommand{\dd}{\mathbf{d}}
\newcommand{\EE}{\ensuremath{\mathbb{E}}}

\newcommand{\PP}{\ensuremath{\mathbb{P}}}
\newcommand{\R}{\ensuremath{\mathbb{R}}}

\newcommand{\C}{\ensuremath{\mathbb{C}}}
\newcommand{\Z}{\ensuremath{\mathbb{Z}}}

\newcommand{\Real}{\ensuremath{\mathfrak{Re}}}
\renewcommand{\Re}{\ensuremath{\mathfrak{Re}}}

\def \Ai {{\rm Ai}}
\def \sgn {{\rm sgn}}

\newtheorem{theorem}{Theorem}[section]

\newtheorem{lemma}[theorem]{Lemma}
\newtheorem{proposition}[theorem]{Proposition}

\theoremstyle{definition}
\newtheorem{remark}[theorem]{Remark}
\theoremstyle{definition}

\theoremstyle{definition}
\newtheorem{definition}[theorem]{Definition}
\theoremstyle{definition}

\newcommand{\PSP}{\ensuremath{\mathbf{PSP}}}
\newcommand{\PSM}{\ensuremath{\mathbf{PSM}}}
\newcommand{\Plancherel}{\ensuremath{\mathrm{Plancherel}}}
\newcommand{\Pf}{\ensuremath{\mathrm{Pf}}}

\newcommand{\Geom}{\ensuremath{\mathrm{Geom}}}
\newcommand{\sign}{\ensuremath{\mathrm{sign}}}
\newcommand{\Conf}{\ensuremath{\mathrm{Conf}}}
\newcommand{\I}{\ensuremath{\mathbf{i}}}
\newcommand{\skeww}{\ensuremath{\mathrm{Skew}_2}}

\begin{document}
\title[Pfaffian Schur processes and LPP in a half-quadrant]{Pfaffian Schur processes and last passage percolation in a half-quadrant}

\author[J. Baik]{Jinho Baik}
\address{J. Baik, University of Michigan, Department of Mathematics,
	530 Church Street,
	Ann Arbor, MI 48109, USA}
\email{baik@umich.edu}

\author[G. Barraquand]{Guillaume Barraquand}
\address{G. Barraquand,
	Columbia University,
	Department of Mathematics,
	2990 Broadway,
	New York, NY 10027, USA}
\email{barraquand@math.columbia.edu}

\author[I. Corwin]{Ivan Corwin}
\address{I. Corwin, Columbia University,
	Department of Mathematics,
	2990 Broadway,
	New York, NY 10027, USA,
	and Clay Mathematics Institute, 10 Memorial Blvd. Suite 902, Providence, RI 02903, USA}
\email{ivan.corwin@gmail.com}

\author[T. Suidan]{Toufic Suidan}
\address{T. Suidan}
\email{tsuidan@gmail.com}



\begin{abstract}
We study last passage percolation in a half-quadrant, which we analyze within the framework of Pfaffian Schur processes. For the model with exponential weights, we prove that the fluctuations of the last passage time to a point on the diagonal are either GSE Tracy-Widom distributed, GOE Tracy-Widom distributed, or Gaussian, depending on the size of weights along the diagonal. Away from the diagonal, the fluctuations of passage times follow the  GUE Tracy-Widom distribution. We also obtain a two-dimensional crossover between the GUE, GOE and GSE distribution by studying the multipoint distribution of last passage times close to the diagonal when the size of the diagonal weights is simultaneously scaled close to the critical point. We expect that this crossover arises universally in KPZ growth models in half-space. Along the way, we introduce a method to deal with diverging correlation kernels of point processes where points collide in the scaling limit. 
\end{abstract}

\maketitle

\setcounter{tocdepth}{1}
\hypersetup{linktocpage}
\tableofcontents

\section{Introduction}

In this paper, we study last passage percolation in a half-quadrant of $\Z^2$. We extend all known results for the case of geometric weights (see discussion of previous results below) to the case of exponential weights. We study in a unified framework the distribution of passage times on and off diagonal, and for arbitrary boundary condition. We do so by realizing the joint distribution of last passage percolation times as a marginal of Pfaffian Schur processes. This allows us to use powerful methods of Pfaffian point processes to prove limit theorems. Along the way, we discuss an issue arising when a simple point process converges to a point process where every point has multiplicity two, which we expect to have independent interest. These results also have consequences about interacting particle systems -- in particular the TASEP on positive integers and the facilitated TASEP -- that we discuss in \cite{baik2017facilitated}. 

\begin{definition}[Half-space exponential weight LPP] Let $\big(w_{n,m}\big)_{n\geqslant m\geqslant 0}$ be a sequence of  independent exponential random variables\footnote{The exponential distribution with rate $\alpha\in(0,+\infty)$, denoted $\mathcal{E}(\alpha)$,  is the probability distribution on $\R_{>0}$ such that if $X\sim \mathcal{E}(\alpha)$,
		$ \PP(X>x) = e^{-\alpha x}.$} with rate $1$  when $n\geqslant m+1$ and with rate $\alpha$ when  $n=m$. We define the exponential last passage percolation (LPP) time on the half-quadrant, denoted $H(n,m)$, by the recurrence for $n\geqslant m$,
	$$ H(n,m) =  w_{n,m} + \begin{cases}
	\max\Big\lbrace H(n-1, m)  , H(n, m-1)\Big\rbrace &\mbox{if } n\geqslant m+1, \\
	H(n,m-1) &\mbox{if } n=m
	\end{cases}$$
	with the boundary condition $H(n,0)=0$.
	\label{def:LPPexp}
\end{definition}

\begin{figure}
	\begin{tikzpicture}[scale=0.55]
	\draw[thick, gray,  ->] (0,0) -- (12.2,0);
	\draw[thick, gray, ->] (0,0) -- (0,8.2);
	\draw[thick, gray] (0,1) -- (1,1) -- (1,2) -- (2,2) -- (2,3) -- (3,3) -- (3,4) -- (4,4) -- (4,5) -- (5,5) -- (5,6) -- (6,6) -- (6,7) -- (7,7) -- (7,8) -- (8,8)  ; 
	\foreach \x in {1, ..., 8}
	\draw[dashed, gray] (\x,0) -- (\x,\x);
	\foreach \x in {9, 10, 11, 12}
	\draw[dashed, gray] (\x,0) -- (\x,8);
	\foreach \x in {1, ..., 8}
	\draw[dashed, gray] (\x,\x) -- (12.1,\x);
	\fill[gray, opacity=0.5] (0,0) -- (0,1) -- (3,1) -- (3,2) -- (3,3) -- (3,4) -- (4,4) -- (4,5) -- (5,5) -- (6,5) -- (6,6) -- (7,6) -- (7,5)-- (7,4) -- (6,4) -- (5,4) -- (5,3) -- (4,3) -- (4,2) -- (4,0) --  (0,0);
	\fill[gray, opacity=0.7] (6,6) -- (7,6) -- (7,5) -- (6,5) -- (6,6);
	\node at (6.5, -0.5) {$n$};
	\node at (-0.5, 5.5) {$m$};
	\draw [thick, dashed] (0,5.5) -- (6.5,5.5);
	\draw [thick, dashed] (6.5,0) -- (6.5,5.5);
	\node at (0.5, 0.5) {\footnotesize{$w_{11}$}};
	\node at (1.5, 0.5) {\footnotesize{$w_{21}$}};
	\node at (1.5, 1.5) {\footnotesize{$w_{22}$}};
	\node at (2.5, 0.5) {\footnotesize{$w_{31}$}};
	\draw[ultra thick, black] (6,6) -- (7,6) -- (7,4) -- (8,4) -- (8,3) -- (10, 3) -- (10,2) -- (11,2) -- (11,0) ;
	\fill[gray, opacity=0.3] (0,0)-- (0,1) -- (1,1) -- (1,2) -- (2,2) -- (2,3) -- (3,3) -- (3,4) -- (4,4) -- (4,5) -- (5,5) -- (5,6) -- (6,6) -- (7,6) -- (7,4) -- (8,4) -- (8,3) -- (10, 3) -- (10,2) -- (11,2) -- (11,0);
	\end{tikzpicture}
	\caption{LPP on the half-quadrant. One admissible path from $(1,1)$ to $(n,m)$ is shown in dark gray. $H(n,m)$ is the maximum over such paths of the sum of the weights $w_{ij}$ along the path.}
	\label{fig:lastpassagehalfquadrant}
\end{figure}
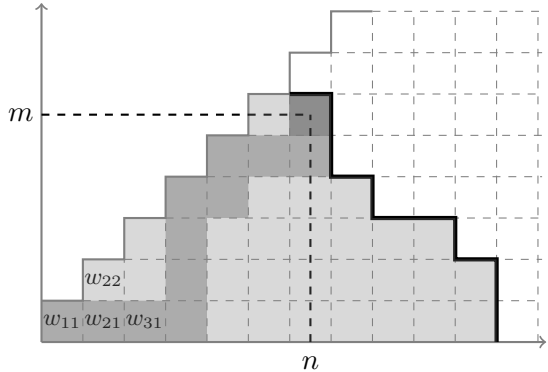

It is useful to notice that the model is equivalent to a model of last passage percolation on the full quadrant where the weights are symmetric with respect to the diagonal ($w_{ij}=w_{ji}$). 

\begin{remark}
	If one imagines that the light gray area in Figure \ref{fig:lastpassagehalfquadrant} corresponds to the percolation cluster at some time, its border (shown in black in Figure \ref{fig:lastpassagehalfquadrant}) evolves as the height function in the Totally Asymmetric Simple Exclusion process on the positive integers with open boundary condition, so that all our results could be equivalently phrased in terms of the latter model.
\end{remark}

\subsection{Previous results in (half-space) LPP}
The history of fluctuation results in last passage percolation goes back to the solution to Ulam's problem about the asymptotic distribution of the size of the longest increasing subsequence in a uniformly random permutation \cite{baik1999distribution}. A proof of the nature of fluctuations in exponential distribution LPP is provided in  \cite{johansson2000shape}. On a half-space, \cite{baik2001algebraic, baik2001asymptotics, baik2001symmetrized} studies the longest increasing subsequence in random involutions. This is equivalent to a half-space LPP problem since for an involution $\sigma\in \mathcal{S}_n$, the graph of $i\mapsto \sigma(i)$ is symmetric with respect to the first diagonal. Actually, \cite{baik2001algebraic, baik2001asymptotics,baik2001symmetrized} treat both the longest increasing subsequences of a random involution and symmetrized LPP with geometric weights. That work contains the geometric weight half-space LPP analogue of Theorem \ref{theo:LPPdiagointro}. The methods used therein only worked when restricted to the one point distribution of passage times exactly along the diagonal.
Further work on half-space LPP was undertaken in  \cite{sasamoto2004fluctuations}, where the results are stated there in terms of the discrete polynuclear growth (PNG) model rather than the equivalent half-space LPP with geometric weights. The framework of \cite{sasamoto2004fluctuations} allows to study the correlation kernel corresponding to multiple points in the space direction, with or without nucleations at the origin. In the large scale limit, the authors performed a non-rigorous critical point derivation of the limiting correlation kernel in certain regimes and found Airy-type process and various limiting crossover kernels near the origin, much in the same spirit as our results (see the discussion about these results in Section \ref{sec:new}).

\subsection{Previous methods}
A key property in the solution of Ulam's problem in \cite{baik1999distribution} was the relation to the RSK algorithm which reveals a beautiful algebraic structure that can be generalized using the formalism of Schur measures \cite{okounkov2001infinite} and Schur processes \cite{okounkov2003correlation}. RSK is just one of a variety of Markov dynamics on sequences of partitions which preserves the Schur process \cite{borodin2008anisotropic, borodin2011schur}. Studying these dynamics gives rise to a number of interesting probabilistic models related to Schur processes. In the early works of \cite{baik1999distribution, baik2001asymptotics} the convergence to the Tracy-Widom distributions was proved by Riemann-Hilbert problem asymptotic analysis. Subsequently Fredholm determinants became the preferred vehicle for asymptotic analysis.

In a half-space, \cite{baik2001algebraic} applied RSK to symmetric matrices. Subsequently, \cite{rains2000correlation} (see also \cite{forrester2006correlation}) showed that geometric weight half-space LPP is a Pfaffian point process (see Section \ref{sec:PPP}) and computed the correlation kernel. Later, \cite{borodin2005eynard} defined Pfaffian Schur processes generalizing the probability measures defined in \cite{rains2000correlation} and \cite{sasamoto2004fluctuations}. The Pfaffian Schur process is an analogue of the Schur process for symmetrized problems.

\subsection{Our methods and new results}
\label{sec:new}
We consider Markov dynamics similar to those of \cite{borodin2008anisotropic} and show that they preserve Pfaffian Schur processes.  This enables us to show how half-space LPP with geometric weights is a marginal of the Pfaffian Schur process. We then apply a general formula giving the correlation kernel of Pfaffian Schur processes \cite{borodin2005eynard} to study the model's multipoint distributions. We degenerate our model and formulas to the case of exponential weight half-space LPP (previous works \cite{baik2001algebraic, baik2001asymptotics, baik2001symmetrized, sasamoto2004fluctuations} only considered geometric weights). We perform an asymptotic analysis of the Fredholm Pfaffians characterizing the distribution of passage times in this model, which leads to Theorems \ref{theo:LPPdiagointro} and \ref{theo:LPPawaydiagointro}. We also study the $k$-point distribution of passage times at a distance $\mathcal{O}(n^{2/3})$ away from the diagonal, and introduce a new two-parameter family of crossover distributions when the rate of the weights on the diagonal are simultaneously scaled close to their critical value. This generalizes the crossover distributions found in \cite{forrester1999correlations, sasamoto2004fluctuations}. The asymptotics of fluctuations away from the diagonal were already studied in  \cite{sasamoto2004fluctuations} for the PNG model on a half-space (equivalently the geometric weight half-space LPP). The asymptotic analysis there was non-rigorous, at the level of studying the behavior of integrals around their critical points without controlling the decay of tails of integrals.

The proof of the first part of Theorem \ref{theo:LPPdiagointro} (in Section \ref{sec:GSEasymptotics}) required a new idea which we believe is the most novel technical contribution of this paper: When $\alpha>1/2$, $H(n,n)$ is the maximum of a simple Pfaffian point process which converges as $n\to \infty$ to a point process where all points have multiplicity $2$. The limit of the correlation kernel does not exist\footnote{For a non-simple point process, the correlation functions -- Radon-Nikodym derivatives of the factorial moment measures -- generically do not exist (see Definitions in Section \ref{sec:PPP}).} as a function (it does have a formal limit involving the derivative of the Dirac delta function). Nevertheless, a careful reordering and non-trivial cancellation of terms in the Fredholm Pfaffian expansions allows us to study the law of the maximum of this limiting point process, and ultimately find that it corresponds to the GSE Tracy-Widom distribution. We actually provide a general scheme for how this works for random matrix type kernels in which simple Pfaffian point processes limit to doubled ones (see Section \ref{formal}).

Theorem 5.3 of \cite{sasamoto2004fluctuations} derives a formula for certain crossover distributions in half-space geometric LPP. This crossover distribution, introduced in \cite{forrester1999correlations},  depends on a parameter $\tau$ (which is denoted $\eta$ in our Theorem \ref{theo:SU}) and it corresponds to a natural transition ensemble between GSE when $\tau\to 0$ and GUE when $\tau\to+\infty$. Although the collision of eigenvalues should occur in the GSE limit of this crossover, this point was left unaddressed in previous literature and the convergence of the crossover kernel to the GSE kernel was not proved in \cite{forrester1999correlations} (see our Proposition \ref{prop:limitcrossoverSU}).
The formulas for limiting kernels, given by \cite[(4.41), (4.44), (4.47)]{forrester1999correlations} or \cite[(5.37)-(5.40)]{sasamoto2004fluctuations}, make sense so long as $\tau>0$. Notice a difference of integration domain between \cite[(4.44)]{forrester1999correlations} and \cite[(5.39)]{sasamoto2004fluctuations}. Our Theorem \ref{theo:SU} is the analogue of \cite[Theorem 5.3]{sasamoto2004fluctuations} for half-space exponential LPP and we recover the exact same kernel. When $\tau=0$ the formula for  $\mathcal{I}_4$ in \cite[(5.39)]{sasamoto2004fluctuations} involves divergent integrals\footnote{The saddle-point analysis in the proof of Theorem 5.3 in \cite{sasamoto2004fluctuations} is valid only when $\tau>0$. Indeed, \cite[(5.44)]{sasamoto2004fluctuations} requires $\tau_1+\tau_2>\eta_1 + \eta_2$ and one needs that $\eta_1, \eta_2>0$ as in the proof of  \cite[Theorem 4.2]{sasamoto2004fluctuations}.}, though if one uses the formal identity $\int_{\R}\Ai(x+r)\Ai(y+r)dr = \delta(x-y)$ then it is possible to rewrite the kernel in terms of an expression involving the derivative of the Dirac delta function, and the expression formally matches with the kernel $K^{\infty}$ introduced in Section \ref{formal}. From that (formal) kernel it is non-trivial to match the Fredholm Pfaffian with a known formula for the GSE distribution, this matching does not seem to have been made previously in the literature and we explain it in Section \ref{formal}.

In random matrix theory, a similar phenomenon of collisions of eigenvalues occurs in the limit from the discrete symplectic ensemble to the GSE \cite{borodin2006averages}. However, \cite{borodin2006averages} used averages of characteristic polynomials to characterize the point process, and computed the correlation kernels from the characteristic polynomials only after the limit, so that collisions of eigenvalues were not an issue. A random matrix ensemble which crosses over between GSE and GUE was studied in \cite{forrester1999correlations}. 

\subsection{Other models related to KPZ growth in a half-space}
A positive temperature analogue of LPP -- directed random polymers -- has also been studied. The log-gamma polymer in a half-space (which converges to half-space LPP with exponential weights in the zero temperature limit) is considered in \cite{o2014geometric} where an exact formula is derived for the distribution of the partition function in terms of Whittaker functions. Some of these formulas are not yet proved and presently there has not been a successful asymptotic analysis preformed of them. Within physics, \cite{gueudre2012directed} and \cite{borodin2015directed} studies the continuum directed random polymer (equivalently the KPZ equation) in a half-space with a repulsive and refecting (respectively) barrier and derives GSE Tracy-Widom limiting statistics. Both works are mathematically non-rigorous due to the ill-posed moment problem and certain unproved forms of the conjectured Bethe ansatz completeness of the half-space delta Bose gas. On the side of particle systems, half-space ASEP has first been  studied in \cite{tracy2013asymmetric}, but the resulting formulas were not amenable to asymptotic analysis. More recently, \cite{barraquand2017stochastic} derives exact Pfaffian formulas for the half-space stochastic six-vertex model, and proves GOE asymptotics for the current at the origin in half-space ASEP, for a certain boundary condition. The exact formulas involve the correlation kernel of the Pfaffian Schur process and the  asymptotic analysis in \cite{barraquand2017stochastic} uses some of the ideas developed  here. 

\subsection{Main results}
We now state the limit theorems that constitute the main results of this paper. They involve the GOE, GSE and GUE Tracy-Widom distributions, respectively  characterized by the distribution functions  $F_{\rm GOE}$, $F_{\rm GSE}$ and  $F_{\rm GUE}$ given  in Section \ref{sec:defdistributions} and the standard Gaussian distribution function denoted by $G(x)$. 
\begin{theorem} The last passage time on the diagonal $H(n,n)$ satisfies the following limit theorems, depending on the rate $\alpha$ of the weights on the diagonal.
	\begin{enumerate}
		\item For $\alpha >1/2$,
		$$ \lim_{n\to\infty} \PP\left( \frac{H(n,n) -4n}{2^{4/3}n^{1/3}} < x \right) = F_{\rm GSE}\left( x\right).$$
		\item For $\alpha =1/2$,
		$$ \lim_{n\to\infty} \PP\left( \frac{H(n,n) -4n}{2^{4/3}n^{1/3}} < x \right) = F_{\rm GOE}\left( x\right).$$
		\item For $\alpha <1/2$,
		$$ \lim_{n\to\infty} \PP\left( \frac{H(n, n) -\frac{n}{\alpha(1-\alpha)}}{\sigma n^{1/2}} < x \right)  = G(x),$$
		where 
		$\sigma= \frac{(1-2\alpha)^{1/2}}{\alpha (1-\alpha)}.$
	\end{enumerate}
	\label{theo:LPPdiagointro}
\end{theorem}
This extends the results of \cite{baik2001algebraic, baik2001asymptotics} on the model with geometric weights. The first part of Theorem \ref{theo:LPPdiagointro} is proved in Section \ref{sec:GSEasymptotics}, while the second and third parts are proved in Section \ref{sec:asymptoticsvarious}.   Section \ref{sec:mainasymptoticresults} includes an explanation for why the transition occurs at $\alpha=1/2$, using a property of the Pfaffian Schur measure (Prop. \ref{prop:diagorow}).

Far away from the diagonal, the limit theorem satisfied by $H(n,m)$ coincides with that of the (unsymmetrized) full-quadrant model.
\begin{theorem}
	For any $\kappa \in (0,1)$ and $\alpha >\frac{\sqrt{\kappa}}{1+\sqrt{\kappa}}$, we have that
	$$ \lim_{n\to\infty} \PP\left( \frac{H(n,\kappa n) -(1+\sqrt{\kappa})^2 n}{\sigma n^{1/3}} < x \right) = F_{\rm GUE}(x),$$
	where
	$ \sigma = \frac{(1+ \sqrt{\kappa})^{4/3}}{\sqrt{\kappa}^{1/3}}.$
	\label{theo:LPPawaydiagointro}
\end{theorem}
This extends the results of \cite{sasamoto2004fluctuations} on the geometric model to the exponential case and to allow any boundary parameter. Theorem \ref{theo:LPPawaydiagointro} is proved in Section \ref{sec:GUEasymptotics}. 
\begin{remark}
	\label{rem:BBP}
	Our asymptotic analyses could be extended to show that the fluctuations of $H(n, \kappa n)$ for $\kappa\in(0,1)$ follow the BBP transition \cite{baik2005phase} when $\alpha= \sqrt{\kappa}/(1+\sqrt{\kappa}) + \mathcal{O}(n^{-1/3})$, and in particular are distributed according to $\big(F_{\rm GOE}\big)^2$ when $\alpha= \sqrt{\kappa}/(1+\sqrt{\kappa})$. The BBP transition would also occur if one varies the rate  of exponential weights in the  first rows of the lattice. Regarding $H(n,n)$, it is not clear if the higher order phase transitions would coincide with  the spiked GSE \cite{bloemendal2013limits, wang2009largest}.
\end{remark}

Our results rely on an asymptotic analysis of the following formula for the joint distribution of passage times. It involves the  Fredholm Pfaffian (see Section \ref{sec:defdistributions} for background on Fredholm Pfaffians)  of a matrix-valued  correlation kernel $K^{\rm exp}$ defined in Section \ref{sec:kernelexp} where the next proposition is proved. 
\begin{proposition}
	For any $h_1, \dots, h_k>0$ and integers $0<n_1 < n_2 < \dots <n_k$ and $m_1> m_2 > \dots > m_k$ such that $n_i>m_i$ for all $i$, we have that
	$$ \mathbb{P}\left(  \bigcap_{i=1}^k \big\lbrace  H(n_i, m_i) < h_i \big\rbrace \right)= \Pf\big(J-K^{\rm exp}\big)_{\mathbb{L}^2\big( \mathbb{D}_k(h_1, \dots, h_k) \big)},$$
	where the R.H.S  is a Fredholm Pfaffian (Definition \ref{def:FredholmPfaffian}) on the domain  
	$$ \mathbb{D}_k(h_1, \dots, h_k) = \lbrace (i,x)\in \lbrace 1, \dots, k\rbrace \times\R: x \geqslant h_i\rbrace.$$
	\label{prop:kernelexponential}
\end{proposition}
In the one point case $k=1$, and for $n_1=m_1$, one could alternatively characterize the probability distribution of $H(n,n)$ by taking the exponential limit of the formulas in \cite{baik2001algebraic, baik2001asymptotics} using for instance the results from \cite{baik2002painleve}. Note that we need only the case $k=1$ in order to prove Theorems \ref{theo:LPPdiagointro} and \ref{theo:LPPawaydiagointro}, but the multipoint case is necessary for Theorems \ref{theo:crossfluctuations} and \ref{theo:SU} stated below and proved in \cite{baik2017facilitated}.

We can generalize the results of Theorems \ref{theo:LPPdiagointro} and \ref{theo:LPPawaydiagointro}   by allowing the parameters $\kappa$ and $\alpha$ to vary in regions of size $n^{-1/3}$  around the critical values $\alpha=1/2, \kappa=1 $. In the limit, we obtain a new two-parametric family of probability distributions. 

More generally, we can compute the finite dimensional marginals of Airy-like processes in various ranges of $\alpha$ and $\kappa$. We state the results below and refer to \cite{baik2017facilitated} for detailed proofs. KPZ scaling predictions (together with Theorem \ref{theo:LPPawaydiagointro}) suggests to define
$$H_n(\eta)= \frac{H\big(n + n^{2/3}\xi\eta , n -n^{2/3}\xi\eta\big)-4n + n^{1/3}\xi^2\eta^2}{\sigma n^{1/3}},$$
where $\eta\geqslant 0$,  $\sigma = 2^{4/3}$ and $\xi=2^{2/3}$. 
Let us scale $\alpha$ as
$$ \alpha = \frac{1+2\sigma^{-1}\varpi n^{-1/3}}{2} $$
where  $\varpi\in\R$ is a free parameter. The limiting joint distribution of multiple points  is characterized by a new crossover kernel $K^{\rm  cross}$ that we introduce in Section \ref{sec:defcrossovers}.
\begin{theorem} For $0\leqslant \eta_1 < \dots < \eta_k$,  $\varpi\in\R$, setting
	$\alpha = \frac{1+2^{2/3}\varpi n^{-1/3}}{2}$, we have  that
	$$ \lim_{n\to\infty} \PP\left( \bigcap_{i=1}^k  \big\lbrace H_n(\eta_i) < x_i \big\rbrace  \right)  = \Pf\big( J- K^{\rm  cross}\big)_{\mathbb{L}^2(\mathbb{D}_k(x_1, \dots, x_k))}.$$
	\label{theo:crossfluctuations}
\end{theorem}
The proof of this result is very similar with the proof of Theorem \ref{theo:LPPawaydiagointro}. We provide the details of the computations in \cite{baik2017facilitated}. 

We note that the published version of \cite{baik2017facilitated} contains several mistakes in the definition and derivation of the correlation kernel $K^{\rm  cross}$ which filtered into the definitions provided herein in Section \ref{sec:defcrossovers}. These mistakes, and their remedy, were pointed out to us by Zongrui Yang. The arXiv version of \cite{baik2017facilitated} now contains a corrected derivation and definition for $K^{\rm  cross}$, and the current version of this paper has been likewise updated. We graciously acknowledge Yang's help in this revision.

In the one-point case ($k=1$), this yields a new probability distribution $\mathcal{L}_{\rm cross}$, depending on two parameters $\varpi, \eta$, (see Definition \ref{def:Fcross}), which realizes a two-dimensional crossover between GSE, GUE and GOE distributions (see Figure \ref{fig:phasediagram} and details in Section \ref{sec:defcrossovers}).  When $\eta=0$, we recover the probability distribution $F(x; w)$ from \cite[Definition 4]{baik2001asymptotics}. When $\varpi=0$, the correlation kernel $K^{\rm cross}$ degenerates to the orthogonal-unitary transition kernel obtained in \cite{forrester1999correlations} and $F_{\rm cross}(x; 0,\eta )$ realizes a crossover between $F_{\rm GOE}$ for $\eta=0$ and $F_{\rm GUE}$ for $\eta \to+\infty$. An analogous result was obtained in \cite[Theorem 4.2]{sasamoto2004fluctuations} for the half-space PNG model.

\begin{figure}
	
	\begin{tikzpicture}[scale=0.75]
	
	
	\fill[fill=gray!15] (0,0) -- plot [smooth, domain=0:4.5]  (\x, {4/(sqrt(4*\x+1)+1)}) --  (4.5, 0) -- cycle;

	\draw[ gray, >=stealth'] (0,-0.2) -- (0,5);
	\draw[ gray, >=stealth'] (-0.2,0) -- (5, 0);
	\fill[] (0,2) circle(0.07);
	\draw[thick] plot [smooth, domain=0:4.5]  (\x, {4/(sqrt(4*\x+1)+1)});
	
	\draw (-1, 0) node {{\footnotesize $\alpha=0$}};
	\draw (-1, 2) node {{\footnotesize $\alpha=1/2$}};
	\draw (-1, 5) node {{\footnotesize $\alpha=+\infty$}};
	
	\draw (0, -0.5) node {{\footnotesize $\frac n m=1$}};
	\draw (5, -0.5) node {{\footnotesize $\frac n m=+\infty$}};
	
	\draw (0, 3.5) node {$\mathrm{GSE}$};
	\draw (2.5, 2.5) node {$\mathrm{GUE}$};
	\draw (0.5, 0.5) node {$\mathrm{Gaussian}$};
	\draw (2.5, 1) node {$\mathrm{GOE}^2$};
	
	\draw[ dotted] (-0.2, 1.5) -- (0.5, 1.5) -- (0.5, 2.5) -- (-0.2, 2.5) -- cycle;
	\draw[ dotted] (0.5, 1.5) -- (8,-3);
	\draw[ dotted] (0.5, 2.5) -- (8,7);
	\draw[dotted] (8, -3) -- (16, -3) -- (16, 7) -- (8, 7) -- cycle;

	\draw[gray] (6, 4.3) node {{\footnotesize $ \alpha=\frac{1+2^{2/3}\varpi k^{-1/3}}{2} $}};
	\draw[gray] (6, 3.6) node {{\footnotesize $n = k + 2^{2/3}k^{2/3}\eta $}};
	\draw[gray] (6, 3) node {{\footnotesize $m = k - 2^{2/3}k^{2/3}\eta $}};
	
	\draw[->, gray, >=stealth'] (9,-2) -- (9,6);
	\draw[->, gray, >=stealth'] (9,2) -- (14, 2);
	\draw[thick, ->, >=stealth'] plot [domain=9:13] (\x, 11-\x);
	\draw (14, 1.5) node {$ \eta $};
	\draw (8.5, 5.5) node {$ \varpi $};
	
	\draw (9, 6.5) node {$\mathrm{GSE}$};
	\draw (9, 2) node {$\mathrm{GOE}$};
	\draw (9.5, -2.5) node {$\mathrm{Gaussian}$};
	\draw (15, 2) node {$\mathrm{GUE}$};
	\draw (13.7, -2.5) node {$\mathrm{GOE}^2$};
	\draw (11.5, 3) node {$\mathcal{L}_{\rm cross}(\cdot; \varpi, \eta)$};
	
	\draw[gray, ->, >=stealth'] (13, 5.2) -- (13.5, 5.7);
	\draw (14, 6) node {$\mathrm{GUE}$};
	\draw[gray, ->, >=stealth'] (12.9, 0) -- (13.5,-0.4) ;
	\draw (14.2, -0.7) node {$\mathrm{GUE}$};
	\end{tikzpicture}
	
	\caption{Phase diagram of the fluctuations of $H(n,m)$ as $n\to\infty$ when $\alpha$ and the ratio $n/m$ varies. The gray area corresponds to a region of the parameter space where the fluctuations are on the scale $n^{1/2}$ and Gaussian. The bounding curve (where fluctuations are expected to be Tracy-Widom  $\mathrm{GOE}^2$ cf Remark \ref{rem:BBP}) asymptotes to zero as $n/m$ goes to $+\infty$. The crossover distribution $\mathcal{F}^{\rm cross}(\cdot; \varpi, \eta)$ is defined in Definition \ref{def:Fcross} and describes the fluctuations in the vicinity of $n/m=1$ and $\alpha=1/2$.}
	\label{fig:phasediagram}
\end{figure}
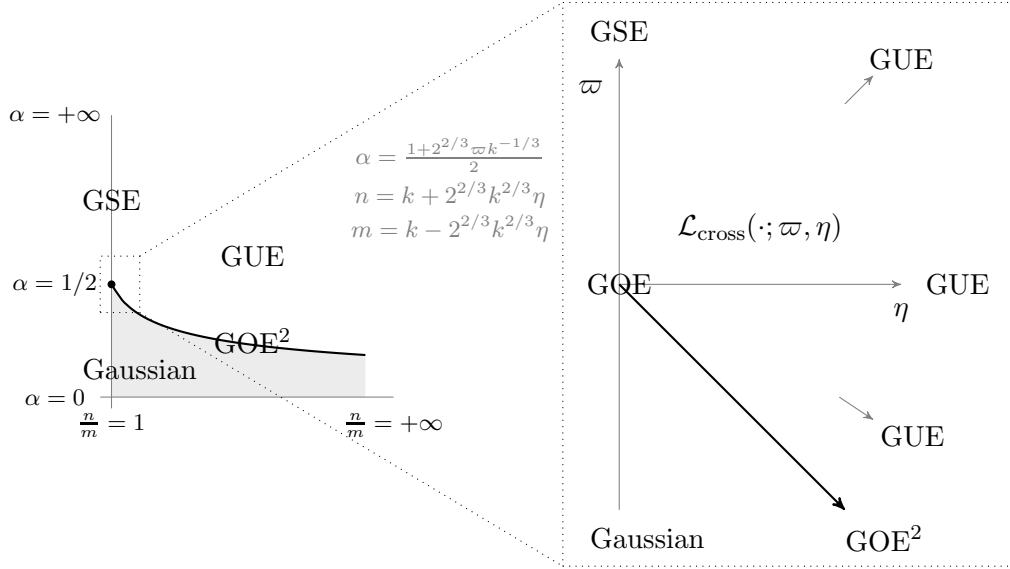

In the case when $\alpha>1/2$ is fixed, the joint distribution of passage-times is governed by the so-called symplectic-unitary transition \cite{forrester1999correlations}.
\begin{theorem}
	For $\alpha>1/2$ and $0< \eta_1< \dots <\eta_k$, 
	we have that
	$$ \lim_{n\to\infty} \PP\big(H_n(\eta_1) < x_1, \dots , H_n(\eta_k) < x_k \big)  = \Pf\big( J- K^{\rm  SU}\big)_{\mathbb{L}^2(\mathbb{D}_k(x_1, \dots, x_k))},$$
	where the kernel $K^{\rm  SU}$ is defined in Section \ref{sec:defcrossovers}. 
	\label{theo:SU}
\end{theorem}
Theorem \ref{theo:SU} can be seen as a $ \varpi\to+\infty$ degeneration of Theorem \ref{theo:crossfluctuations} (see Section \ref{sec:defcrossovers}). The proof is very similar with that  of Theorem \ref{theo:crossfluctuations}. We provide the details of the computations in \cite{baik2017facilitated}. 

An analogous result was found for the half-space PNG model without nucleations at the origin  \cite{sasamoto2004fluctuations} although the statement of \cite[Theorem 5.3]{sasamoto2004fluctuations} makes rigorous sense only when $\tau>0$. The correlation kernel corresponding to the symplectic-unitary transition was studied first in  \cite{forrester1999correlations} \footnote{The integration domain in \cite[(4.44)]{forrester1999correlations} should be $\R_{<0}$ instead of $\R_{>0}$. The correct formula was found in \cite[(5.39)]{sasamoto2004fluctuations}.}.  This random matrix ensemble is the point process corresponding to the eigenvalues of a Hermitian complex matrix $X_{\eta}$ for $\eta\in (0,+\infty)$ with density  proportional to
$$ \exp\left( \tfrac{-\mathrm{Tr}\big( (X_{\eta}-e^{-\eta}X_{0})^2\big)}{ 1-e^{-2\eta} }\right), $$
where $X_{0}$ is a GSE matrix. Various limits of the symplectic-unitary and orthogonal-unitary transition kernels are considered in \cite[Section 5]{forrester1999correlations}. In particular \cite[Section 5.6]{forrester1999correlations} considers the limit as $\eta\to0$. While the GOE distribution is recovered in the orthogonal-unitary case, the explanations are missing in the symplectic-unitary case (the scaling argument at  the end of  \cite[Section 5.6]{forrester1999correlations} is not the reason why one does not recover $K^{\rm GSE}$). We provide a rigorous proof in Section \ref{sec:rigorousSUtoGSE}. 
\begin{remark}
	One can also study the limiting $n$-point distribution of
	$$ \eta\mapsto \frac{H(n + n^{2/3}\xi\eta , \kappa n -n^{2/3}\xi\eta)-(1+\sqrt{\kappa})^2 n -n^{1/3}x^2}{2^{4/3}n^{1/3}},$$
	for a fixed $\kappa\in(0,1)$ and several values of $x$. In the $ n\to\infty$  limit, one would obtain the extended Airy kernel for $K^{\rm exp}_{12}$ and $0$ for $K^{\rm exp}_{11}$  and $K^{\rm exp}_{22}$ but we do not pursue that direction.
\end{remark}

\subsection*{Outline of the paper}

In Section \ref{sec:defdistributions}, we provide convenient Fredholm Pfaffian formulas  for the Tracy-Widom distributions and their generalizations. 
In Section \ref{sec:PSP}, we define Pfaffian Schur processes and construct dynamics preserving them, thus making a connection to half-space LPP (Proposition \ref{prop:SchurLPPcorrelations}). In Section \ref{sec:kpointdistribution}, we apply a general result of Borodin and Rains giving the correlation structure of Pfaffian Schur processes, in order to express the $k$-point distribution along space-like paths in half-space LPP with geometric (Proposition \ref{prop:kernelgeom}) and exponential  (Proposition \ref{prop:kernelexponential}) weights. In Section \ref{sec:GSEasymptotics}, we discussed the issues related to the multiplicity of the limiting point process, and prove the first part of Theorem \ref{theo:LPPdiagointro}. In Section \ref{sec:asymptoticsvarious}, we perform all other asymptotic analysis: we prove limit theorems towards the GUE Tracy-Widom distribution (Theorem \ref{theo:LPPawaydiagointro}),  GOE Tracy-Widom distribution (second part of Theorem \ref{theo:LPPdiagointro}), and Gaussian distribution (third part of \ref{theo:LPPdiagointro}).

\subsection*{Acknowledgements}

The authors would like to thank Alexei Borodin and Eric Rains for sharing their insights about Pfaffian Schur processes. We are especially grateful to Alexei Borodin for discussions related to the dynamics preserving Pfaffian Schur processes and for drawing our attention to the reference \cite{borodin2006averages}. We also thank Tomohiro Sasamoto and Takashi Imamura for discussions regarding their paper \cite{sasamoto2004fluctuations}. We graciously acknowledge Zongrui Yang's help in catching and fixing some mistakes in our formulas for the crossover kernel in \cite[Section 5]{baik2017facilitated}, which filtered into Section \ref{sec:defcrossovers} of this paper.

Part of this work was done during the stay of J. Baik, G. Barraquand and I. Corwin  at the KITP and supported by  NSF PHY11-25915. J. Baik was supported in part by NSF grant DMS-1361782. G. Barraquand was supported in part by the LPMA UMR CNRS 7599, Universit\'e Paris-Diderot--Paris 7  and the Packard Foundation through I. Corwin's Packard Fellowship. I. Corwin was  supported in part by the NSF through DMS-1208998,  a Clay Research Fellowship, the  Poincaré Chair, and  a Packard Fellowship for Science and Engineering.

\section{Fredholm Pfaffian formulas}
\label{sec:defdistributions}
Let us introduce a convenient notation that we use throughout the paper to specify integration contours in the complex plane.
\begin{definition}
	Let $\mathcal{C}_{a}^{\varphi}$ be the union of two semi-infinite rays departing $a\in\C$ with angles $\varphi$ and $-\varphi$. We  assume that the contour is oriented from $a+\infty e^{-i\varphi}$ to $a+\infty e^{+i\varphi}$.
	\label{def:basicrays}
\end{definition}
\subsection{GUE Tracy-Widom distribution}
For a kernel $K:\mathbb{X}\times \mathbb{X} \to \R,$
we define its Fredholm determinant $ \det(I+K)_{\mathbb{L}^2(\mathbb{X},\mu)}$ as given by the series expansion
$$ \det(I+K)_{\mathbb{L}^2(\mathbb{X},\mu)} = 1+\sum_{k=1}^{\infty}\frac{1}{k!} \int_{\mathbb{X}} \dots \int_{\mathbb{X}} \det\big( K(x_i, x_j)\big)_{i,j=1}^k\ \mathrm{d}\mu^{\otimes k}(x_1 \dots x_k), $$
whenever it converges. We will generally omit the measure $\mu$ in the notations and write simply  $\mathbb{L}^2(\mathbb{X})$ when the uniform or the Lebesgue  measure is considered. 

\begin{definition}
	The GUE Tracy-Widom distribution, denoted $\mathcal{L}_{\rm GUE}$ is a probability distribution on $\R$ such that if $X\sim \mathcal{L}_{\rm GUE}$,
	$$ \PP(X\leqslant x) = F_{\rm GUE}(x) = \det(I-K_{\rm Ai})_{\mathbb{L}^2(x,+\infty )} $$
	where  $K_{\rm Ai}$ is the Airy kernel,
	\begin{equation}
	K_{\rm Ai} (u, v) = \int_{\mathcal{C}_{-1}^{2\pi/3}} \dd w \int_{\mathcal{C}_1^{\pi/3}} \dd z \frac{e^{z^3/3-zu}}{e^{w^3/3-wv}}\frac{1}{z-w}.
	\label{eq:defairykernel}
	\end{equation}
	Throughout the paper, all integrals over a contour of the complex plane will take a factor $1/(2\I\pi)$, and thus we use the notation $\dd z := \frac{1}{2\I\pi}\mathrm{d} z$. 
	\label{def:GUEdistribution}
\end{definition}

\subsection{Fredholm Pfaffian}
In order to define the GOE and GSE distribution in a form which is convenient for later purposes, we introduce the concept of Fredholm Pfaffian.  We refer to Section \ref{sec:PPP} explaining how Fredholm Pfaffians naturally arise in the study of Pfaffian point processes.
\begin{definition}[\cite{rains2000correlation} Section 8]
	For a $2\times 2$-matrix valued  skew-symmetric kernel,
	$$ K(x,y) = \begin{pmatrix}
	K_{11}(x,y) & K_{12}(x,y)\\
	K_{21}(x, y) & K_{22}(x,y)
	\end{pmatrix},\ \ x,y\in \mathbb{X},$$
	(such kernel is called skew-symmetric if the $2k\times 2k$ matrix $\Big( K(x_i, x_j)\Big)_{i,j=1}^k$ is skew-symmetric) we define its Fredholm Pfaffian  by the series expansion
	\begin{equation}
	\Pf\big(J+K\big)_{\mathbb{L}^2(\mathbb{X},\mu)} = 1+\sum_{k=1}^{\infty} \frac{1}{k!}
	\int_{\mathbb{X}} \dots \int_{\mathbb{X}}  \Pf\Big( K(x_i, x_j)\Big)_{i,j=1}^k \mathrm{d}\mu^{\otimes k}(x_1 \dots x_k),
	\label{eq:defFredholmPfaffian}
	\end{equation}
	provided the series converges, and we recall that for an skew-symmetric  $2k\times 2k$ matrix $A$, its Pfaffian is defined by
	\begin{equation}
	\Pf(A) = \frac{1}{2^k k!} \sum_{\sigma\in\mathcal{S}_{2k}} \sign(\sigma) a_{\sigma(1)\sigma(2)}a_{\sigma(3)\sigma(4)} \dots a_{\sigma(2k-1)\sigma(2k)}.
	\label{def:pfaffian}
	\end{equation}
	The kernel $J$ is defined by
	$$ J(x,y)  = \mathds{1}_{x=y}
	\begin{pmatrix}
	0 & 1\\-1 & 0
	\end{pmatrix} .$$
	\label{def:FredholmPfaffian}
\end{definition}
\begin{remark} 
	We must consider matrix kernels as made up of $k^2$ blocks, each of which has size $2 \times 2$. Considering  $2^2$ blocks of size $k \times k$ instead would change the value of its Pfaffian by a factor $(-1)^{k(k-1)/2}$.
\end{remark}

In Sections \ref{sec:kpointdistribution}, \ref{sec:GSEasymptotics} and \ref{sec:asymptoticsvarious}, we will need to control the convergence of Fredholm Pfaffian series expansions. This can be done using Hadamard's bound. We recall that for a matrix
$M$ of size $k$, if the $(i,j)$ entry of $M$ is bounded by $a_i b_j$ for all $i,j\in\lbrace 1, \dots, k\rbrace$ then
\begin{equation*}
\big\vert \det[M]\big\vert \leqslant k^{k/2}\prod_{i=1}^k a_ib_i.
\label{eq:hadamard}
\end{equation*}
Using this inequality and the fact that $ \Pf[A] = \sqrt{\det[A]}$ for a skew-symmetric matrix $A$, we have 
\begin{lemma}
	Let $K(x, y)$ a $2\times 2$ matrix valued skew symmetric kernel. Assume that there exist constants $C>0$ and constants $a>b\geqslant 0$ such that
	\begin{equation*}
	\vert K_{11}(x, y) \vert <C e^{-ax-ay}, \ \
	\vert K_{12}(x, y)\vert   <C e^{-ax+by}, \ \
	\vert K_{22}(x, y)\vert  <C e^{bx+by}.
	\end{equation*}
	Then, for all $k\in \Z_{>0}$,
	$$ \Big\vert \Pf\big[ K(x_i, x_j)\big]_{i,j=1}^k \Big\vert < (2k)^{k/2}C^k\prod_{i=1}^{k}e^{-(a-b) x_i}.$$
	\label{lem:hadamard}
\end{lemma}
\subsection{GOE Tracy-Widom distribution}
The GOE Tracy-Widom distribution, denoted $\mathcal{L}_{\rm GOE}$, is a continuous probability distribution on $\R$. The following is a convenient formula for its cumulative distribution function. 
\begin{lemma}
	For $X\sim \mathcal{L}_{\rm GOE}$, 
	$ F_{\rm GOE}(x):=\PP(X\leqslant x) =  \Pf\big( J- K^{\rm GOE}\big)_{\mathbb{L}^2(x, \infty)},$
	where $K^{\rm GOE}$ is the $2\times 2$ matrix valued kernel defined by
	\begin{align*}
	K_{11}^{\rm GOE}(x,y) &=  \int_{\mathcal{C}_{1}^{\pi/3}}\dd z\int_{\mathcal{C}_{1}^{\pi/3}}\dd w  \frac{z-w}{z+w} e^{z^3/3 + w^3/3 - xz -yw},\\
	K_{12}^{\rm GOE}(x,y) &= -K_{21}^{\rm GOE}(x,y) = \int_{\mathcal{C}_{1}^{\pi/3}}\dd z\int_{\mathcal{C}_{-1/2}^{\pi/3}}\dd w  \frac{w-z}{2w(z+w)} e^{z^3/3 + w^3/3 - xz -yw} ,\\
	K_{22}^{\rm GOE}(x,y) &=  \int_{\mathcal{C}_{1}^{\pi/3}}\dd z\int_{\mathcal{C}_{1}^{\pi/3}}\dd w  \frac{z-w}{4zw(z+w)} e^{z^3/3 + w^3/3 - xz -yw} \\
	& +\int_{\mathcal{C}_{1}^{\pi/3}} e^{z^3/3-zx}\frac{\dd z}{4z} -  \int_{\mathcal{C}_{1}^{\pi/3}} e^{z^3/3-zy}\frac{\dd z}{4z} -\frac{\sgn{(x-y)}}{4}.
	\end{align*}
	Throughout the paper, we adopt the convention that
	$$\sgn(x-y) = \mathds{1}_{x>y} - \mathds{1}_{x<y}. $$
	\label{def:GOEdistribution}
\end{lemma}
\begin{proof}
	According to \cite{tracy2005matrix} (see also equivalent formulas \cite[(2.9) and (6.17)]{ferrari2004polynuclear}),  the GOE distribution function can be defined by $F_{\rm GOE} = \Pf\big(J-K^{1}\big)$, where $K^{1}$ is defined by the matrix kernel
	\begin{align*}
	K^{1}_{11}(x,y) &= \int_0^{\infty} \mathrm{d}\lambda \Ai(x+\lambda)\Ai'(y+\lambda) - \int_0^{\infty} \mathrm{d}\lambda \Ai(y+\lambda)\Ai'(x+\lambda), \\
	K^{1}_{12}(x,y) &=- K^{1}_{21}(y,x) =  \int_{0}^{\infty} \Ai(x+\lambda)\Ai(y+\lambda) +\frac 1 2 \Ai(x) \int_0^{\infty} \mathrm{d}\lambda \Ai(y-\lambda) \\
	K^{1}_{22}(x,y) &=\frac{1}{4} \int_0^{\infty} \mathrm{d}\lambda \int_{\lambda}^{\infty}\mathrm{d}\mu \Ai(y+\mu)\Ai(x+\lambda) - \frac{1}{4} \int_0^{\infty} \mathrm{d}\lambda \int_{\lambda}^{\infty}\mathrm{d}\mu \Ai(x+\mu)\Ai(y+\lambda) \\
	&- \frac{1}{4}\int_{0}^{+\infty} \Ai(x+\lambda)\mathrm{d}\lambda + \frac{1}{4}\int_{0}^{+\infty} \Ai(y+\lambda)\mathrm{d}\lambda - \frac{\sgn(x-y)}{4}.
	\end{align*}
	Using the contour integral representation of the Airy function
	\begin{equation}
	\Ai(x) =\int_{\mathcal{C}_{a}^{\varphi}} e^{z^3/3-zx}\dd z, \text{ for any }\varphi\in(\pi/6, \pi/2]\text{ and }a\in \R_{>0},
	\label{eq:defairy}
	\end{equation}
	with integration on $\mathcal{C}_{1}^{\pi/3}$ for $\Ai(x+\lambda)$ and  $\Ai'(x+\lambda)$,  we readily see that
	\begin{align}
	K^{1}_{11}(x,y) &= \int_0^{\infty} \mathrm{d}\lambda  \int_{\mathcal{C}_{1}^{\pi/3}}\dd  z\int_{\mathcal{C}_{1}^{\pi/3}}\dd w (z-w) e^{z^3/3 + w^3/3 - xz -yw - \lambda z - \lambda w}\nonumber\\
	&=   \int_{\mathcal{C}_{1}^{\pi/3}}\dd z\int_{\mathcal{C}_{1}^{\pi/3}}\dd w (z-w) e^{z^3/3 + w^3/3 - xz -yw }\int_0^{\infty} \mathrm{d}\lambda e^{-\lambda z - \lambda w} \label{eq:explainGOE11}\\
	&= K^{\rm GOE}_{11}(x,y) .\nonumber
	\end{align}
	The exchange of integrations in \eqref{eq:explainGOE11} is justified because $z+w$ has positive real part along the contours.
	Turning to $K_{12}^{\rm GOE}$, we first deform the contour $\mathcal{C}_{-1/2}^{\pi/3}$ for the variable $w$ to $\mathcal{C}_{1}^{\pi/3}$. When doing this, we encounter a pole at zero. Thus, taking into account the residue at zero, we find that
	$$ K_{12}^{\rm GOE}(x,y) = \frac 1 2 \Ai(x) + \int_{\mathcal{C}_{1}^{\pi/3}}\dd z\int_{\mathcal{C}_{1}^{\pi/3}}\dd w  \frac{w-z}{2w(z+w)} e^{z^3/3 + w^3/3 - xz -yw} $$
	Using again \eqref{eq:defairy} with the contour $\mathcal{C}_{1}^{\pi/3}$, we find that
	\begin{align*}
	K_{12}^{\rm GOE}(x,y)&= \frac 1 2 \Ai(x) + \int_{0}^{\infty} \Ai(x+\lambda)\Ai(y+\lambda)\mathrm{d}\lambda - \frac 1 2 \Ai(x) \int_0^{\infty} \mathrm{d}\lambda \Ai(y+\lambda)  \\
	&=  K_{12}^{1}(x,y),
	\end{align*}
	where in the last equality we have used that
	$ \int_{-\infty}^{+\infty}\Ai(\lambda)\mathrm{d}\lambda =1.$
	For $K^{1}_{22}$ we use similarly  \eqref{eq:defairy} with integration on $\mathcal{C}_{1}^{\pi/3}$ for $\Ai(x+\lambda)$ and $\Ai(x+\mu)$ and  get that $K^{1}_{22}(x,y) = K^{\rm GOE}_{22}(x,y)$.
\end{proof}
\subsection{GSE Tracy-Widom distribution}
The GSE Tracy-Widom distribution, denoted $\mathcal{L}_{\rm GSE}$, is a continuous probability distribution on $\R$. 
\begin{lemma}
	For $X\sim \mathcal{L}^{\rm GSE}$, 
	$ F_{\rm GSE}\left( x\right) :=\PP(X\leqslant x)  = \Pf\big( J- K^{\rm GSE}\big)_{\mathbb{L}^2(x, \infty)} ,$
	where $K^{\rm GSE}$ is a $2\times 2$-matrix valued kernel
	defined by
	\begin{align*}
	K_{11}^{\rm GSE}(x,y) &=  \int_{\mathcal{C}_{1}^{\pi/3}}\dd z\int_{\mathcal{C}_{1}^{\pi/3}}\dd w \frac{z-w}{4zw(z+w)} e^{z^3/3 + w^3/3 - xz -yw} ,\\
	K_{12}^{\rm GSE}(x,y) &= -K_{21}^{\rm GSE}(x,y) =  \int_{\mathcal{C}_{1}^{\pi/3}}\dd z\int_{\mathcal{C}_{1}^{\pi/3}}\dd w \frac{z-w}{4z(z+w)} e^{z^3/3 + w^3/3 - xz -yw} ,\\
	K_{22}^{\rm GSE}(x,y) &= 
	\int_{\mathcal{C}_{1}^{\pi/3}}\dd z\int_{\mathcal{C}_{1}^{\pi/3}}\dd w \frac{z-w}{4(z+w)} e^{z^3/3 + w^3/3 - xz -yw} .
	\end{align*}
	\label{def:GSEdistribution}
\end{lemma}
\begin{proof}
	$F_{\rm GSE}(x)$ is defined as
	$F_{\rm GSE}(x) = \sqrt{\det(I - K^{4})_{\mathbb{L}^2(x, \infty)}} $  in  \cite[Section III]{tracy2005matrix},
	where $K^{4}$ is the $2\times 2$ matrix valued kernel defined by
	\begin{align*}
	K^{4}_{11}(x,y) &= K^{4}_{22}(y,x) =  \frac{1}{2} K_{\rm Ai}(x,y) - \frac 1 4 \Ai(x) \int_y^{\infty} \Ai(\lambda)\mathrm{d}\lambda,\\
	K^{4}_{12}(x,y) &= \frac{-1}{2}\partial_y  K_{\rm Ai}(x,y) - \frac 1 4 \Ai(x)\Ai(y),\\
	K^{4}_{21}(x,y) &=  \frac{-1}{2} \int_x^{\infty}K_{\rm Ai}(\lambda,y)\mathrm{d}\lambda + \frac 1 4 \int_{x}^{\infty}\Ai(\lambda)\mathrm{d}\lambda \int_{y}^{\infty}\Ai(\mu)\mathrm{d} \mu.
	\end{align*}
	In order to have a Pfaffian formula, we need a skew-symmetric kernel. We compute the kernel $K^{\rm GSE}:=JK^{4}$ using
	$$ \begin{pmatrix}
	0 & 1 \\
	-1 & 0
	\end{pmatrix} \begin{pmatrix}
	a & b \\ c & d
	\end{pmatrix} = \begin{pmatrix}
	c & d \\
	-a & -b
	\end{pmatrix},$$
	and find that the operator $K^{\rm GSE}$ is given by the matrix kernel
	\begin{align*}
	K_{11}^{\rm GSE}(x,y) &= K^{4}_{21}(x,y), \\
	K_{12}^{\rm GSE}(x,y) &= K^{4}_{22}(x,y), \\
	K_{21}^{\rm GSE}(x,y) &=  -K^{4}_{11}(x,y) = -K^{4}_{22}(y,x)  = -K^{\rm GSE}_{12}(y,x),\\
	K_{22}^{\rm GSE}(x,y)  &= -K^{4}_{12}(x,y).
	\end{align*}
	Using the contour integral representations for the Airy function \eqref{eq:defairy} with integrations on $\mathcal{C}_{1}^{\pi/3}$
	and the definition of the Airy kernel \eqref{eq:defairykernel}, we find that the entries
	$ K_{ij}^{\rm GSE}(x,y)$ match those given in the statement of Lemma \ref{def:GSEdistribution}.
	Now $K^{\rm GSE}$ is skew-symmetric. Using that for a skew-symmetric kernel $A$, $\Pf(J+A)^2 = \det(I-JA)$ as soon as Fredholm expansions are convergent (\cite[Lemma 8.1]{rains2000correlation} , see also \cite[Proposition B.4]{ortmann2015pfaffian} for a proof),
	we have that (using $J^2=I$)
	\begin{equation*}
	F_{\rm GSE}(x) = \sqrt{\det(I - K^{4})_{\mathbb{L}^2(x, \infty)}} = \Pf\big(J-K^{\rm GSE}\big)_{\mathbb{L}^2(x, \infty)}. \qedhere
	\end{equation*}
\end{proof}
\begin{remark}
	There exists an alternative scalar kernel which yields the GSE distribution function  \cite{gueudre2012directed}: 
	\begin{equation}
	F_{\rm GSE}(x) = \Pf\big( J -K^{\rm GSE})_{\mathbb{L}^2(x, \infty)} = \sqrt{\det\big( I-K^{\rm GLD} \big)_{\mathbb{L}^2(x, \infty)} },
	\label{eq:alterGSE}
	\end{equation}
	where
	$$ K^{\rm GLD}(x,y) = K_{\Ai}(x,y) -  \frac 1 2 \Ai(x) \int_0^{\infty}\dd z\ \Ai(y+z).$$
	The identity \eqref{eq:alterGSE} can be shown by factoring $K^{\rm GSE}$ and using the $\det(I+AB)=\det(I+BA)$ trick.
\end{remark}

\subsection{Crossover kernels}
\label{sec:defcrossovers}
The crossover kernel in Theorem \ref{theo:crossfluctuations} concerns the limiting fluctuations of multiple points, hence the kernel is indexed by elements of $\lbrace 1, \dots, k\rbrace \times \R$. We introduce a matrix kernel $K^{\rm cross}$ , depending on parameters $\varpi \in \R$ and $0\leqslant \eta_1 < \dots < \eta_k$,  which decomposes as 
$$ K^{\rm cross}(i, x;j, y) = I^{\rm cross}(i, x;j, y)+R^{\rm cross}(i, x;j, y).$$
We have  $ K^{\rm cross}_{21}(i, x;j, y)=-K^{\rm cross}_{12}(j, y; i, x)$, and  
\begin{align*}
I_{11}^{\rm cross}(i, x;j, y) &= \int_{\mathcal{C}_{1}^{\pi/3}} \dd z\int_{\mathcal{C}_{1}^{\pi/3}} \dd w
\frac{z+\eta_i -w-\eta_j}{z+w+ \eta_i+\eta_j}  \frac{z+\varpi+\eta_i}{z+\eta_i}\frac{w+\varpi+\eta_j}{w+\eta_j}  e^{z^3/3 + w^3/3 - x z -y w},\\
I_{12}^{\rm cross}(i, x;j, y) &= \int_{\mathcal{C}_{a_z}^{\pi/3}} \dd z\int_{\mathcal{C}_{a_w}^{\pi/3}} \dd w
\frac{z +\eta_i -w+\eta_j  }{2(z+\eta_i)(z+\eta_i+w-\eta_j)}
\frac{z+\varpi+\eta_i}{-w+\varpi+\eta_j}  e^{z^3/3 + w^3/3 - x z -y w} ,\\
I_{22}^{\rm cross}(i, x;j, y) &=  \int_{\mathcal{C}_{b_z}^{\pi/3}} \dd z\int_{\mathcal{C}_{b_w}^{\pi/3}} \dd w
\frac{z-\eta_i-w+\eta_j}{4(z-\eta_i+w-\eta_j)}
\frac{e^{z^3/3 + w^3/3 - x z-yw }}{(z-\varpi-\eta_i)(w-\varpi-\eta_j)}.
\end{align*} 
The contours in $I_{12}^{\rm cross}$ are chosen so  $a_z>-\eta_i$, $a_z+a_w>\eta_j-\eta_i$ and $a_w<\varpi+ \eta_j$.
The contours in $I_{22}^{\rm cross}$ are chosen so that  (1) if $\varpi\leqslant0$ then $b_z>\eta_i$ and $b_w>\eta_j$, and (2) if $\varpi>0$ then $b_z\in(\eta_i,\eta_i+\varpi)$ and $b_w\in(\eta_j,\eta_j+\varpi)$. 
We have $R_{11}^{\rm cross}(i, x;j, y)=0$, and $R_{12}^{\rm cross}(i, x;j, y)=0$ when $i\geqslant j$. When $i<j$,
$$ R_{12}^{\rm cross}(i, x;j, y) = \frac{-\exp\left(\frac{-(\eta_i-\eta_j)^4+ 6(x+y)(\eta_i-\eta_j)^2+3(x-y)^2}{12(\eta_i-\eta_j)}\right)}{\sqrt{4\pi(\eta_j-\eta_i)}}, $$
that we may also write
$$R_{12}^{\rm cross}(i, x;j, y)= -\int_{-\infty}^{+\infty} \mathrm{d}\lambda e^{-\lambda(\eta_i - \eta_j)} \Ai(x_i+\lambda)\Ai(x_j+\lambda).$$
The kernel  $R_{22}^{\rm cross}$ is antisymmetric, and when $ x -   \eta_i^2>y -  \eta_j^2$ we have
\begin{multline*}
 	R_{22}^{\rm cross}(i, x;j, y)=   
 	 \frac{\mathds{1}_{\varpi\leqslant0}}{4} \int_{\mathcal{C}_{a_z}^{\pi/3}}\dd z \frac{e^{ (z+\eta_j)^3/3 +(\varpi+\eta_i)^3/3 -y (z+\eta_j) -x(\varpi+\eta_i) }}{\varpi+z} \\
 -\frac{\mathds{1}_{\varpi\leqslant0}}{4} \int_{\mathcal{C}_{a_z}^{\pi/3}}\dd z \frac{e^{ (z+\eta_i)^3/3 +(\varpi+\eta_j)^3/3 -x (z+\eta_i) -y(\varpi+\eta_j)}}{\varpi+z} \\
 	-\frac{1}{2}\int_{\mathcal{C}_{b_z}^{\pi/3}}\dd z\frac{z e^{ (z+\eta_i)^3/3 +(-z+\eta_j)^3/3 -x (z+\eta_i) -y(-z+\eta_j) }}{(\varpi+z)(\varpi -z)} 
         -\frac{\mathds{1}_{\varpi=0}}{4} e^{\eta_i^3/3+\eta_j^3/3-\eta_ix-\eta_jy }           , 
 	\end{multline*}
 	where the contours are chosen so that $a_z<-\varpi$ and  (1) if $\varpi\neq0$ then $-\vert \varpi\vert <b_z<\vert \varpi \vert$, and (2) if $\varpi=0$ then $b_z>0$.  
\begin{definition}
	We define the probability distribution $\mathcal{L}_{\rm cross}$ by the distribution function
	\begin{equation*}
	F_{\rm cross}(x; \varpi, \eta ) = \Pf\Big( J- K^{\rm  cross}(1, \cdot ; 1, \cdot)\Big)_{\mathbb{L}^2(x, \infty)}.
	\end{equation*}
	\label{def:Fcross}
\end{definition}
This family of distribution functions realizes a two-dimensional crossover between GSE, GUE and GOE distributions, in the sense that
$F_{\rm cross}(x; 0,0 ) = F_{\rm GOE}(x) $ and
$$F_{\rm cross}(x; \varpi, \eta )\longrightarrow \begin{cases}
F_{\rm GSE}(x) &{\footnotesize \text{ when } \varpi\to+\infty \text{ and }\eta=0,}\\
\big(F_{\rm GOE}(x)\big)^2 &{\footnotesize \text{ when } \varpi\to -\infty, \eta\to+\infty \text{ with }\varpi/\eta=-1,}\\
F_{\rm GUE}(x) &{\footnotesize \text{ when } \eta\to+\infty \text{ and } \varpi+\eta\to +\infty,}\\
0 &{\footnotesize \text{ when } \varpi\to-\infty  \text{ and }\varpi+\eta\to -\infty.}\\
\end{cases} $$
Moreover, when $\varpi\to-\infty$ and $\varpi+\eta\to -\infty$, 
$$
F_{\rm cross}((\varpi+\eta)^2+\sqrt{2|\varpi+\eta|}y; \varpi, \eta )
\longrightarrow \int_{-\infty}^y \frac{1}{\sqrt{2\pi}} e^{-\frac12s^2} ds.
$$
In addition, for $a\in \R$,
$
F_{\rm cross}(x; \varpi, \eta )
\longrightarrow F_1(x;a)$ when 
$\varpi\to-\infty, \eta\to +\infty$ with $\varpi+\eta=a$,
where $F_1(x;a)$ is the distribution in \cite[Definition 1.3]{baik2005phase} that interpolates $F_{\rm GUE}(x)$
and $\big(F_{\rm GOE}(x)\big)^2$. When $\eta=0$, $F_{\rm cross}(x; \varpi,0 ) =F(x; \varpi)$ where the crossover distribution $F(x; w)$ is defined in \cite[Definition 4]{baik2001asymptotics}.

The crossover kernel arising in Theorem \ref{theo:SU} is a matrix kernel 
$K^{\rm SU}$, depending on parameters $\varpi \in \R$ and $0< \eta_1 < \dots < \eta_k$,  of the form
$$ K^{\rm SU}(i, x;j, y) = I^{\rm SU}(i, x;j, y)+R^{\rm SU}(i, x;j, y).$$
We have  $ K^{\rm SU}_{21}(i, x;j, y)=-K^{\rm SU}_{12}(j, y; i, x)$, and  
\begin{align*}
I^{\rm  SU}_{11}(i, x;j, y) &=  \int_{\mathcal{C}_{1}^{\pi/3}} \dd z\int_{\mathcal{C}_{1}^{\pi/3}} \dd w
\frac{(z+\eta_i-w-\eta_j)e^{z^3/3 + w^3/3 - x z -yw}}{4(z+\eta_i)(w+\eta_j)(z+w+ \eta_i+\eta_j)}   , \\
I^{\rm  SU}_{12}(i, x;j, y) &= \int_{\mathcal{C}_{a_z}^{\pi/3}} \dd z\int_{\mathcal{C}_{a_w}^{\pi/3}} \dd w
\frac{(z + \eta_i -w + \eta_j) e^{z^3/3 + w^3/3 - xz -yw}}{2(z+\eta_i)(z+w+\eta_i-\eta_j)}    ,\\
I^{\rm  SU}_{22}(i, x;j, y)&=  \int_{\mathcal{C}_{b_z}^{\pi/3}} \dd z\int_{\mathcal{C}_{b_w}^{\pi/3}} \dd w
\frac{z-\eta_i-w+\eta_j}{z-\eta_i+w-\eta_j}  e^{z^3/3 + w^3/3 - xz-yw},
\end{align*}
The contours in $I_{12}^{\rm SU}$ are chosen so $a_z>-\eta_i$,  $a_z+a_w>\eta_j-\eta_i$.  
The contours in $I_{22}^{\rm SU}$ are chosen so that $b_z>\eta_i$ and $b_w>\eta_j$.

We have $R_{11}^{\rm SU}(i, x;j, y)=0$, and $R_{12}^{\rm SU}(i, x;j, y)=0$ when $i\geqslant j$.
When $i<j$,
\begin{equation*} 
R_{12}^{\rm SU}(i, x;j, y)  = \frac{-e^\frac{-(\eta_i-\eta_j)^4+ 6(x+y)(\eta_i-\eta_j)^2+3(x-y)^2}{12(\eta_i-\eta_j)}}{\sqrt{4\pi(\eta_j-\eta_i)}}= R_{12}^{\rm cross}(1, x;j, y).
\end{equation*}
The kernel  $R_{22}^{\rm SU}$ is antisymmetric, and when $ x -   \eta_i^2>y -  \eta_j^2$ we have
\begin{equation*}
R_{22}^{\rm SU}(i, x;j, y)= 
-\frac{1}{2}\int_{\mathcal{C}_{0}^{\pi/3}}\dd z\ \  ze^{ (z+\eta_i)^3/3 +(-z+\eta_j)^3/3 -x (z+\eta_i) -y(-z+\eta_j)},
\end{equation*}
where the contours are chosen so that $a_z>-\varpi$ and $b_z$ is between $-\varpi$ and $\varpi$.

Modulo a conjugation of the kernel, $K^{\rm  SU}$ is the limit of $K^{\rm cross}$ when $\varpi$ goes to $+\infty$, i.e.
$$K^{\rm  SU} = \lim_{\varpi\to \infty} \begin{pmatrix}\frac{1}{4\varpi^2}K^{\rm  cross}_{11} & K^{\rm  cross}_{12}\\
K^{\rm  cross}_{21} &4 \varpi^2 K^{\rm  cross}_{22}
\end{pmatrix}  .$$

\section{Pfaffian Schur process}
\label{sec:PSP}
The key tool in our analysis of last passage percolation on a half-space is the  Pfaffian Schur process. In this Section, we first introduce the Pfaffian Schur process as in \cite{borodin2005eynard}. This is a Pfaffian analogue of the determinantal Schur process  introduced in \cite{okounkov2003correlation}. We refer to \cite[Section 6]{borodin2012lectures} and references therein for background on the Schur process.  Then, in order to connect it to last passage percolation, we introduce an equivalent presentation of the  Pfaffian Schur processes indexed by  lattice paths in the half-quadrant, and we study dynamics preserving the Pfaffian Schur process.

\subsection{Definition of the Pfaffian Schur process}
\label{sec:defPSP}
\subsubsection{Partitions and Schur positive specializations}
Pfaffian Schur processes are measures on sequences of integer partitions. A partition is a nonincreasing sequence of nonnegative integers
$ \lambda = (\lambda_1\geqslant \lambda_2 \geqslant \dots \geqslant \lambda_{k}\geqslant 0), $
with finitely many non-zero components.
Given two partitions $\lambda, \mu$, we write $\mu\subset\lambda$ if $\lambda_i\geqslant \mu_i$ for all $i\geqslant 1$.  We write $\mu\prec\lambda$ and say that $\mu$ interlaces $\lambda$ if for all $i\geqslant 1$,
$ \lambda_i\geqslant \mu_i \geqslant \lambda_{i-1}. $
%
%
%
%
%
For a given partition $\lambda$, its dual, denoted $\lambda'$, is the partition corresponding to the Young diagram which is symmetric to the Young diagram of $\lambda$ with respect to the first diagonal. In other words, $\lambda'_i = \sharp\lbrace j : \lambda_j \geqslant i \rbrace$. We say that a partition $\lambda$ is even if all its component are even integers. Thus, for a partition $\lambda$, we say that its dual $\lambda'$ is even when we have  $\lambda_1=\lambda_2, \lambda_3=\lambda_4, $ etc. We denote by $\vert \lambda \vert $ the number of boxes in the diagram corresponding to $\lambda$. 

The probabilities in the Pfaffian Schur process -- and the usual Schur process as well --  are expressed in terms of skew Schur symmetric functions $ s_{\lambda/\mu}$ indexed by pairs of partitions $\mu,\lambda$. We use the convention that $s_{\lambda/\mu}=0 $ if $\mu\not\subset \lambda$. We refer to \cite[Section 2]{borodin2012lectures} for a definition of (skew) Schur functions and  background on symmetric functions. We record later in Section \ref{sec:identities} the few properties of Schur functions that we will use.

A specialization $\rho$ of the algebra $\mathrm{Sym}$ of symmetric functions is an algebra morphism from $\mathrm{Sym}$ to $\C$. For example, the evaluation of a symmetric function into one fixed variable $\alpha\in \C$ defines such a morphism. We denote $f(\rho)$ the application of the specialization $\rho$ to $f\in \mathrm{Sym}$. A specialization can be defined by its values on a basis of $\mathrm{Sym}$ and we generally choose the power sum symmetric functions
$ p_n(x_1, x_2, \dots) = \sum x_i^n .$
For two specializations $\rho_1$ and $\rho_2$, their union, denoted $(\rho_1, \rho_2)$, is defined by
$$ p_n(\rho_1, \rho_2) = p_n(\rho_1)+ p_n(\rho_2) \text{ for all }n\geqslant 1. $$
More generally, the specialization $(\rho_1, \dots, \rho_n)$ is the union of specializations $\rho_1, \dots, \rho_n$. When $\rho_1$ and $\rho_2$ are  specializations corresponding to the evaluation into sets of variables, then $(\rho_1, \rho_2)$ corresponds to the evaluation into the union of both sets of variables, hence the term union for this operation on specializations.

Schur nonnegative specializations of $\mathrm{Sym}$ are specializations taking values in $\R_{\geqslant 0}$ when  applied to skew Schur functions $s_{\lambda/\mu}$ for any partitions $\lambda$ and $\mu$.
Thoma's theorem (see \cite{kerov2003asymptotic} and references therein) provides a classification of such specializations. Let $\alpha = \big( \alpha_i\big)_{i\geqslant 1}$, $\beta = \big( \beta_i\big)_{i\geqslant 1}$ and $\gamma$ be non-negative numbers such that $\sum (\alpha_i+\beta_i ) <\infty$. Any Schur nonnegative specialization $\rho$ is determined by parameters $(\alpha, \beta, \gamma)$ through the formal series identity
$$ \sum_{n\geqslant 0} h_n(\rho) z^n  = \exp(\gamma z) \prod_{i\geqslant 1} \frac{1+\beta_i z}{1-\alpha_i z} =:H(z; \rho).$$
where the $h_n$ are complete homogeneous symmetric functions. The specialization $(0,0,\gamma)$ is called (pure-)Plancherel and will be denoted by $\Plancherel(\gamma)$. It can be obtained as a limit of specializations $(\alpha, 0, 0)$ where $\alpha$ is the sequence of length $M$, $(\gamma/M, \dots, \gamma/M)$, when $M$ is sent to infinity. From now on, we assume that all specializations are always Schur non-negative.
\subsubsection{Useful identities}
\label{sec:identities}
We collect some identities from \cite{macdonald1995symmetric} that were used in  \cite{borodin2005eynard} to define the Pfaffian Schur process. The sums below (as in \eqref{eq:skewCauchy}) are always taken over all partitions, unless otherwise restricted. Starting with one  already useful for the determinantal Schur process, we have the skew-Cauchy identity
\begin{equation}
\sum_{\nu} s_{\nu/\lambda}(\rho)s_{\nu/\mu}(\rho')  = H(\rho ; \rho') \sum_{\tau} s_{\lambda/\tau}(\rho')s_{\mu/\tau}(\rho),
\label{eq:skewCauchy}
\end{equation}
where,  if $\rho$ and $\rho'$ are the specializations into  sets of variables $\lbrace x_i\rbrace , \lbrace y_i\rbrace$,
$$  H(\rho ; \rho') = \prod_{i,j} \frac{1}{1-x_iy_j}.$$
We also use the branching rule for Schur functions
\begin{equation}
\sum_{\nu} s_{\lambda/\nu}(\rho) s_{\nu/\mu}(\rho') = s_{\lambda/\mu}(\rho, \rho').
\label{eq:branchingrule}
\end{equation}

The following identity is crucial to go from the  determinantal Schur processes to its Pfaffian analog. It is a  variant of the skew-Littlewood identity,
\begin{equation}
\sum_{\nu' \text{even}} s_{\nu/\lambda}(\rho) = H^\circ(\rho) \sum_{\kappa' \text{even}} s_{\lambda/\kappa}(\rho),
\label{eq:skewLittlewoodvariant}
\end{equation}
where if $\rho$ is the specialization into a set of variables $\lbrace x_i\rbrace$,
$$  H^\circ(\rho) = \prod_{i<j} \frac{1}{1-x_ix_j}.$$
In particular,
$$  \sum_{\lambda' \text{even}} s_{\lambda}(\rho)  = H^o(\rho).$$

\subsubsection{Definition of the Pfaffian Schur process} 
\begin{definition}[\cite{borodin2005eynard}]
	The Pfaffian Schur process parametrized by two sequences of Schur nonnegative specializations $\rho_{0}^+,\rho_{1}^+ , \dots, \rho_{n-1}^+  $ and $ \rho_1^-, \dots, \rho_n^-$, denoted  $\PSP[\rho_{0}^+;\rho_{1}^+; \dots; \rho_{n-1}^+ \vert  \rho_1^-; \dots; \rho_n^-]$, is a probability measure on sequences of integer partitions
	$$ \varnothing \subset \lambda^{(1)} \supset \mu^{(1)} \subset \lambda^{(2)} \supset \mu^{(2)} \dots \mu^{(n-1)} \subset \lambda^{(n)} \supset \varnothing. $$
	Under this measure, the probability of the sequence  $ \bar \lambda = (\lambda^{(1)}, \dots, \lambda^{(n)})$, $\bar\mu  = (\mu^{(1)}, \dots, \mu^{(n)})  $ is given by
	$$ \PSP[\rho_{0}^+;\rho_{1}^+; \dots; \rho_{n-1}^+ \vert  \rho_1^-; \dots; \rho_n^-]\big(\bar\lambda, \bar\mu\big) = \frac{\mathcal{V} \big(\bar\lambda, \bar\mu\big)}{Z^{\circ}(\rho)},$$
	where $Z^{\circ}(\rho)$ is a renormalization constant depending on the choice of specializations and
	the weight $\mathcal{V} (\bar\lambda, \bar\mu)$ is given by
	\begin{equation}
	\mathcal{V} \big(\bar \lambda, \bar \mu\big)  = \tau_{\lambda^{(1)}}(\rho_{0}^+)s_{\lambda^{(1)}/\mu^{(1)}}(\rho_1^-)s_{\lambda^{(2)}/\mu^{(1)}}(\rho_1^+) \dots s_{\lambda^{(n)}/\mu^{(n-1)}}(\rho_{n-1}^+)s_{\lambda^{(n)}}(\rho_n^-),
	\label{eq:weightPSP}
	\end{equation}
	where
	$$ \tau_{\lambda} = \sum_{\kappa' \text{even}} s_{\lambda/\kappa}.$$
	\label{def:PSP}
\end{definition}
We may encode the choice of specializations by the diagram shown in Figure \ref{fig:diagram}.
\begin{figure}
	\begin{center}
		\begin{tikzpicture}[scale=1.4]
		\usetikzlibrary{arrows}
		\usetikzlibrary{shapes}
		\usetikzlibrary{shapes.multipart}
		\tikzstyle{topartition}=[thick]
		\node (emptyleft) at (0,0) {$\varnothing$};
		\node (lambda1) at (1,1) {$\lambda^{(1)}$};
		\node (mu1) at (2,0) {$\mu^{(1)}$};
		\node (lambda2) at (3,1) {$\lambda^{(2)}$};
		\node (mu2) at (4,0) {};
		\node (dots) at (5,0.1) {$\dots$};
		\node (muend) at (6,0) {$\mu^{(n-1)}$};
		\node (lambdaend) at (7,1) {$\lambda^{(n)} $};
		\node (emptyright) at (8,0) {$ \varnothing$};
		\draw[topartition] (emptyleft) -- (lambda1) node[midway, anchor=south east]{{\footnotesize  $\rho_{0}^+$ }};
		\draw[topartition] (lambda1) -- (mu1) node[midway, anchor=south west]{{\footnotesize  $\rho_1^- $}};
		\draw[topartition] (mu1) -- (lambda2) node[midway, anchor=south east]{{\footnotesize  $\rho_1^+$ }};
		\draw[topartition] (lambda2) -- (mu2);
		\draw[topartition] (muend) -- (lambdaend) node[midway, anchor=south east]{{\footnotesize  $\rho_{n-1}^+$ }};
		\draw[topartition] (lambdaend) -- (emptyright) node[midway, anchor=south west]{{\footnotesize  $\rho_n^- $}};
		\end{tikzpicture}
	\end{center}
	\caption{Diagram corresponding to The Pfaffian Schur process $\PSP[\rho_{0}^+;\rho_{1}^+; \dots; \rho_{n-1}^+ \vert  \rho_1^-; \dots; \rho_n^-]$. The ascending and descending links represent inclusions of partitions.}
	\label{fig:diagram}
\end{figure}
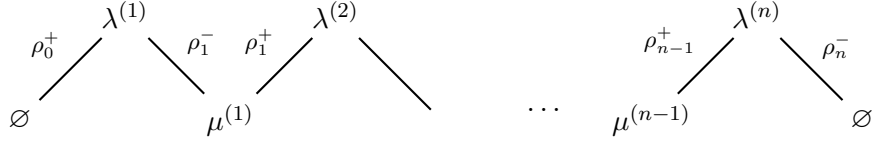
The normalization constant $Z^{\circ}(\rho)$ for the weights $\mathcal{V} (\bar\lambda, \bar\mu) $ is 
$$Z^{\circ} (\rho) := \sum_{ \varnothing \subset \lambda^{(1)} \supset \mu^{(1)} \subset \lambda^{(2)} \supset \mu^{(2)} \dots \mu^{(n-1)} \subset \lambda^{(n)} \supset \varnothing} \mathcal{V} (\bar\lambda, \bar\mu), $$
and can be  computed  \cite{borodin2005eynard} to be 
$$  Z^{\circ} (\rho) = H^{\circ}(\rho^-_1, \dots, \rho^-_n) \prod_{0\leqslant i<j\leqslant n} H(\rho_i^+ ; \rho_j^-).$$

\begin{definition}
	The Pfaffian Schur measure $\PSM[\rho\vert \rho']$ is a probability measure on a single partition $\lambda$ such that
	$$\PSM[\rho\vert\rho']\big(\lambda\big)  = \frac{1}{H^o(\rho') H(\rho; \rho')} \tau_{\lambda}(\rho) s_{\lambda}(\rho'),$$
	\label{def:PSM}
\end{definition}
\begin{lemma}
	The law of $\lambda^{(k)}$ under the Pfaffian Schur process \\
	$\PSP[\rho_{0}^+;\rho_{1}^+; \dots; \rho_{n-1}^+ \vert  \rho_1^-; \dots; \rho_n^-]$
	is the Pfaffian Schur measure\\
	$\PSM[\rho_{0}^+, \dots \rho_{k-1}^+ \vert \rho_{k}^-, \dots, \rho_n^-]. $
	\label{lem:marginalsPSP}
\end{lemma}
\begin{proof}
	Summing over all partitions except $\lambda^{(k)}$ in \eqref{eq:weightPSP} and using identities \eqref{eq:skewCauchy}, \eqref{eq:branchingrule}  and \eqref{eq:skewLittlewoodvariant} yields the result.
\end{proof}
The following property of the Pfaffian Schur measure will be useful to interpret the results in Section \ref{sec:mainasymptoticresults}.
\begin{proposition}[Corollary 7.6 \cite{baik2001algebraic}]
	Let $c$ be a positive real and $\rho$ a Schur nonnegative specialization.
	If $\mu$ is distributed according to $\PSM[c\vert \rho]$ and $\lambda$ is distributed according to $\PSM[0 \vert c,  \rho]$, then we have the equality in law
	$ \lambda_1 \overset{(d)}{=} \mu_1.$
	\label{prop:diagorow}
\end{proposition}
\begin{proof}
	By definition, we have that
	\begin{equation}
	\PP(\lambda_1\leqslant x) = \sum_{\lambda: \lambda_1\leqslant x} \frac{\tau_{\lambda}(0) s_{\lambda}(c, \rho)}{H^o(c, \rho)}, \ \ \PP(\mu_1\leqslant x) = \sum_{\mu: \mu_1\leqslant x} \frac{\tau_{\mu}(c) s_{\mu}(\rho)}{H^o(\rho)H(c; \rho)}.
	\label{eq:deflambda}
	\end{equation}
	The normalization constants $H^o(c, \rho)$ and $H^o(\rho)H(c; \rho)$ are equal. Since $\tau_{\lambda}(0)=\mathds{1}_{\lambda' \text{ even}}$, we have that
	\begin{align*}
	L.H.S. \eqref{eq:deflambda} &= \frac{1}{H^o(\rho)H(c; \rho)} \sum_{\underset{\lambda'\text{ even}}{\lambda: \lambda_1\leqslant x}} s_{\lambda}(c, \rho)\\
	&= \frac{1}{H^o(\rho)H(c; \rho)} \sum_{\underset{\lambda'\text{ even}}{\lambda: \lambda_1\leqslant x}} \sum_{\mu} s_{\mu}(\rho)s_{\lambda/\mu}(c)  \text{ (Branching rule)}\\
	&=\frac{1}{H^o(\rho)H(c; \rho)} \sum_{\mu: \mu_1\leqslant x} s_{\mu}(\rho) c^{\mu_1 - \mu_2+ \mu_3-\mu_4+ \dots}\\
	&=\frac{1}{H^o(\rho)H(c; \rho)} \sum_{\mu: \mu_1\leqslant x} s_{\mu}(\rho) \tau_{\mu}(c)\\
	&=R.H.S. \eqref{eq:deflambda}.
	\end{align*}
	Note that in the third equality, we use that if $\lambda'$ is even and $\mu\prec\lambda$, then we have for all $i\geqslant 1$, $\lambda_{2i-1}=\mu_{2i-1}=\lambda_{2i}$. Thus, the fact that $c$ is a single variable specialization is necessary. The fourth equality comes from the fact that for a single variable specialization $c$,
	\begin{equation}
	\tau_{\mu}(c) = \sum_{\kappa' \text{ even}} s_{\lambda/\kappa}(c)  = c^{\mu_1 - \mu_2+ \mu_3-\mu_4+ \dots}. \label{eq:calcultau}
	\end{equation}
	Indeed, there is only one partition $\kappa$ which gives a nonzero contribution in the sum in \eqref{eq:calcultau}.
\end{proof}
\begin{remark}
	The proof of Prop. \ref{prop:diagorow} actually shows that
	$(\lambda_1, \lambda_3, \lambda_5, \dots)$ has the same law as  $ (\mu_1, \mu_3, \mu_5, \dots),$
	but we will use only the first coordinates in applications.
\end{remark}

\subsection{Pfaffian Schur processes indexed by zig-zag paths}

\label{sec:zigzagformulation}
In Definition \ref{def:PSP}, Pfaffian Schur processes were determined by two sequences of specializations $\rho_{\circ}^+, \dots, \rho_{n-1}^+$ and $\rho_1^-, \dots, \rho_{n}^-$.
We describe here an equivalent  way of representing Pfaffian Schur processes, indexing them  by certain zig-zag  paths in the first quadrant of $\Z^2$ (as depicted in Figure \ref{fig:downrightpath}) where each edge of the path is labelled by some Schur nonnegative specialization. This representation is convenient  for defining dynamics on Pfaffian Schur processes.

More precisely, we consider oriented paths starting on the horizontal axis at $(+\infty,0)$,  proceeding by unit steps horizontally to the left or vertically to the top, hitting the diagonal at some point $(n,n)$ for some $n\in \Z_{>0}$ and then taking one final edge of length $n\sqrt{2}$ to the origin $(0,0)$ (see Figure \ref{fig:downrightpath}). Let us denote $\Omega$ the set of all such paths. For $\gamma\in \Omega$, let us denote $E(\gamma)$ its set of edges and $V(\gamma)$ its set of vertices. We denote $E^{\uparrow}(\gamma)$ the set of vertical edges and $E^{\leftarrow}(\gamma)$ the set of horizontal edges. Each edge $e$ is labelled by a Schur nonnegative specialization $\rho_e$.  We label by a specialization $\rho_{\circ}$ the edge joining the point $(n,n)$ to $(0,0)$. Each vertex $v\in \gamma$ indexes a  partition $\lambda^{v}$.
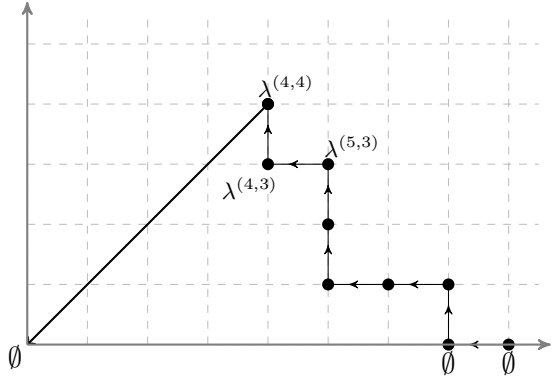
\begin{figure}
	\begin{center}
		\begin{tikzpicture}[scale=0.8]
		\tikzstyle{fleche}=[>=stealth', postaction={decorate}]
		
		\begin{scope}[decoration={
			markings,
			mark=at position 0.5 with {\arrow{<}}}]
		\draw[dashed, gray!50] (0,0) grid(8.5,5.5);
		\draw[thick, <-] (0,0) -- (4,4);
		\draw[fleche] (4,4) -- (4,3);
		\draw[fleche] (4,3)-- (5,3);
		\draw[fleche] (5,3) -- (5,2);
		\draw[fleche] (5,2) -- (5,1);
		\draw[fleche] (5,1) -- (6,1);
		\draw[fleche] (6,1) -- (7,1);
		\draw[fleche] (7,1) -- (7,0);
		\draw[fleche]  (7,0) -- (8,0);
		\fill (4,4) circle(0.1);
		\fill (4,4) circle(0.1);
		\fill (4,3) circle(0.1);
		\fill (5,3) circle(0.1);
		\fill (5,2) circle(0.1);
		\fill (5,1) circle(0.1);
		\fill (6,1) circle(0.1);
		\fill (7,1) circle(0.1);
		\fill (7,0) circle(0.1);
		\fill (8,0) circle(0.1);
		\draw (-0.2,-0.2) node{$\emptyset$};
		\draw (7,-0.3) node{$\emptyset$};
		\draw (8,-0.3) node{$\emptyset$};
		\draw (4.3,4.3) node{{\footnotesize $\lambda^{(4,4)}$}};
		\draw (3.7,2.6) node{{\footnotesize $\lambda^{(4,3)}$}};
		\draw (5.4,3.3) node{{\footnotesize $\lambda^{(5,3)}$}};
		\end{scope}
		\draw[->, >=stealth', thick, gray] (0,0) -- (0,5.7) ;
		\draw[->, >=stealth', thick, gray] (0,0) -- (8.7, 0) ;
		\end{tikzpicture}
	\end{center}
	\caption{A zig-zag path indexing the Pfaffian Schur process.}
	\label{fig:downrightpath}
\end{figure}

For $\gamma\in \Omega$,
the Pfaffian Schur process indexed by $\gamma$ and the specializations on $E(\gamma)$  is a probability measure on the sequence of partitions $\uplambda:=\big( \lambda^{v}\big)_{v\in V(\gamma)}$ where the weight of $\lambda$ is given by
$$ \mathcal{V}(\uplambda) = \tau_{\lambda^{(n,n)}}(\rho_{\circ})\prod_{e\in E^{\uparrow}(\gamma)\cup E^{\leftarrow}(\gamma) }  \mathcal{W}(e) , $$
where for an edge $e_1 \in E^{\leftarrow}(\pi)$ and for an edge $e_2 \in E^{\uparrow}(\pi)$, 
$$ \mathcal{W}\Big( e_1=\tikz[scale=0.4, baseline=-3pt]{\draw (0,0) node{$v$};\draw[thick, <-] (0.5, 0) -- (1.5,0);\draw (2,0) node{$u$}; } \Big) = s_{\lambda^{u}/\lambda^{v}}(\rho_{e_1}) \ \  \text{and} \  \ \mathcal{W}\bigg(e_2=\tikz[scale=0.4, baseline=7pt]{\draw (0,0) node{$u$};\draw[thick, ->] (0, 0.5) -- (0,1.5);\draw (0,2) node{$v$}; } \bigg) = s_{\lambda^{v}/\lambda^{u}}(\rho_{e_2}).$$
We adopt the convention that for all vertices $v$ on the horizontal axis, $\lambda^v$ is the empty partition, which in particular implies that specializations are empty on the horizontal axis. We also recall that  $s_{\lambda/\mu}\equiv 0$ if $\mu\not\subset \lambda$. Using the above formalism, the normalisation constant associated to the weights $\mathcal{V}(\uplambda)$ is given  by
$$ Z^{\circ}(\gamma) = \prod_{e\in E^{\uparrow}(\gamma)} H^{\circ}(\rho_e) H(\rho_{\circ} ; \rho_e) \prod_{\underset{e>e'}{e\in E^{\uparrow}(\gamma),\  e'\in E^{\leftarrow}(\gamma)}} H(\rho_e ; \rho_{e'}),$$
where $e>e'$ means that the edge $e$ occurs after the edge $e'$ in the oriented path $\gamma$.

\begin{remark}
	The above description is equivalent to the one given in Definition \ref{def:PSP}. As we have defined in Section \ref{sec:defPSP}, the Pfaffian Schur process is a measure on sequences
	$$ \varnothing \subset \lambda^{(1)} \supset \mu^{(1)} \subset \lambda^{(2)} \supset \mu^{(2)} \dots \mu^{(n-1)} \subset \lambda^{(n)} \supset \varnothing.$$
	However, by using empty specializations, one can force that $\lambda^{(i)} = \mu^{(i)}$ or $ \mu^{(i)}= \lambda^{(i+1)}$. Hence, we can consider sequences where the order of inclusions is any word on the alphabet $\lbrace \subset, \supset\rbrace$. By matching the inclusions $ \supset$ with vertical edges in the zig-zag path formulation and the inclusions  $\subset $ with horizontal edges, we see that both formulations are equivalent (zig-zag paths are in bijection with finite words over a two-letter alphabet). See for instance Figure \ref{fig:choicespecializations} for a particular example.  This zig-zag construction also applies to the usual Schur process \cite{okounkov2003correlation}. In that case, the processes are indexed by paths from a point on the horizontal axis to a point on the vertical axis.
\end{remark}

\subsection{Dynamics on Pfaffian Schur processes}
\label{sec:dynamics}
Here we give a sequential construction of Pfaffian Schur processes using the formalism of Section
\ref{sec:zigzagformulation} by defining Markov chains preserving the Pfaffian Schur structure. More precisely, we will define Markov dynamics such that the pushforward of the Pfaffian Schur process indexed by a path $\gamma\in \Omega$ and a set of specializations on $E(\gamma)$ is the Pfaffian Shur process indexed by a path $\gamma'$ and the same set of specializations, where $\gamma'$ is obtained from $\gamma$ by adding one box to the shape it defines, or half a box when it grows along the diagonal. These dynamics are a special case of those developed in \cite{barraquand2016pfaffian} with Alexei Borodin in the Macdonald case \cite{borodin2014macdonald}. They are an adaptation to the Pfaffian setting of Markov dynamics preserving the (determinantal) Schur process introduced in \cite[Section 2]{borodin2008anisotropic} and studied in a more general setting in \cite{borodin2011schur} (see the review  \cite[Section 6.4]{borodin2012lectures}).

Let us explain precisely how to obtain a path
$\gamma\in \Omega$ as the outcome of a sequence of elementary moves (see Figure \ref{fig:growingpath}).
$(i)$ We start with the path with only one vertex $(0,0)$.
$(ii)$ We may first make the path grow along the diagonal and get $ (1, 0) \to (1,1) \to (0,0)$.
$(iii)$ We can always freely move to the right the starting point of the path. For instance, assume that we get the new path $(2,0) \to (1, 0) \to (1,1) \to (0,0)$.
$(iv)$ Then, we may make the path grow by one box and get $ (2, 0) \to (2,1) \to (1,1) \to (0,0)$.
$(v)$ We will continue iteratively adding boxes to the path or growing along the diagonal, until we arrive at $\gamma$.

In a parallel way, we construct a sequence of Pfaffian Schur processes for each path described above.
$(i)$ We start with an empty Pfaffian Schur process $
\lambda^{(0,0)} = \varnothing$.
$(ii)$ The application of Markov dynamics -- that we shall define momentarily --
corresponding to the growth of the path along the diagonal will  define a Pfaffian Schur process on $\varnothing = \lambda^{(1,0)} \subset \lambda^{(1,1)} \supset \lambda^{(0,0)}= \varnothing$.
$(iii)$ By convention, the partitions indexed by vertices on the horizontal axis are empty under the Pfaffian Schur process. Hence, sliding the starting point of the path to the right has no effect in terms of Pfaffian Schur process.
$(iv)$ Then, the application of Markov dynamics corresponding to the growth of the path by one box at the corner $(1,0)$ will define a Pfaffian Schur process $\varnothing = \lambda^{(2,0)} \subset \lambda^{(2,1)} \supset \lambda^{(1,1)} \supset \lambda^{(0,0)}=\varnothing$.
$(v)$ We continue iteratively by applying Markov operators that update one partition in the Pfaffian Schur process.
\begin{figure}
	\begin{center}
		\begin{tikzpicture}[scale=0.85]
		\tikzstyle{fleche}=[>=stealth', postaction={decorate}]
		\tikzstyle{axis}=[->, >=stealth', thick, gray]
		\draw[axis] (0,0) -- (0,1.5) ;
		\draw[axis] (0,0) -- (1.5, 0) ;
		\begin{scope}[decoration={
			markings,
			mark=at position 0.5 with {\arrow{<}}}]
		\draw[dashed, gray] (0,0) grid(1.5, 1.5);
		\fill (0,0) circle(0.1);
		\draw (-0.2,-0.3) node{$\lambda^{(0,0)} = \emptyset$};
		\draw (1, -0.8) node{$(i)$};
		\end{scope}
		
		\begin{scope}[xshift=2cm]
		\draw[axis] (0,0) -- (0,2.5) ;
		\draw[axis] (0,0) -- (2.5, 0) ;
		\begin{scope}[decoration={
			markings,
			mark=at position 0.5 with {\arrow{<}}}]
		\draw[dashed, gray] (0,0) grid(2.5, 2.5);
		\draw[thick, <-] (0,0) -- (1,1);
		\draw[fleche] (1,1) -- (1,0);
		\fill (1,1) circle(0.1);
		\fill (1,0) circle(0.1);
		\draw (-0.2,-0.2) node{$\emptyset$};
		\draw (1,-0.3) node{$\emptyset$};
		\draw (1.3,1.3) node{{\footnotesize $\lambda^{(1,1)}$}};
		\draw (1, -0.8) node{$(ii)$};
		\end{scope}
		\end{scope}
		
		\begin{scope}[xshift=5cm]
		\draw[axis] (0,0) -- (0,2.5) ;
		\draw[axis] (0,0) -- (2.5, 0) ;
		\begin{scope}[decoration={
			markings,
			mark=at position 0.5 with {\arrow{<}}}]
		\draw[dashed, gray] (0,0) grid(2.5, 2.5);
		\draw[thick, <-] (0,0) -- (1,1);
		\draw[fleche] (1,1) -- (1,0);
		\draw[fleche] (1,0) -- (2,0);
		\fill (1,1) circle(0.1);
		\fill (1,0) circle(0.1);
		\fill (2,0) circle(0.1);
		\draw (-0.2,-0.2) node{$\emptyset$};
		\draw (1,-0.3) node{$\emptyset$};
		\draw (2,-0.3) node{$\emptyset$};
		\draw (1.3,1.3) node{{\footnotesize $\lambda^{(1,1)}$}};
		\draw (1, -0.8) node{$(iii)$};
		\end{scope}
		\end{scope}

		\begin{scope}[xshift=8cm]
		\draw[axis] (0,0) -- (0,2.5) ;
		\draw[axis] (0,0) -- (2.5, 0) ;
		\begin{scope}[decoration={
			markings,
			mark=at position 0.5 with {\arrow{<}}}]
		\draw[dashed, gray] (0,0) grid(2.5, 2.5);
		\draw[thick, <-] (0,0) -- (1,1);
		\draw[fleche] (1,1) -- (2,1);
		\draw[fleche] (2,1) -- (2,0);
		\fill (1,1) circle(0.1);
		\fill (2,1) circle(0.1);
		\fill (2,0) circle(0.1);
		\draw (-0.2,-0.2) node{$\emptyset$};
		\draw (2,-0.3) node{$\emptyset$};
		\draw (1.3,1.3) node{{\footnotesize $\lambda^{(1,1)}$}};
		\draw (2.3,1.3) node{{\footnotesize $\lambda^{(2,1)}$}};
		\draw (1, -0.8) node{$(iv)$};
		\end{scope}
		\end{scope}
		
		\begin{scope}[xshift=11cm]
		\draw[axis] (0,0) -- (0,2.5) ;
		\draw[axis] (0,0) -- (2.5, 0) ;
		\begin{scope}[decoration={
			markings,
			mark=at position 0.5 with {\arrow{<}}}]
		\draw[dashed, gray] (0,0) grid(2.5, 2.5);
		\draw[thick, <-] (0,0) -- (2,2);
		\draw[fleche] (2,2) -- (2,1);
		\draw[fleche] (2,1) -- (2,0);
		\fill (2,2) circle(0.1);
		\fill (2,1) circle(0.1);
		\fill (2,0) circle(0.1);
		\draw (-0.2,-0.2) node{$\emptyset$};
		\draw (2,-0.3) node{$\emptyset$};
		\draw (2.3,2.3) node{{\footnotesize $\lambda^{(2,2)}$}};
		\draw (2.4,1.3) node{{\footnotesize $\lambda^{(2,1)}$}};
		\draw (1, -0.8) node{$(v)$};
		\end{scope}
		\end{scope}
		\end{tikzpicture}
	\end{center}
	\caption{Illustration of the first steps according to which the path grows.}
	\label{fig:growingpath}
\end{figure}
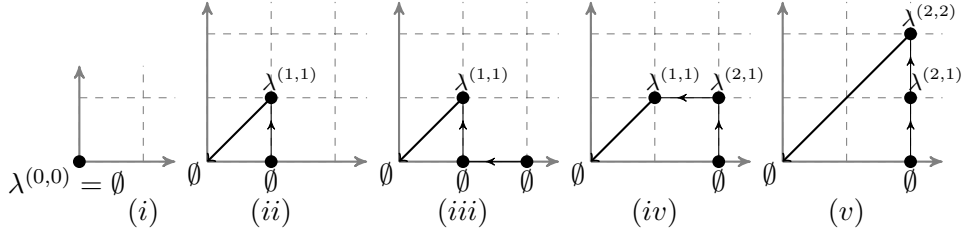

We now have to define the above-mentioned Markov dynamics, and explain how the specializations are chosen. We have seen that there are two distinct cases corresponding to the growth of the path by one box or the growth by a half box along the diagonal.

We define a transition operator $\mathcal{U}^{\llcorner}_{\rho_1, \rho_2}$ corresponding to the growth of the path by one box, as a probability distribution  $\mathcal{U}^{\llcorner}_{\rho_1, \rho_2}(\pi \vert \nu, \mu, \kappa)$ on partitions $\pi$, given partitions  $\nu, \mu, \kappa$ with $\mu\subset\nu, \kappa$. In terms of growing paths, this operator corresponds to adding one box in a corner formed by partitions $\mu\subset\nu, \kappa$, and it will update the partition $\mu $ to a partition $\pi$ containing $\nu$ and $\kappa$, in such a way  that the corner formed by partition $\kappa, \pi, \nu$ is the  marginal of a new Pfaffian Schur process. Pictorially, the action of the transition operator $\mathcal{U}^{\llcorner}_{\rho_1, \rho_2}$ can be represented by the following diagram.
\begin{center}
	\begin{tikzpicture}[scale=01]
	\node (v1) at (0,1) {$\nu$};
	\node (m) at (0,0) {$\mu$};
	\node (k1) at (1,0) {$\kappa$};
	\draw[thick, <-] (v1) -- (m) node[midway, anchor=east]{$\rho_1$};
	\draw[thick, <-] (m) -- (k1) node[midway, anchor=north]{$\rho_2$};
	
	\node[gray] (p1) at (1,1) {$?$};
	\draw[thick, gray, <-] (v1) -- (p1) node[midway, anchor=south]{$\rho_2$};
	\draw[thick, gray, <-] (p1) -- (k1) node[midway, anchor=west]{$\rho_1$};
	
	\draw[thick, ->, >=stealth'] (2,0.5) --(4,0.5) node[midway, anchor=south] {$\mathcal{U}^{\llcorner}_{\rho_1, \rho_2}$};
	
	\node (v2) at (5,1) {$\nu$};
	\node (p) at (6,1) {$\pi$};
	\node (k2) at (6,0) {$\kappa$};
	\draw[thick,<-] (v2) -- (p) node[midway, anchor=south]{$\rho_2$};
	\draw[thick, <-] (p) -- (k2) node[midway, anchor=west]{$\rho_1$};
	\node[gray] (m2) at (5,0) {$\mu$};
	\draw[thick, gray,<-] (v2) -- (m2) node[midway, anchor=east]{$\rho_1$};
	\draw[thick, gray, <-] (m2) -- (k2) node[midway, anchor=north]{$\rho_2$};

	\end{tikzpicture}
\end{center}
$\mathcal{U}^{\llcorner}_{\rho_1, \rho_2}$ updates the Pfaffian Schur process formed by partitions $\kappa \supset \mu \subset \nu$ on the left to the the Pfaffian Schur process formed by partitions $\kappa \subset \pi \supset \nu$ on the right. After the update, one can forget the information supported by the gray arrows on the right.
The specializations in the new Pfaffian Schur process are carried from the previous step by the transition operator as shown above.
\begin{lemma}
	Assume that  $\mathcal{U}^{\llcorner}_{\rho_1, \rho_2}$ satisfies
	\begin{equation}
	\sum_{\mu} s_{\kappa/\mu}(\rho_2) s_{\nu/\mu}(\rho_1)\mathcal{U}^{\llcorner}_{\rho_1, \rho_2}(\pi \vert \nu, \mu, \kappa) = \frac{s_{\pi/\kappa}(\rho_1)s_{\pi/\nu}(\rho_2)}{H(\rho_2 ; \rho_1)}.
	\label{eq:preserveschur}
	\end{equation}
	Then $\mathcal{U}^{\llcorner}_{\rho_1, \rho_2}$ preserves the Pfaffian Schur process measure in the following sense.
	Let $\gamma'\in \Omega$ contain the corner defined by the vertices
	$$ (i+1, j)\xrightarrow{\rho_2} (i,j)\xrightarrow{\rho_1} (i, j+1), $$
	where $\rho_1$ and $\rho_2$ are the specializations indexing edges of the corner.
	Let  $\gamma''\in \Omega$ contain the same set of vertex as $\gamma'$ except that the vertex $(i,j)$ is replaced by the vertex $ (i+1, j+1) $, and assume that the specializations are now chosen as
	$$  (i+1, j) \xrightarrow{\rho_1} (i+1,j+1) \xrightarrow{\rho_2} (i, j+1).$$
	Then $\mathcal{U}^{\llcorner}_{\rho_1, \rho_2}(\lambda^{(i+1, j+1)} \vert \lambda^{(i, j+1)}, \lambda^{(i,j)}, \lambda^{(i+1, j)}  )$ maps the Pfaffian Schur process indexed by $\gamma'$ and the set of specializations on $E(\gamma')$ to the Pfaffian Schur process indexed by $\gamma''$ and the same set of specializations.
\end{lemma}
\begin{proof}
	The relation \eqref{eq:preserveschur} implies that applying $\mathcal{U}^{\llcorner}_{\rho_1, \rho_2}$ to the  Pfaffian Schur process indexed by $\gamma'$ and averaging over $\lambda^{(i,j)}$ gives a weight proportional to
	$$ s_{\lambda^{(i+1, j+1)}/\lambda^{(i+1, j)}}(\rho_1)s_{\lambda^{(i+1, j+1)}/\lambda^{(i, j+1)}}(\rho_2),$$
	as in the Pfaffian Schur process indexed by $\gamma''$. The normalization $H(\rho_2 ; \rho_1)$ accounts for the fact that $Z^{\circ}(\gamma'') = H(\rho_2 ; \rho_1) Z^{\circ}(\gamma')$.
\end{proof}
We choose a particular solution to  \eqref{eq:preserveschur} given by
\begin{equation}
\mathcal{U}^{\llcorner}_{\rho_1, \rho_2}(\pi \vert \nu, \mu, \kappa) =\mathcal{U}^{\llcorner}_{\rho_1, \rho_2}(\pi \vert \nu, \kappa) =  \frac{s_{\pi/\nu}(\rho_2) s_{\pi/\kappa}(\rho_1)}{\sum_{\lambda}s_{\lambda/\nu}(\rho_2) s_{\lambda/\kappa}(\rho_1)}.
\label{eq:defucorner}
\end{equation}
The Cauchy identity \eqref{eq:skewCauchy} ensures that this choice satisfies \eqref{eq:preserveschur}. There exists  other solutions to \eqref{eq:preserveschur} (see \cite{borodin2013nearest, matveev2015q}) where
$\mathcal{U}^{\llcorner}(\pi \vert \nu, \mu, \kappa)$ depends on $\mu$.

We also define a diagonal transition operator
$\mathcal{U}^{\angle}_{\rho_{\circ}, \rho_1}$ corresponding to the growth of the path by a half box along the diagonal as a probability distribution $ \mathcal{U}^{\angle}_{\rho_{\circ}, \rho_1}(\pi \vert \kappa, \mu) $ on partitions $\pi$ given partitions $\mu\subset \kappa$. It will update the partition $\mu\subset \kappa$ to a partition $\pi\supset \kappa$ such that $\pi$ is a marginal of a new Pfaffian Schur process. Pictorially,
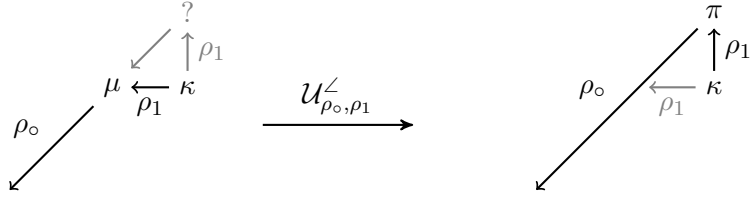
\begin{center}
	\begin{tikzpicture}[scale=1]
	\node (m) at (0,0) {$\mu$};
	\node (bas) at (-1.5,-1.5){};
	\node (k1) at (1,0) {$\kappa$};
	\draw[thick, <-] (bas) -- (m) node[midway, anchor=south east]{$\rho_{\circ}$};
	\draw[thick, <-] (m) -- (k1) node[midway, anchor=north]{$\rho_1$};
	
	\node[gray] (p) at (1,1) {$?$};
	\draw[thick, gray, <-] (m) -- (p) node[midway, anchor=south east]{};
	\draw[thick, gray, <-] (p) -- (k1) node[midway, anchor=west]{$\rho_1$};
	
	\draw[thick, ->, >=stealth'] (2,-0.5) --(4,-0.5) node[midway, anchor=south] {$\mathcal{U}^{\angle}_{\rho_{\circ}, \rho_1}$};
	
	\node (p) at (8,1) {$\pi$};
	\node (bas) at (5.5,-1.5){};
	\node (k1) at (8,0) {$\kappa$};
	\draw[thick, <-] (bas) -- (p) node[midway, anchor=south east]{$\rho_{\circ}$};
	\draw[thick, <-] (p) -- (k1) node[midway, anchor=west]{$\rho_1$};
	
	\node[gray] (m) at (7,0) {};
	\draw[thick, gray, <-] (m) -- (k1) node[midway, anchor=north]{$\rho_1$};
	\end{tikzpicture}
\end{center}
Again, the specializations in the new Pfaffian Schur process are carried from the previous step  by the transition operator.

\begin{lemma} Assume that  $\mathcal{U}^{\angle}_{\rho_{\circ}, \rho_1}$ satisfies
	\begin{equation}
	\sum_{\mu} s_{\kappa/\mu}(\rho_1) \tau_{\mu}(\rho_{\circ})\mathcal{U}^{\angle}_{\rho_{\circ}, \rho_1}(\pi \vert \kappa, \mu)  = \frac{s_{\pi/\kappa}(\rho_1) \tau_{\pi}(\rho_{\circ})}{H^{\circ}(\rho_1) H(\rho_1 ; \rho_{\circ})}.
	\label{eq:preserveschur2}
	\end{equation}
	Let $\gamma'\in \Omega$ contain the vertices $(n,n)$ and  $(n+1, n)$, with the edge  $(n+1,n)\to (n,n)$ labelled by a specialization $\rho_1$ and the diagonal edge labelled by $\rho_{\circ}$ as usual. Let  $\gamma''\in \Omega$  contain the vertices $(n+1, n+1)$ and  $(n+1,n)$, with the edge $(n+1, n) \to (n+1, n+1)$ labelled by the specialization $\rho_1$ and the diagonal edge labelled by $\rho_{\circ}$.
	Then $\mathcal{U}^{\angle}(\lambda^{(n+1, n+1)} \vert \lambda^{(n+1, n)}, \lambda^{(n,n)})$ maps the Pfaffian Schur process indexed by $\gamma'$ and the set of specializations on $E(\gamma')$ to the Pfaffian Schur process indexed by $\gamma''$ and the same set of specializations.
\end{lemma}
\begin{proof}
	\eqref{eq:preserveschur2} implies that applying $\mathcal{U}^{\angle}_{\rho_{\circ}, \rho_1}$ to the Pfaffian Schur process indexed by $\gamma'$  and averaging over $ \lambda^{(n,n)} $ yields a weight proportional to
	$$s_{\lambda^{(n+1, n+1)}/\lambda^{(n+1, n)}}(\rho_1) \tau_{\lambda^{(n+1, n+1)}}(\rho_{\circ})$$
	as in the Pfaffian Schur process indexed by $\gamma''$. The normalization $H^{\circ}(\rho_1) H(\rho_1 ; \rho_{\circ})$ accounts for the fact that  $Z^{\circ}(\gamma'') = H^{\circ}(\rho_1) H(\rho_1 ; \rho_{\circ}) Z^{\circ}(\gamma')$.
\end{proof}
We choose a particular solution to \eqref{eq:preserveschur2} given by
\begin{equation}
\mathcal{U}^{\angle}_{\rho_{\circ}, \rho_1}(\pi \vert \kappa, \mu)  =\mathcal{U}^{\angle}_{\rho_{\circ}, \rho_1}(\pi \vert \kappa)  =
\frac{1}{H^o(\rho_1) H(\rho_1 ; \rho_{\circ})}
\frac{\tau_{\pi}(\rho_{\circ}) s_{\pi/\kappa}(\rho_1) }{\tau_{\kappa}(\rho_{\circ}, \rho_1)}.
\label{eq:defuangle}
\end{equation}

Let us check that \eqref{eq:preserveschur2} is satisfied. Denoting $Z^{\circ}= H^o(\rho_1) H(\rho_1 ; \rho_{\circ})$, 
\begin{align*}
\sum_{\mu} s_{\kappa/\mu}(\rho_1) \tau_{\mu}(\rho_{\circ})\mathcal{U}^{\angle}_{\rho_{\circ}, \rho_1}(\pi \vert \kappa, \mu) &=  \frac{1}{Z^{\circ}}  \sum_{\mu} s_{\kappa/\mu}(\rho_1) \tau_{\mu}(\rho_{\circ})
\frac{\tau_{\pi}(\rho_{\circ}) s_{\pi/\kappa}(\rho_1) }{\tau_{\kappa}(\rho_{\circ}, \rho_1)} \\
&=  \frac{\tau_{\pi}(\rho_{\circ}) s_{\pi/\kappa}(\rho_1)}{Z^{\circ}} \frac{\sum_{\nu' \text{ even}}\sum_{\mu} s_{\kappa/\mu}(\rho_1) s_{\mu/\nu}(\rho_{\circ})}{\tau_{\kappa}(\rho_{\circ}, \rho_1)}\\
&=\frac{\tau_{\pi}(\rho_{\circ}) s_{\pi/\kappa}(\rho_1)}{Z^{\circ}} \frac{\sum_{\nu' \text{ even}} s_{\kappa/\nu}(\rho_1, \rho_{\circ})}{\tau_{\kappa}(\rho_{\circ}, \rho_1)}\\
&=\frac{\tau_{\pi}(\rho_{\circ}) s_{\pi/\kappa}(\rho_1)}{Z^{\circ}}.\\
\end{align*}
It is not a priori obvious that $\mathcal{U}^{\angle}_{\rho_{\circ}, \rho_1}(\pi \vert \kappa)$ does define a probability distribution, but this can be checked using identities \eqref{eq:skewLittlewoodvariant} and \eqref{eq:skewCauchy}.

There may exist other solutions to \eqref{eq:preserveschur2} where  $ \mathcal{U}^{\angle}(\pi \vert \mu, \kappa) $  depends on $\mu$. We do not attempt to classify other possible choices.

\subsection{The first coordinate marginal: last passage percolation}
\label{sec:marginalLPP}
In this Section, we relate the dynamics on Pfaffian Schur process constructed in Section \ref{sec:dynamics} to half-space LPP.
\begin{definition}[Half-space geometric weight LPP]
	Let  $(a_i)_{i\geqslant 1} $ be a sequence of positive real numbers and $\big(g_{n,m})_{n\geqslant m\geqslant 0}$ be a sequence of independent geometric random variables\footnote{The geometric distribution with parameter $q\in(0,1)$, denoted $\Geom(q)$, is the probability distribution on $\Z_{\geqslant 0}$ such that if $X\sim\Geom(q)$,
		$ \ \PP(X=k) = (1-q)q^k.$} with parameter $a_n a_m$ when $n\geqslant m+1$ and with parameter $c a_n$ when  $n=m$. We define  the geometric last passage percolation time on the half-quadrant (see Figure \ref{fig:lastpassagehalfquadrant}), denoted $G(n,m)$,  by the recurrence for $n\geqslant m$,
	$$ G(n,m) =  g_{n,m} + \begin{cases}
	\max\Big\lbrace G(n-1, m)  , G(n, m-1)\Big\rbrace &\mbox{if } n\geqslant m+1, \\
	G(n,m-1) &\mbox{if }n=m,
	\end{cases}$$
	with the boundary condition $G(n,0)\equiv 0$.
	\label{def:LLPgeom}
\end{definition}
Consider integers $0<i_1 < i_2 < \dots <i_n$ and $j_1> j_2 > \dots > j_n$ such that $i_k>j_k$ for all $k$. The path $k\mapsto (i_k, j_k)$ is a down right path on the half-quadrant, it is called a space-like path in the context of interacting particle systems \cite{borodin2008large1, borodin2008large2}. We can  express the joint law of the last passage times $G(i_k, j_k)$ using a Pfaffian Schur process.

\begin{proposition}
	Let $c$ and $a_1, a_2, \dots $ be positive real numbers.
	We have the equality in law
	$$ \Big(G(i_1, j_1), \dots , G(i_n, j_n)\Big)  \overset{(d)}{=} \Big(\lambda_1^{(1)}, \dots ,\lambda_1^{(n)}\Big), $$
	where the sequence of partitions $\lambda^{(1)}, \dots ,\lambda^{(n)}$ is distributed according the Pfaffian Schur process $\PSP(\rho_{\circ}^+;\rho_{1}^+; \dots; \rho_{n-1}^+ \vert  \rho_1^-; \dots; \rho_n^-)$ with
	\begin{align*}
	\rho_{\circ}^+&= (c, a_{j_1+1}, \dots ,a_{i_1}) , \\
	\rho_k^+&= (a_{i_k+1}, \dots , a_{i_{k+1}}) \text{ for } k=1, \dots, n-1,\\
	\rho_k^-&=  (a_{j_{k+1}+1}, \dots, a_{j_k})\text{ for } k=1, \dots, n -1, \\
	\rho_n^-&=  (a_{1}, \dots, a_{j_n}).
	\end{align*}
	Equivalently, in terms of zig-zag paths,
	$$\Big(G(i_1, j_1), \dots , G(i_n, j_n)\Big)  \overset{(d)}{=} \Big(\lambda_1^{(i_1, j_1)}, \dots ,\lambda_1^{(i_n, j_n)}\Big)$$
	where the sequence of partition $\lambda^{(i_1, j_1)}, \dots ,\lambda^{(i_n, j_n)}$ is distributed according to  the Pfaffian Schur process indexed by a path $\gamma\in\Omega$ going through the points $(i_1, j_1), \dots, (i_n, j_n)$; where  vertical (resp. horizontal) edges with horizontal (resp. vertical) coordinates $i-1$ and $i$ are labelled by specializations into the single variable $a_i$, and the diagonal edge is labelled by the specialization into the variable $c$ (See Figure \ref{fig:choicespecializations}).
	\label{prop:SchurLPPcorrelations}
\end{proposition}

\begin{figure}
	\begin{center}
		\begin{tikzpicture}[scale=1]
		\tikzstyle{fleche}=[>=stealth', postaction={decorate}]
		\usetikzlibrary{arrows}
		\usetikzlibrary{shapes}
		\usetikzlibrary{shapes.multipart}
		\tikzstyle{topartition}=[thick]
		\draw[->, >=stealth', thick, gray] (0,0) -- (0,3.7) ;
		\draw[->, >=stealth', thick, gray] (0,0) -- (6.7, 0) ;
		\begin{scope}[decoration={
			markings,
			mark=at position 0.5 with {\arrow{<}}}]
		\draw[dashed, gray] (0,0) grid(6.3,3.3);
		\draw[thick] (0,0) -- (3,3);
		\draw[fleche] (3,3) -- (4,3);
		\draw[fleche] (4,3)-- (5,3);
		\draw[fleche] (5,3) -- (5,2);
		\draw[fleche] (5,2) -- (6,2);
		\draw[fleche] (6,2) -- (6,1);
		\draw[fleche] (6,1) -- (6,0);
		\fill (3,3) circle(0.10);
		\fill (4,3) circle(0.05);
		\fill (5,3) circle(0.10);
		\fill (5,2) circle(0.05);
		\fill (6,2) circle(0.10);
		\fill (6,1) circle(0.05);
		\draw (3.5, 3.2) node{{\footnotesize $a_4$}};
		\draw (4.5, 3.2) node{{\footnotesize $a_5$}};
		\draw (5.5, 2.2) node{{\footnotesize $a_6$}};
		\draw (4.7, 2.5) node{{\footnotesize $a_3$}};
		\draw (5.7, 1.2) node{{\footnotesize $a_2$}};
		\draw (5.7, 0.2) node{{\footnotesize $a_1$}};
		\draw (1.3, 1.7) node{$c$};
		\draw (3.2, 3.5) node{$\lambda^{(3,3)}$};
		\draw (5.2, 3.5) node{$\lambda^{(5,3)}$};
		\draw (6.7, 2) node{$\lambda^{(6,2)}$};
		\end{scope}

		\begin{scope}[xshift=8cm, yshift=1cm]
		
		\node (emptyleft) at (0,0) {$\varnothing$};
		\node (lambda1) at (1,1.5) {$\lambda^{(1)}$};
		\node (mu1) at (2,0) {$\mu^{(1)}$};
		\node (lambda2) at (3,1.5) {$\lambda^{(2)}$};
		\node (mu2) at (4,0) {$\mu^{(2)}$};
		\node (lambda3) at (5,1.5) {$\lambda^{(3)} $};
		\node (emptyright) at (6,0) {$ \varnothing$};
		\draw[topartition] (emptyleft) -- (lambda1) node[midway, anchor=south, rotate=56]{{\footnotesize  $c$ }};
		\draw[topartition] (lambda1) -- (mu1) node[midway, anchor=south, rotate=-56]{{\footnotesize  $\varnothing$}};
		\draw[topartition] (mu1) -- (lambda2) node[midway, anchor=south , rotate=56]{{\footnotesize  $a_4, a_5$ }};
		\draw[topartition] (lambda2) -- (mu2) node[midway, anchor=south, rotate=-56]{{\footnotesize  $a_3$}};
		\draw[topartition] (mu2) -- (lambda3) node[midway, anchor=south, rotate=56]{{\footnotesize  $a_6$ }};
		\draw[topartition] (lambda3) -- (emptyright) node[midway, anchor=south, rotate=-56]{{\footnotesize  $a_2, a_1$}};
		\end{scope}
		\end{tikzpicture}
	\end{center}
	\caption{Equivalence of the two formulations in Proposition \ref{prop:SchurLPPcorrelations}. Left: The components $\lambda_1^{(i_1, j_1)}, \lambda_1^{(i_2, j_2)} ,\lambda_1^{(i_3, j_3)}$  of the Pfaffian Schur process considered in Proposition \ref{prop:SchurLPPcorrelations}, for $(i_1, j_1) = (3,3)$, $(i_2, j_2) = (5,3)$ and $(i_3, j_3) = (6,2)$.
		Right: The corresponding diagram in the setting of Definition \ref{def:PSP}. There is equality in law between $\lambda_1^{(3,3)}, \lambda_1^{(5, 3)} ,\lambda_1^{(6, 2)}$ on the left, and $\lambda^{(1)},\lambda^{(2)},\lambda^{(3)}$ on the right.}
	\label{fig:choicespecializations}
\end{figure}
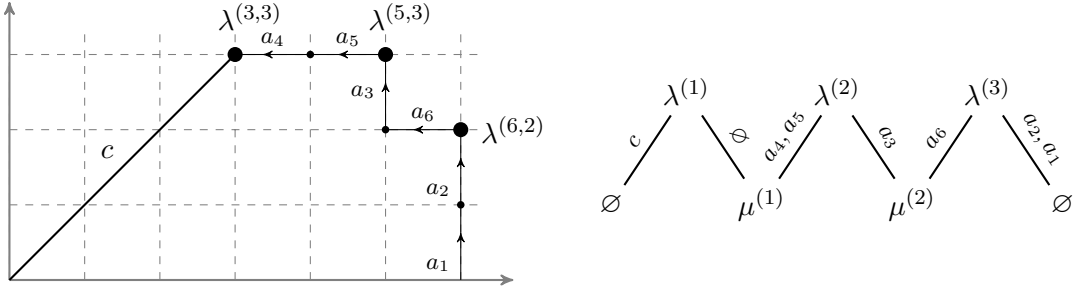
\begin{remark}
	Note that it is also true that $\mu_1^{(1)}$ has the same law as $G(i_1, j_2)$, $\mu_1^{(2)}$ has the same law as $G(i_2, j_3)$, etc. but we do not use this fact.
\end{remark} 
\begin{proof}
	We begin with two lemmas explaining the action of the transition operators $\mathcal{U}^{\llcorner}$ and $\mathcal{U}^{\angle}$ on the first coordinates of the partitions.
	\begin{lemma}[\cite{borodin2008anisotropic}]
		\label{lem:interpretationgeom1}
		Consider two specializations determined by  single variables $a, b>0$ and the transition operator $\mathcal{U}^{\llcorner}_{a,b}(\pi \vert \nu, \kappa)$ defined in \eqref{eq:defucorner} that updates randomly  the partition $\mu$ to become $\pi$ given partitions $\nu, \kappa$. 
		The first coordinate of the partition $\pi$ is such that
		$$ \pi_1  = \max(\nu_1, \kappa_1) + \Geom(ab).$$
	\end{lemma}
	\begin{proof}
		This Lemma is a particular case of two-dimensional push-block dynamics defined in \cite[Section 2.6, Ex. (3)]{borodin2008anisotropic}. For two partitions $\lambda, \mu$, we have that
		$ s_{\lambda/\mu}(a) = \mathds{1}_{\mu \prec \lambda} a^{\vert \lambda\vert - \vert \mu\vert}.$
		Hence we have that 
		$$\mathcal{U}^{\llcorner}(\pi \vert \nu,  \kappa) =  \mathds{1}_{\nu\prec \pi}\mathds{1}_{\kappa \prec \pi} a^{\vert \pi\vert - \vert \nu\vert} b^{\vert \pi\vert - \vert \kappa\vert}.$$
		Summing out $\pi_2, \pi_3, \dots$ and using the fact that $\mathcal{U}^{\llcorner}$ is a stochastic transition kernel, we find that the distribution of $\pi_1$ is given by
		$$ \mathcal{U}^{\llcorner}(\pi_1 \vert \nu_1, \kappa_1)  = (1-ab) \cdot \mathds{1}_{\pi_1 \geqslant \max(\nu_1, \kappa_1)} \big( a b\big)^{\pi_1 - \max(\nu_1, \kappa_1)}.$$
	\end{proof}
	\begin{lemma}
		\label{lem:interpretationgeom2}
		Consider two specializations determined by  single variables $a, c>0$ and the transition operator $\mathcal{U}^{\angle}_{c, a}(\pi \vert \kappa)$ defined in \eqref{eq:defuangle} that updates randomly the partition $\mu$ to become $\pi$, given the partition $\kappa$.
		The first coordinate of the partition $\pi$ is such that
		$$ \pi_1  = \kappa_1 + \Geom(ac).$$
	\end{lemma}
	\begin{proof}
		Similar to the proof of Lemma \ref{lem:interpretationgeom1}, 
		we have that \\
		$\tau_{\pi}(c) = \allowbreak c^{\pi_1 - \pi_2}c^{\pi_3 - \pi_4}\cdots$. 
		Hence, the distribution of $\pi_1$ under $ \mathcal{U}^{\angle}(\pi \vert \kappa) $ is proportional to $\mathds{1}_{\pi_1 \geqslant \kappa_1}(ac)^{\pi_1-\kappa_1}$.
	\end{proof}
	Fix some path $\gamma\in \Omega$. Let us sequentially grow the Pfaffian Schur process according to the procedure defined in Section \ref{sec:dynamics}. We need to specify the diagonal specialization and  the specializations we start from on the horizontal axis. Assume that the diagonal specialization is the single variable specialization into $c>0$, and the specialization on the edge $(i, 0)\to (i-1, 0)$ is the single variable specialization $a_i$. Given how the specializations are transported on the edges of the lattice when we apply $\mathcal{U}^{\llcorner}$ and  $\mathcal{U}^{\angle}$, this choice of initial specializations implies that for all edges of $\gamma$, vertical (resp. horizontal) edges with horizontal (resp. vertical) coordinates $i-1$ and $i$ are labelled by specializations into the single variable $a_i$,  and the diagonal edge is labelled by the specialization into the variable $c$ (See Figure \ref{fig:choicespecializations}).
	
	It follows from Lemmas \ref{lem:interpretationgeom1} and \ref{lem:interpretationgeom2} that the first coordinates of the partitions $\lambda^{v}$ in the sequence of Pfaffian Schur processes are such that for $i>j$
	$$ \lambda^{(i,j)}_1 =  \max\Big(\lambda^{(i-1,j)}_1, \lambda^{(i,j-1)}_1 \Big) + \Geom(a_i a_j),$$
	and
	$$  \lambda^{(i,i)}_1 = \lambda_1^{(i, i-1)} + \Geom(a_i c).$$
	Hence,  for any collection of vertices $v_1, \dots, v_k$ along $\gamma\in \Omega$, we have
	$$ (\lambda^{v_1}_1, \dots, \lambda^{v_1}_1)  \overset{(d)}{=} \big(G(v_1), \dots, G(v_1)\big),$$
	where $G(v)$ is the last passage time to $v$ as in Definition \ref{def:LLPgeom}.
\end{proof}

\begin{remark}
	\label{rem:symmetricRSK}
	Consider a symmetric matrix $S=(g_{ij})_{i,j=1}^n$ of size $n$ whose entries $g_{i,j}$ are such that $g_{i,j}\sim \Geom(a_ia_j)$ for $i\neq j$ and $ g_{i,i}\sim \Geom(c a_i)$. The image of $S$ under RSK correspondence is a Young tableau\footnote{Insertion tableau and Recording tableau are the same because $S$ is symmetric.} whose shape $\lambda$ is distributed according to the Pfaffian Schur measure \\
	$\PSM(c \vert  a_1, \dots, a_n).$
	This can be deduced from \cite[(7.47)]{baik2001algebraic}, see also Eq. (10.151) in \cite{forrester2010log}. As a consequence, this provides a very short proof that $G(n,n)$ has the law of the $\lambda_1$ marginal of $\PSM(c \vert  a_1, \dots, a_n)$. Moreover, this provides an interpretation of the other coordinates of the partition: $\lambda_1+\dots + \lambda_i$ is the maximum over $i$-tuples of non intersecting up-right paths between the points $(1,1), \dots, (1, i)$ and $(n,n), \dots, (n, n-i+1)$ of the sum of weights along these $i$ paths.  Note that the RSK insertion procedure also defines dynamics on (ascending) sequences of interlacing partitions, which are different from the push-block dynamics that we consider.
\end{remark}

\begin{remark}
	Taking a Poisson limit of geometric LPP, such that the points on the diagonal also limits to a one dimensional Poisson process, yields the Poissonized random involution longest increasing subsequence problem considered in \cite{baik2001algebraic, baik2001asymptotics} (equivalently half-space continuous PNG with a source at the origin \cite{sasamoto2004fluctuations}). The limit of Proposition \ref{prop:SchurLPPcorrelations} shows that this corresponds to the PLancherel specialization of the Pfaffian Schur process. 
\end{remark}

\section{Fredholm Pfaffian formulas for $k$-point distributions}
\label{sec:kpointdistribution}

It is shown in \cite[Theorem 3.3]{borodin2005eynard} that if $\bar \lambda, \bar\mu$ 
is distributed according to a Pfaffian Schur process (Definition \ref{def:PSP}), then
$$ \mathcal{L}(\bar\lambda) :=\Big\lbrace \big(1, \lambda_i^{(1)} -i\big)\Big\rbrace_{i\geqslant 1} \cup \dots \cup \Big\lbrace \big(k, \lambda_i^{(k)} -i\big)\Big\rbrace_{i\geqslant 1}$$
is a  Pfaffian point process. In particular, for any $S\geqslant 1$ and pairwise distinct points $(i_s, u_s)$, $1\leqslant s
\leqslant S$ of $\lbrace 1, \dots, k\rbrace \times \Z$ we have the formal\footnote{In \cite{borodin2005eynard} this is written as a formal identity in variables of the symmetric functions, but it can be specialized (as in Section \ref{sec:kernelgeom}) to numeric identities in the cases we consider.} series identity
\begin{equation}
\PP\big( \lbrace (i_1,u_1), \dots , (i_S, u_S)\rbrace  \subset \mathcal{L}(\lambda) \big)  = \Pf\big[K(i_s, u_s; i_t, u_t)\big]_{s,t=1}^S,
\label{eq:correlationfunctionsgeometric}
\end{equation}
where $K(i, u; j, v)$ is a $2 \times 2$ matrix-valued kernel
$$ K= \begin{pmatrix}
K_{11} & K_{12} \\
K_{21} & K_{22}
\end{pmatrix}   $$
with $K_{11}$, $K_{12} = -\big(K_{21}\big)^T$, and $K_{22}$  given by explicit formulas (see  \cite[Theorem 3.3]{borodin2005eynard}).
In Section \ref{sec:PPP}, we first review the formalism of Pfaffian point processes.
In Sections \ref{sec:kernelgeom} and \ref{sec:kernelexp}, we will explain how \cite[Theorem 3.3]{borodin2005eynard} applies to last passage percolation in a half-quadrant with geometric or exponential weights.

\subsection{Pfaffian point processes}
\label{sec:PPP}
We consider presently only simple\footnote{i.e. without multiplicities} point processes and introduce the formalism of Pfaffian point processes.
Let us first  start with the case of a finite state space $\mathbb{X}$. A random (simple) point process on $\mathbb{X}$ is a probability measure on subsets $X$ of $\mathbb{X}$. The correlation functions of the point process are defined by
$$ \rho(Y) = \PP\Big( X\subset \mathbb{X} \text{ such that } Y\subset X\Big), \text{ for }Y\subset \mathbb{X}. $$
This definition implies that for a domain $I_1\times I_2\times\dots \times I_k\subset\mathbb{X}^k$,
\begin{multline}
\sum_{\underset{distinct}{(x_1, \dots, x_k)\in I_1\times\dots \times I_k}}\rho(x_1, \dots, x_k) =\\  \EE\Big[\# \lbrace  k\text{-tuples of distinct points } x_1\in X\cap I_1, \dots , x_k\in X\cap I_k  \rbrace \Big].
\label{eq:momentfactorialmeasurediscrete}
\end{multline}
Such a point process is called Pfaffian if there exists a $2\times 2$ matrix valued  $\vert \mathbb{X} \vert \times \vert \mathbb{X} \vert$ skew-symmetric matrix $K$ with rows and columns parametrized by points of $\mathbb{X}$, such that the correlation functions of the random point process are
\begin{equation}
\rho(Y)  = \Pf\big(K_Y\big), \text{ for any }Y\subset \mathbb{X},
\label{eq:pfaffiankernel}
\end{equation}
where $K_Y:=\big(K(y_i, y_j) \big)_{i,j=1}^k$ for $Y=\lbrace y_1, \dots, y_k \rbrace$, and $\Pf$ is the Pfaffian (see \eqref{def:pfaffian}).

More generally, let $(\mathbb{X}, \mu)$ be a measure space, and $\PP$ (and $\EE$) be a probability measure (and expectation) on the set  $ \Conf(\mathbb{X})$  of countable and locally finite subsets (configurations) $X\subset\mathbb{X}$. We define the $k$th factorial moment  measure on $\mathbb{X}^k$ by
$$ B_1\times \dots \times B_k \rightarrow \EE\Big[\# \lbrace  k\text{-tuples of distinct points } x_1\in X\cap B_1, \dots , x_k\in X\cap B_k  \rbrace \Big],$$
where $B_1, \dots, B_k\subset\mathbb{X}$ are Borel subsets. Assuming it is absolutely continuous with respect to $\mu^{\otimes k}$, the $k$-point correlation function $\rho(x_1, \dots, x_k)$ is the Radon-Nykodym derivative of the $k$-moment factorial measure with respect to $\mu^{\otimes k}$. As in the discrete setting, $\PP$ is a Pfaffian point process if there exists a kernel
$$ K : \mathbb{X}\times \mathbb{X} \to \mathrm{Skew}_2(\R), $$
where $\mathrm{Skew}_2(\R)$ is the set of skew-symmetric $2\times 2$ matrices,
such that for all $Y\subset\mathbb{X}$,
\begin{equation}
\rho(Y) = \Pf(K_Y).
\label{eq:pfaffiankernelcontinuous}
\end{equation}
For measurable functions $f : \mathbb{X}\to \C$, the definition of correlation functions is equivalent to
\begin{equation}
\EE\Big[\sum_{\underset{distinct}{(x_1, \dots, x_k)\in X^k}} f(x_1)\dots f(x_k)\Big] = \int_{ \mathbb{X}^k} \rho(x_1, \dots, x_k)f(x_1)\dots f(x_k)\ \mathrm{d}\mu^{\otimes k},
\label{eq:propertycorrelations}
\end{equation}
and thus  \eqref{eq:pfaffiankernelcontinuous} implies that (\cite[Theorem 8.2]{rains2000correlation}),
\begin{equation}
\EE\Big[ \prod_{x\in X} (1+f(x))\Big] = \Pf\Big(J+K\Big)_{\mathbb{L}^2(\mathbb{X}, f\mu)},
\label{eq:PfaffianObservables}
\end{equation}
whenever both sides are absolutely convergent. The R.H.S of \eqref{eq:PfaffianObservables} is a Fredholm Pfaffian, defined in Definition \ref{def:FredholmPfaffian}.  We refer to \cite[Section 8]{rains2000correlation} and \cite[Appendix B]{ortmann2015pfaffian} for properties of Fredholm Pfaffians. In particular, \eqref{eq:PfaffianObservables} implies that the gap probabilities are given by the Pfaffian formula
\begin{equation}
\PP\big(\text{no point lie in }\mathbb{Y} \big)  = \Pf(J-K)_{\mathbb{L}^2(\mathbb{Y}, \mu)} \text{ for }\mathbb{Y}\subset\mathbb{X}.
\label{eq:gapprobability}
\end{equation}

\subsection{Half-space geometric weight LPP}
\label{sec:kernelgeom}
Let $a_1, a_2, \dots\in(0,1)$, and $c>0$ such that for all $i$, $ca_i<1$; and let $0<n_1 < n_2 < \dots <n_k$ and $m_1> m_2 > \dots > m_k$ be sequences of integers such that $n_i>m_i$ for all $i$.
Consider the Pfaffian Schur process indexed by the unique path in $\Omega$  going through the points
$$(n_k, 0), (n_k, m_k), (n_{k-1}, m_k), (n_{k-1}, m_{k-1}) , \dots , (n_1, m_1), (m_1, m_1),(0,0),$$
where a horizontal edge $(i-1, j)\to(i,j)$ (resp. a vertical edge $(i, j-1) \to (i, j)$) is labelled by the specialization into the single variable $a_i$ (resp. $a_j$), and the diagonal edge is labelled by $c$. In other words, we are in the same setting as in Section \ref{sec:marginalLPP}.
In this case, it follows from \cite[Theorem 3.3]{borodin2005eynard} that  the correlation kernel of $\mathcal{L}(\bar\lambda)$, that we will denote by $K^{\rm geo}$, indexed by couples of points in $\lbrace 1, \dots, k\rbrace \times \Z$,  is given by the following formulas. It is convenient to introduce the notations 
\begin{align*}
h^{\rm geo}_{11}(z,w)&:= \prod_{\ell=1}^{n_i} \frac{z-a_{\ell}}{z} \prod_{\ell=1}^{n_j}\frac{w-a_{\ell}}{w} \prod_{\ell=1}^{m_i} \frac{1}{1-a_{\ell}z} \prod_{\ell=1}^{m_j} \frac{1}{1-a_{\ell} w},\\
h^{\rm geo}_{12}(z,w)&:= \prod_{\ell=1}^{n_i}  \frac{z-a_{\ell}}{z}  \prod_{\ell=1}^{n_j} \frac{1}{1-a_{\ell} w}    \prod_{\ell=1}^{m_i} \frac{1}{1-a_{\ell}z } \prod_{\ell=1}^{m_j} \frac{w-a_{\ell}}{w},\\
h^{\rm geo}_{22}(z,w)&:= \prod_{\ell=1}^{n_i}\frac{1}{1-a_{\ell}z} \prod_{\ell=1}^{n_j} \frac{1}{1-a_{\ell}w} \prod_{\ell=1}^{m_i}  \frac{z-a_{\ell}}{z} \prod_{\ell=1}^{m_j} \frac{w-a_{\ell}}{w}.
\end{align*}
Then the kernel is given by 
\begin{equation}
K^{\rm geo}_{11}(i, u;j,v) =  \iint \frac{(z-w)h^{\rm geo}_{11}(z,w) }{(z^2-1)(w^2-1)(zw-1)}
\frac{z-c}{z} \frac{w-c}{w}  \frac{\dd z}{z^u}\frac{\dd w}{w^v},
\label{eq:K11geoinitial}
\end{equation}
where the contours are positively oriented circles around $0$ of such that for all $i$, $1<\vert z\vert , \vert w\vert <1/a_i$;
\begin{equation}
K^{\rm geo}_{12}(i, u;j,v)   =  \iint \frac{(z-w)h^{\rm geo}_{12}(z,w)}{(z^2-1)w(zw-1)} 
\frac{z-c}{z(1-cw)} \frac{\dd z}{z^u}\frac{\dd w}{w^v}= -K_{21}^{\rm geo}(j,v ; i,u),
\label{eq:K12geoinitial}
\end{equation}
where the contours are positively-oriented circles around $0$ of radius such that for all $i$,  $1<\vert z\vert  <1/a_i$, $\vert w \vert <1/c, 1/a_i$ and if $i\geqslant j$ then $\vert zw\vert >1$,
while  if $i< j$ then $\vert zw\vert <1$; 
\begin{equation}
K^{\rm geo}_{22}(i, u;j,v) =  \iint \frac{(z-w)h^{\rm geo}_{22}(z,w)}{zw(zw-1)} 
\frac{1}{1-cz}\frac{1}{1-cw}\frac{\dd z}{z^u}\frac{\dd w}{w^v},
\label{eq:K22geoinitial}
\end{equation}
where the contours are positively oriented circles around $0$ such that $\vert zw\vert <1$ and $\vert z\vert , \vert w\vert <1/c$, and for all $i$,  $\vert z\vert, \vert w\vert <1/a_i$.

The conditions on the contours ensure that \eqref{eq:correlationfunctionsgeometric} is not only a formal series identity but a numeric equality: when $c<1$, all sums in the proof of \cite[Theorem 3.3]{borodin2005eynard} are actually absolutely convergent, and the formulas can then be extended to any $c$  such that $c a_i<1$ by analytic continuation. The correlation functions of the model are clearly analytic in $c$, and the Pfaffians of integrals on finite contours are analytic as well.  Using Proposition \ref{prop:SchurLPPcorrelations}, \eqref{eq:gapprobability} implies the following.
\begin{remark}
	In the statement of Theorem 3.3 in \cite{borodin2005eynard}, the factor $(zw-1)$ which appear in the denominator of the integrand in $K_{22}$ is replaced by $(1-zw)$. This seems to be a a typo, as the proof of Theorem 3.3 in \cite{borodin2005eynard} suggests that the correct factor should be $(zw-1)$, and \cite{ghosal2017correlation} recently confirmed the correct sign. 
\end{remark}
\begin{proposition}
	For any  $g_1, \dots, g_k \in \Z_{\geqslant 0}$,
	\begin{equation}
	\mathbb{P}\left( \bigcap_{i=k}^k \big \lbrace G(n_i, m_i)\leqslant  g_i \big\rbrace \right)  = \Pf\big(J-K^{\rm geo}\big)_{\mathbb{L}^2\big(\tilde{\mathbb{D}}_k(g_1, \dots, g_k) \big)},
	\label{eq:pdfG}
	\end{equation}
	where
	$$ \tilde{\mathbb{D}}_k(g_1, \dots, g_k) = \lbrace (i,x)\in \lbrace 1, \dots, k\rbrace \times\Z: x \geqslant g_i\rbrace,$$
	and $J$ is the matrix kernel
	\begin{equation}
	J(i, u ; j, v) = \mathds{1}_{(i, u)=(j,v)}\left(\begin{matrix}
	0 &1 \\
	-1 & 0
	\end{matrix} \right).
	\label{eq:kernelJ}
	\end{equation}
	\label{prop:kernelgeom}
\end{proposition}
By Definition \ref{def:FredholmPfaffian}, $\Pf\big(J-K^{\rm geo}\big)$ is defined by the expansion
\begin{multline}
\Pf\big(J-K^{\rm geo}\big)_{\mathbb{L}^2\big(\tilde{\mathbb{D}}_k(g_1, \dots, g_k) \big)} = \\ 1+\sum_{n=1}^{\infty} \frac{(-1)^n}{n!}  \sum_{i_1, \dots, i_n=1 }^k
\sum_{x_1=g_{i_1}}^{\infty} \dots \sum_{x_n=g_{i_n}}^{\infty}
\Pf\big(K^{\rm geo}(i_s, x_s ; i_t, x_t )\big)_{s, t=1}^n.
\label{eq:defPfaffian1}
\end{multline}

\begin{remark}
	In the special case $k=1$ and $n_1=m_1=n$, the Pfaffian representation of the correlation functions for $\lbrace \lambda^{(n, n)}_i -i\rbrace$ was given in  \cite[Corollary 4.3]{rains2000correlation} and subsequently used in \cite{forrester2006correlation} to study involutions with a fixed number of fixed points.
\end{remark}

\subsection{Deformation of contours}
\label{sec:deformedcontours}
We will later need (in order to deduce the correlation kernel of the exponential model from the geometric case) to set all $a_i$ to $\sqrt{q}$ and let $q$ go to $1$.  Before taking this limit,   we need to deform some of the contours used in the formulas for $K^{\rm geo}$ and compute the residues involved. It will be much more convenient\footnote{If it was not the case, we would find issues related to the choice of limiting contours when performing the asymptotic analysis.} for the following asymptotic analysis if all integration contours in the definition of $K^{\rm geo}$ are circles such that $\vert zw \vert >1$. We do not need to transform the expressions for $K^{\rm geo}_{11}$ and for $K^{\rm geo}_{12}$ when $i\geqslant j$. When $c<1$, these residues correspond to terms that arise in the Pfaffian version of Eynard-Mehta's theorem (\cite[Theorem 1.9]{borodin2005eynard}), as in the proof of Theorem 3.3 in \cite{borodin2005eynard}. When $c\geqslant 1$, we get more residues which were not present in the proof of Theorem 3.3 in \cite{borodin2005eynard} (recall that we make an analytic continuation to obtain the correlation kernel when $c\geqslant 1$). These extra residues are the signature of the occurrence of a phase transition when $c$ varies, as discovered in \cite{baik2001asymptotics}.

\textbf{Case $c<1$:}  For $K^{\rm geo}_{22}$ we write
\begin{equation}
K^{\rm geo}_{22}(i, u;j,v) = I^{\rm geo}_{22}(i, u;j,v)+ R^{\rm geo}_{22}(i, u;j,v)
\label{eq:K22geodeformed}
\end{equation}
where $I^{\rm geo}_{22}(i, u;j,v)$ is the same  as \eqref{eq:K22geoinitial}
but the contours are now positively oriented circles around $0$ such that $1<\vert z\vert , \vert w\vert <\min(1/c, 1/a_1, 1/a_2, \dots)$; and
\begin{equation}
R^{\rm geo}_{22}(i, u;j,v) =  \int \frac{1-z^2}{z^2}  \frac{1}{1-cz}\frac{1}{1-c/z}    \frac{h^{\rm geo}_{22}(z,1/z) }{z^{u-v}}\dd z.
\end{equation}
To see the equivalence between \eqref{eq:K22geoinitial} and \eqref{eq:K22geodeformed}, deform the $w$ contour so that the radius exceeds $z^{-1}$. Of course, we have  to substract the residue for $w=1/z$, which is expressed as an integral in $z$, which equals $-R^{\rm geo}_{22}(i, u;j,v)$. Finally, we can freely move the $z$ contour in the two-fold integral (only after having moved the $w$ contour) and get   \eqref{eq:K22geodeformed}.

In the case where $i<j$ we need to also decompose $K_{12}$, because we pick a residue at $w=z^{-1}$ when moving the $w$ contour. We  rewrite $K^{\rm geo}_{12}$ as
\begin{equation}
K^{\rm geo}_{12}(i, u;j,v) = I^{\rm geo}_{12}(i, u;j,v)+ R^{\rm geo}_{12}(i, u;j,v)
\label{eq:decompositionK12geo}
\end{equation}
where $I^{\rm geo}_{12}(i, u;j,v)$ is the same as  \eqref{eq:K12geoinitial}
but the contours are now positively-oriented circles around $0$ of radius such that  $1<\vert z\vert , \vert w \vert <\min(1/c, 1/a_1, 1/a_2, \dots)$ and
\begin{equation}
R^{\rm geo}_{12}(i, u;j,v) = - \int
\frac{h^{\rm geo}_{12}(z,1/z)  }{z^{u-v+1}}\dd z.
\label{eq:R12geosimple}
\end{equation}

\textbf{Case $c>1$:} In that case, we need to take into account the residues at the poles of $1/(1-cz)$ and $1/(1-cw)$ in $K^{\rm geo}_{12}$ and $K^{\rm geo}_{22}$. When deforming the $w$ contour in \eqref{eq:K22geoinitial}, we first encounter a pole at $1/c$ and then a pole at $1/z$. Each of these residues can be written as an integral in the variable $z$, where we may again deform the contour picking a residue at $z=1/c$ when necessary. Then we deform the contour for $z$ in the two-fold integral and pick a residue at $z=1/c$, which is expressed as an integral in $w$. We find 
\begin{equation}
K^{\rm geo}_{22}(i, u;j,v) = I^{\rm geo}_{22}(i, u;j,v)+ \hat{R}^{\rm geo}_{22}(i, u;j,v)
\label{eq:K22geodeformed2}
\end{equation}
where $I^{\rm geo}_{22}$ is as in the $c<1$ case and 
\begin{multline}
\hat{R}^{\rm geo}_{22}(i, u;j,v) =  -c^v \int \frac{h^{\rm geo}_{22}(z,1/c)}{z^{u+1}(z-c)}\dd z
+ c^u \int \frac{h^{\rm geo}_{22}(w,1/c) }{w^{v+1}(w-c)}\dd w \\
+  \int \frac{1-z^2}{z^2}\frac{h^{\rm geo}_{22}(z,1/z)}{(1-cz)(1-c/z)}\frac{\dd z}{z^{u-v}}
+c^{u-v-1} h^{\rm geo}_{22}(1/c, c).
\label{eq:R22geocomplicated}
\end{multline}
Because of the particular sequence of contour deformations that we have chosen, the first integral in \eqref{eq:R22geocomplicated} is such that $c$ is outside the contour, while $c$ is inside the contour in the second and third integral in \eqref{eq:R22geocomplicated}. We get a formula  where the antisymmetry is  apparent  by deforming again the $z$ contour and writing  $\hat{R}^{\rm geo}_{22}$ as
\begin{multline}
\hat{R}^{\rm geo}_{22}(i, u;j,v) =  -c^v \int \frac{h^{\rm geo}_{22}(z,1/c)}{z^{u+1}(z-c)}\dd z
+ c^u\int \frac{h^{\rm geo}_{22}(w,1/c) }{w^{v+1}(w-c)}\dd w \\
+ \int \frac{1-z^2}{z^2}\frac{h^{\rm geo}_{22}(z,1/z)}{(1-cz)(1-c/z)}\frac{\dd z}{z^{u-v}}
-c^{u-v-1} h^{\rm geo}_{22}(1/c, c)
+c^{v-u-1} h^{\rm geo}_{22}(c, 1/c).
\label{eq:R22geobetter}
\end{multline}
where the contours are circles with radius between $c$ and the $1/a_{\ell}$'s in the first and second integral, and radius between $1/c$ and $c$ in the third integral.

For $K^{\rm geo}_{12}$, in the case where $i<j$, we similarly rewrite
\begin{equation*}
K^{\rm geo}_{12}(i, u;j,v) = I^{\rm geo}_{12}(i, u;j,v)+ \hat{R}^{\rm geo}_{12}(i, u;j,v)
\end{equation*}
where $I^{\rm geo}_{12}$ is as in the $c<1$ case and 
\begin{equation}
\hat{R}^{\rm geo}_{12}(i, u;j,v) =  - \int\frac{h^{\rm geo}_{12}(z, 1/z)}{z^{u-v+1}}\dd z
- c^{w} \int \frac{1-cz}{(z^2-1)z} h^{\rm geo}_{12}(z, 1/c) \frac{\dd z}{z^u},
\label{eq:R12geocasc>1}
\end{equation}
and the contours are circles around $0$ such that $1<\vert z\vert <\min(1/a_1, 1/a_2,  \dots)$.

\textbf{Case $c=1$:} When deforming contours as previously, the sequence of residues encountered is slightly different than in the case $c>1$, and we get
\begin{equation}
K^{\rm geo}_{22}(i, u;j,v) = I^{\rm geo}_{22}(i, u;j,v)+ \bar{R}^{\rm geo}_{22}(i, u;j,v)
\label{eq:K22geodeformed3}
\end{equation}
where $I^{\rm geo}_{22}$ is as in the $c<1$ case and 
\begin{multline}
\bar{R}^{\rm geo}_{22}(i, u;j,v) =  \frac{-1}{2\I\pi} \int \frac{h^{\rm geo}_{22}(z,1)}{z^{u+1}(z-1)}\dd z
+ \frac{1}{2\I\pi} \int \frac{h^{\rm geo}_{22}(w,1) }{w^{v+1}(w-1)}\dd w \\
- \frac{1}{2\I\pi} \int \frac{1+z}{z^2(1-z)}\frac{h^{\rm geo}_{22}(z,1/z)}{z^{u-v}}\dd z
- h^{\rm geo}_{22}(1, 1).
\end{multline}
where all contours are circles such that    $1<\vert z\vert , \vert w\vert < \min(1/a_1, 1/a_2,  \dots) $.
In the case where $i<j$,
\begin{equation*}
K^{\rm geo}_{12}(i, u;j,v) = I^{\rm geo}_{12}(i, u;j,v)+ \bar{R}^{\rm geo}_{12}(i, u;j,v)
\end{equation*}
where  $I^{\rm geo}_{12}$ is as in the $c<1$ case and 
\begin{equation*}
\bar{R}^{\rm geo}_{12}(i, u;j,v) = -\frac{1}{2\I\pi} \int\frac{h^{\rm geo}_{12}(z, 1/z)}{z^{u-v+1}}\dd z
+ \frac{1}{2\I\pi} \int \frac{1}{z(1+z)} \frac{h^{\rm geo}_{12}(z, 1)}{z^u} \dd z,
\end{equation*}
and the contours are circles around $0$ such that $\vert z\vert <\min(1/a_1, 1/a_2,  \dots)$.

\subsection{Half-space exponential weight LPP}
\label{sec:kernelexp}
The following lemma is a direct consequence of the geometric to exponential weak convergence. 
\begin{lemma}
	Set all parameters $a_i \equiv \sqrt{q}$ and $c=\sqrt{q} \big( 1+(\alpha-1)(q-1)\big)$. Then, for any sequences of integers $n_1, \dots, n_k$ and $m_1, \dots, m_k$ such that $n_i\geqslant m_i$, we have the weak convergence as $q\to 1$,
	$$ \big\lbrace (1-q)G(n_i,m_i) \big\rbrace_{i=1}^k\Longrightarrow \big\lbrace H(n_i,m_i) \big\rbrace_{i=1}^k. $$
	(Recall $H(n,m)$ from  Definition \ref{def:LPPexp} and $G(n,m)$ from  Definition \ref{def:LLPgeom}.)
	\label{prop:limitgeomtoexp}
\end{lemma}
We will use this lemma to deduce the correlation functions for the exponential model. 
Define the kernel $K^{\rm exp}:(\lbrace 1, \dots, k\rbrace \times \R)^2\to \skeww(\R)$ by
$$ K^{\rm exp}(i,x;j,y) = I^{\rm exp}(i,x;j,y) + \begin{cases} R^{\rm exp}(i,x;j,y) & \mbox{when }\alpha>1/2,\\
\hat{R}^{\rm exp}(i,x;j,y) & \mbox{when }\alpha<1/2,\\
\bar{R}^{\rm exp}(i,x;j,y) & \mbox{when }\alpha=1/2.\\  \end{cases}$$
Recalling Definition \ref{def:basicrays} for integration contours, we define $I^{\rm exp}$ by:
\begin{equation}
I^{\rm exp}_{11}(i, x; j, y) :=  \int_{\mathcal{C}_{1/4}^{\pi/3}}\dd z\int_{\mathcal{C}_{1/4}^{\pi/3}}\dd w \frac{(z-w)e^{-xz-yw}}{4zw(z+w)} 
\frac{(1+2z)^{n_i}(1+2w)^{n_j}}{(1-2z)^{m_i}(1-2w)^{m_j}} (2z+2\alpha -1)(2w+2\alpha -1);
\end{equation}
\begin{equation}
I^{\rm exp}_{12}(i, x; j, y) :=  \int_{\mathcal{C}_{a_z}^{\pi/3}}\dd z\int_{\mathcal{C}_{a_w}^{\pi/3}} \dd w \frac{(z-w)e^{-xz-yw}}{2z(z+w)} 
\frac{(1+2z)^{n_i}}{(1-2w)^{n_j}}\frac{(1+2w)^{m_j}}{(1-2 z)^{m_i}}  \frac{2\alpha -1 + 2z}{2\alpha -1-2w} ,
\label{eq:I12exp}
\end{equation}
where in the definition of the contours $\mathcal{C}_{a_z}^{\pi/3}$ and $ \mathcal{C}_{a_w}^{\pi/3} $, the real  constants $a_z, a_w$ are chosen so that $0<a_z<1/2$, $a_z+a_w >0$  and $a_w<(2\alpha-1)/2$;
\begin{equation}
I^{\rm exp}_{22}(i, x; j, y) :=    \int_{\mathcal{C}_{b_z}^{\pi/3}}\dd z\int_{\mathcal{C}_{b_w}^{\pi/3}}\dd w \frac{(z-w)e^{-xz-yw}}{z+w} 
\frac{(1+2z)^{m_i}(1+2w)^{m_j}}{(1-2z)^{n_i}(1-2w)^{n_j}} \frac{1}{2\alpha -1 - 2z}\frac{1}{2\alpha -1 - 2w},
\label{eq:I22exp}
\end{equation}
where in the definition of the contours $\mathcal{C}_{b_z}^{\pi/3}$ and $ \mathcal{C}_{b_w}^{\pi/3} $, the real constants $b_z, b_w$ are chosen so that  $0<b_z, b_w<(2\alpha-1)/2$ when $\alpha>1/2$, while we impose only $b_z, b_w>0$ when $\alpha\leqslant 1/2$.

We set $R_{11}^{\rm exp}(i,x ; j,y) = 0,$ and $R_{12}^{\rm exp}(i,x ; j,y)=0$ when $i\geqslant j$, and likewise for $\hat{R}^{\rm exp}$ and $\bar{R}^{\rm exp}$. The other entries  depend on the value of $\alpha$ and the sign of $x-y$.

\textbf{Case $\alpha>1/2$:} When $x>y$,
$$ R_{22}^{\rm exp}(i,x ; j,y)  = - \int_{\mathcal{C}_{a_z}^{\pi/3}} \frac{(1+2z)^{m_i}(1-2z)^{m_j}}{(1-2z)^{n_i}(1+2z)^{n_j}} \frac{2ze^{-\vert x-y\vert z}\dd z}{(2\alpha-1-2z)2\alpha-1+2z} ,$$
and when $x<y$
$$ R_{22}^{\rm exp}(i,x ; j,y)  = \int_{\mathcal{C}_{a_z}^{\pi/3}} \frac{(1+2z)^{m_j}(1-2z)^{m_i}}{(1-2z)^{n_j}(1+2z)^{n_i}} \frac{2ze^{-\vert x-y\vert z}\dd z}{(2\alpha-1-2z)(2\alpha-1+2z)},$$
where $(1-2\alpha)/2<a_z<(2\alpha-1)/2$. One immediately checks that $R_{22}^{\rm exp}$ is antisymmetric as we expect. When $i<j$ and $x>y$
$$ R_{12}^{\rm exp}(i,x ; j,y)  = -\frac{1}{2\I\pi} \int_{\mathcal{C}_{1/4}^{\pi/3}}
\frac{(1+2z)^{n_i}}{(1+2z)^{n_j}}\frac{(1-2z)^{m_j}}{(1-2z)^{m_i}}
e^{-\vert x-y\vert z}\dd z,$$
while if $x<y$, $ R_{12}^{\rm exp}(i,x ; j,y)=R_{12}^{\rm exp}(i,y ; j,x)$. Note that $R_{12}$ is not antisymmetric nor symmetric (except when $k=1$, i.e. for the one point distribution).

\textbf{Case $\alpha<1/2$:} When $x>y$, we have
\begin{multline}
\hat{R}_{22}^{\rm exp}(i,x ; j,y)  =\frac{-e^{\frac{1-2\alpha}{2}y}}{2}  \int \frac{(1+2z)^{m_i}(2\alpha)^{m_j}}{(1-2z)^{n_i}(2-2\alpha)^{n_j}} \frac{e^{-xz}}{2\alpha-1+2z}\dd z \\
+ \frac{e^{\frac{1-2\alpha}{2}x}}{2} \int \frac{(1+2z)^{m_j}(2\alpha)^{m_i}}{(1-2z)^{n_j}(2-2\alpha)^{n_i}} \frac{e^{-yz}}{2\alpha-1+2z}\dd z \\
- \int_{\mathcal{C}_{a_z}^{\pi/3}} \frac{(1+2z)^{m_i}(1-2z)^{m_j}}{(1-2z)^{n_i}(1+2z)^{n_j}} \frac{1}{2\alpha-1-2z}\frac{1}{2\alpha-1+2z} 2ze^{-\vert x-y\vert z}\dd z \\
-\frac{e^{(x-y)\frac{1-2\alpha}{2}}}{4}\frac{(2\alpha)^{m_i}(2-2\alpha)^{m_j}}{(2-2\alpha)^{n_i}(2\alpha)^{n_j}}+ \frac{e^{(y-x)\frac{1-2\alpha}{2}}}{4}\frac{(2\alpha)^{m_j}(2-2\alpha)^{m_i}}{(2-2\alpha)^{n_j}(2\alpha)^{n_i}},
\label{eq:R22expcasalphapetit1}
\end{multline}
where the contours in the two first integrals pass to the right of $ (1-2\alpha)/2 $. When $x<y$, the sign of the third term is flipped  so that $\hat{R}_{22}^{\rm exp}(i,x ; j,y)=-\hat{R}_{22}^{\rm exp}(j, y ; i,x)$.
We can write slightly simpler formulas by reincorporating residues in the first two integrals: thus, when $x>y$,
\begin{multline}
\hat{R}_{22}^{\rm exp}(i,x ; j,y)  =\frac{-e^{\frac{1-2\alpha}{2}y}}{2}  \int_{\mathcal{C}_{a_z}^{\pi/3}} \frac{(1+2z)^{m_i}(2\alpha)^{m_j}}{(1-2z)^{n_i}(2-2\alpha)^{n_j}} \frac{e^{-xz}}{2\alpha-1+2z}\dd z \\
+ \frac{e^{\frac{1-2\alpha}{2}x}}{2} \int_{\mathcal{C}_{a_z}^{\pi/3}} \frac{(1+2z)^{m_j}(2\alpha)^{m_i}}{(1-2z)^{n_j}(2-2\alpha)^{n_i}} \frac{e^{-yz}}{2\alpha-1+2z}\dd z \\
- \int_{\mathcal{C}_{a_z}^{\pi/3}} \frac{(1+2z)^{m_i}(1-2z)^{m_j}}{(1-2z)^{n_i}(1+2z)^{n_j}} \frac{1}{2\alpha-1-2z}\frac{1}{2\alpha-1+2z} 2ze^{-\vert x-y\vert z}\dd z,
\end{multline}
where $\frac{2\alpha-1}{2} <a_z<\frac{1-2\alpha}{2}$.
When $i<j$, if $x>y$
\begin{equation}
\hat{R}_{12}^{\rm exp}(i,x ; j,y)  = -\int_{\mathcal{C}_{1/4}^{\pi/3}}\frac{(1+2z)^{n_i}}{(1+2z)^{n_j}}\frac{(1-2z)^{m_j}}{(1-2z)^{m_i}}e^{-\vert x-y\vert z}\dd z,
\label{eq:R12expcasalphapetit1}
\end{equation}
while if $x<y$, $ \hat{R}_{12}^{\rm exp}(i,x ; j,y)=\hat{R}_{12}^{\rm exp}(i,y ; j,x)$.
\begin{remark}
	The formula \eqref{eq:R12expcasalphapetit1} for $\hat{R}_{12}^{\rm exp}$ is not exactly the limit of the expression for $\hat{R}_{12}^{\rm geo}$ in \eqref{eq:R12geocasc>1}.  Indeed, the second term in \eqref{eq:R12geocasc>1}, which corresponds to a residue when $w=1/c$ does not have its counterpart in \eqref{eq:R12expcasalphapetit1}. This is because we have assumed that $a_w<(2\alpha-1)/2$ in the contours for $I_{12}^{\rm exp}$ in \eqref{eq:I12exp},  while we had $w>1/c$ in the contours for  $I_{12}^{\rm geo}$ in the case $c>1$.
	\label{rem:echangepole}
\end{remark}
\textbf{Case $\alpha=1/2$:} When $x>y$,
\begin{multline*}
\bar{R}_{22}^{\rm exp}(i,x ; j,y)  = - \int_{\mathcal{C}_{1/4}^{\pi/3}} \frac{(1+2z)^{m_i}}{(1-2z)^{n_i}} \frac{e^{-xz}}{4z}\dd z
+  \int_{\mathcal{C}_{1/4}^{\pi/3}} \frac{(1+2z)^{m_j}}{(1-2z)^{n_j}} \frac{e^{-yz}}{4z}\dd z \\
+ \int_{\mathcal{C}_{1/4}^{\pi/3}} \frac{(1+2z)^{m_i}(1-2z)^{m_j}}{(1-2z)^{n_i}(1+2z)^{n_j}} \frac{e^{-\vert x-y\vert z}\dd z}{2z} \ \  -\frac{1}{4},\\
\end{multline*}
with a modification of the last two terms when $x<y$ so that $\bar{R}_{22}^{\rm exp}(i,x ; j,y)=-\bar{R}_{22}^{\rm exp}(j, y ; i,x)$.
When $i<j$, if $x>y$
\begin{equation*}
\bar{R}_{12}^{\rm exp}(i,x ; j,y)  = - \int_{\mathcal{C}_{1/4}^{\pi/3}}\frac{(1+2z)^{n_i}}{(1+2z)^{n_j}}\frac{(1-2z)^{m_j}}{(1-2z)^{m_i}}e^{-\vert x-y\vert z}\dd z,
\end{equation*}
while if
$x<y$, $ \bar{R}_{12}^{\rm exp}(i,x ; j,y)=\bar{R}_{12}^{\rm exp}(i,y ; j,x)$.
\begin{proof}[Proof of Proposition \ref{prop:kernelexponential}]
	By Lemma \ref{prop:limitgeomtoexp}, the passage times $\lbrace H(n_i, m_i)\rbrace_i$ are the limits when $q$ goes to $1$ of the passage times $\lbrace G(n_i,m_i)\rbrace_i$, and we know the probability distribution function of $\lbrace G(n_i,m_i)\rbrace_i$ from \eqref{eq:pdfG}. Thus, the proof proceeds with two steps: (1) Take the asymptotics of the correlation kernel $K^{\rm geo}$  involved in the probability distribution of $\lbrace G(n_i,m_i)\rbrace_i$ (Section \ref{sec:kernelgeom}), under the scalings  $a_i \equiv \sqrt{q}$ and $c=\sqrt{q} \big( 1+(\alpha-1)(q-1)\big)$, $x=(1-q)u$, $y=(1-q)v$;
	(2) Deduce the asymptotic behavior of $\Pf\big(J+K^{\rm geo}\big)$ from the pointwise asymptotics of $K^{\rm geo}$.
	Using the weak convergence of Lemma \ref{prop:limitgeomtoexp}, we get the distribution of $\lbrace H(n_i, m_i)\rbrace_i$.
	
	\paragraph{\textbf{Step (1):}} We use Laplace's method to take asymptotics of $K^{\rm geo}(i,u;j,v)$ under the scalings above.
	Let us provide a detailed proof for $K^{\rm geo}_{11}$. The arguments are quite similar for $K^{\rm geo}_{12}$ and $K^{\rm geo}_{22}$. We have
	\begin{equation}
	K^{\rm geo}_{11}(i, u ; j, v) =   \iint_{\mathcal{C}^2} \frac{(z-w)(z-c)(w-c)}{(z^2-1)(w^2-1)(zw-1)}
	\frac{(z-\sqrt{q})^{n_i}(w-\sqrt{q})^{n_j}}{(1-z\sqrt{q})^{m_i}(1-w\sqrt{q})^{m_j}} \frac{\dd z  \dd w}{z^{n_i+1}w^{n_j+1}z^uw^v},
	\label{eq:expressionforK11kernelexp}
	\end{equation}
	where the contour $\mathcal{C}$ can be chosen as a circle around $0$ with radius $\frac{1+q^{-1/2}}{2}$.
	Under the scalings considered, the dominant term in the integrand is
	$$ \frac{1}{z^{u}w^v}  = \frac{1}{z^{(1-q)^{-1}x} w^{(1-q)^{-1}y}} = \exp\big((1-q)^{-1} (-x\log(z) - y\log(w))\big).$$
	\newcommand{\pacman}{<}
	Using Cauchy's theorem, we can deform the integration contours as long as we do not cross any pole. Hence, we can deform the contour $\mathcal{C}$ to be the contour $\mathcal{C}_{\pacman}$ depicted in Figure \ref{fig:newcontour}. It is constituted of two rays departing $\frac{1+q^{-1/2}}{2}$ with angles $\pm \pi/3$ in a neighbourhood of $1$ of size $\epsilon$ (for some small $\epsilon >0$), and the rest of the contour is given by an arc of a circle centered at $0$. Its radius $R>\frac{1+q^{-1/2}}{2}$  is chosen so that the circle intersects with the endpoints  of the two rays departing $\frac{1+q^{-1/2}}{2}$ with angles $\pm \pi/3$.
	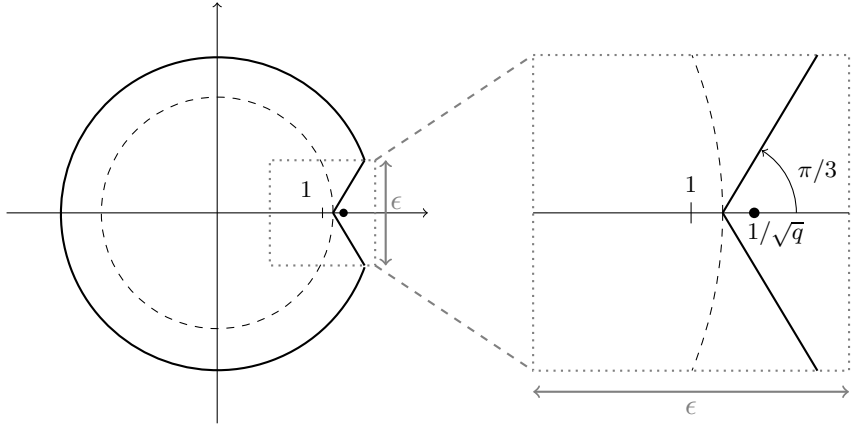
\begin{figure}
		\begin{tikzpicture}[scale=1.4]
		\draw[->] (-2, 0) -- (2, 0);
		\draw[->] (0,-2) -- (0,2);
		\draw (1,-0.05) -- (1, 0.05) node[anchor=south east] {\footnotesize{$1$}};
		\draw[dashed]  (0,0) circle(1.1);
		\fill (1.2, 0) circle(0.04);
		\draw[thick] (1.1, 0) -- (1.4, 0.5);
		\draw[thick] (1.1, 0) -- (1.4, -0.5);
		\draw[gray, thick, dotted] (1.5, 0.5) -- (1.5, -0.5) -- (0.5, -0.5) -- (0.5, 0.5) -- (1.5, 0.5);
		\draw[thick] (1.4, 0.5) arc(20:340:1.4866);
		\draw[thick, gray, <->] (1.6, -0.5) -- (1.6, 0.5);
		\draw[gray] (1.7, 0.1) node{$\epsilon$};

		\draw[thick, gray, dashed] (1.5, 0.5) -- (3, 1.5);
		\draw[thick, gray, dashed] (1.5, -0.5) -- (3, -1.5);
		\draw[thick, gray, dotted] (3, 1.5) -- (6,1.5) -- (6,-1.5) -- (3, -1.5) -- (3,1.5);
		
		\draw (3, 0) -- (6, 0);
		\draw (4.5,-0.1) -- (4.5, 0.1) node[anchor=south] {\footnotesize{$1$}};
		\draw[dashed]  (4.8, 0) arc(0:22:4);
		\draw[dashed]  (4.8, 0) arc(0:-22:4);
		\fill (5.1, 0) circle(0.05);
		\draw (5.3, -0.2) node{\footnotesize{$1/\sqrt{q}$}};
		\draw[thick] (4.8, 0) -- (5.7, 1.5);
		\draw[thick] (4.8, 0) -- (5.7, -1.5);
		\draw[->] (5.5, 0) arc (0:58:0.7);
		\draw (5.7, 0.4) node{\footnotesize{$\pi/3$}};
		\draw[thick, gray, <->] (3, -1.7) -- (6,-1.7);
		\draw[gray] (4.5, -1.85) node{$\epsilon$};
		\end{tikzpicture}
		\caption{The deformed contour
			$\mathcal{C}_{\pacman}$
			(thick black line) and the contour $\mathcal{C}$ (dashed line). On the right, zoom on the portion of the contour which contributes to the limit. }
		\label{fig:newcontour}
	\end{figure}
	Along the contour $\mathcal{C}_{\pacman}$, $\Real[-\log(z)]$ attains its maximum at $\frac{1+q^{-1/2}}{2}$. Thus, the contribution of integration along the portion of contour which is outside the $\epsilon$-box around $1$ is negligible in the limit. More precisely, it  has a contribution of order $\mathcal{O}(e^{-d_{\epsilon}(1-q)^{-1}}) $ for some $d_{\epsilon}>0$. 
	
	Now we estimate the contribution of the integration in the neighbourhood of $1$. Let us rescale the variables around $1$ by setting $z=1+\tilde{z}(1-q) $ and $w=1+\tilde{w}(1-q)$. Cauchy's theorem provides us some freedom in the choice of contours, and we can assume that after the change of variables, the integration contours are given by two rays of length $\epsilon(1-q)^{-1}$ departing $1/4$ in the directions $\pi/3$ and $-\pi/3$. When $q\to 1$, we have the point-wise limits
	\begin{eqnarray*}
		z^u \longrightarrow  e^{\tilde{z}x}, \ \  & \frac{z-\sqrt{q}}{1-\sqrt{q}} \longrightarrow  1+2\tilde{z}, \ \  &\frac{z-c}{1-\sqrt{q}} \longrightarrow -1+2\tilde{z}+2\alpha  \\
		w^v \longrightarrow  e^{\tilde{w}y}, \ \  & \frac{1-z\sqrt{q}}{1-\sqrt{q}} \longrightarrow 1-2\tilde{z}, \ \  &
		\frac{1-zc}{1-\sqrt{q}} \longrightarrow -1-2\tilde{z} +2\alpha
	\end{eqnarray*}
	Using Taylor expansions to a high enough order, the integral in \eqref{eq:expressionforK11kernelexp} converges to the integral of the pointwise limit of the integrand in \eqref{eq:expressionforK11kernelexp}, modulo a $\mathcal{O}(\epsilon)$ error, uniformly in $q$. The kind of estimates needed to control the error are given with more details in a similar situation in Section \ref{sec:pointwiseGSE} (see in particular Equations \eqref{eq:decoupage1} and \eqref{eq:decoupage2}). Moreover, since the limiting integrand has exponential decay in direction $\pm\pi/3$, one can replace the finite integration contours by infinite contours going to $\infty e^{\I\pi/3}$ and $\infty e^{-\I\pi/3}$ with an error going to $0$ as $q$ goes to 1.
	Since the size $\epsilon$ of the box around $1$ (on which we have restricted the integration) can be taken as small as we want, we finally get that
	\begin{multline*}
	\frac{K^{\rm geo}_{11}(i,u:j,v)}{(1-\sqrt{q})^{n_i+n_j-m_i - m_j+2}}  \xrightarrow[q\to 1]{}  \int_{\mathcal{C}_{1/4}^{\pi/3}}\dd\tilde z\int_{\mathcal{C}_{1/4}^{\pi/3}}\dd\tilde w \frac{\tilde z-\tilde w}{4\tilde z\tilde w(\tilde z+\tilde w)} e^{-x\tilde z-y\tilde w} \\ \times
	\frac{(1+2\tilde z)^{n_i}(1+\tilde 2w)^{n_j}}{(1-2\tilde z)^{m_i}(1-2\tilde w)^{m_j}} (2\tilde z+2\alpha -1)(2\tilde w+2\alpha -1) .
	\end{multline*}
	Likewise, we obtain
	\begin{multline*}
	\frac{I^{\rm geo}_{12}(i, u: j,v)}{(1-q)(1-\sqrt{q})^{n_i+m_j-n_j-m_i}}  \xrightarrow[q\to 1]{}  \int_{\mathcal{C}_{a_z}^{\pi/3}}\dd\tilde z\int_{\mathcal{C}_{a_w}^{\pi/3}}\dd\tilde w \frac{\tilde z-\tilde w}{2\tilde z(\tilde z+\tilde w)} e^{-x\tilde z-y\tilde w}  \\ \times
	\frac{(1+2\tilde z)^{n_i}}{(1-2\tilde w)^{n_j}}  \frac{(1+2\tilde w)^{m_j}}{(1-2 \tilde z)^{m_i}} \frac{2\alpha -1 + 2\tilde z}{2\alpha -1-2\tilde w} ,
	\end{multline*}
	where the new contours $\mathcal{C}_{a_z}^{\pi/3}$ and $\mathcal{C}_{a_w}^{\pi/3}$ are chosen as in \eqref{eq:I12exp}. In the case where  $\alpha<1/2$ or $\alpha=1/2$, the contour $\mathcal{C}_{a_w}^{\pi/3}$ is not exactly the limit of the contour in the definition of $I_{12}^{\rm geo}$ in  \eqref{eq:decompositionK12geo}, but we can exclude the pole at $w=(2\alpha-1)/2$ (so that $a_w<(2\alpha-1)/2$) and remove the corresponding residue in $\hat{R}^{\rm exp}_{12}$ and  $\bar{R}^{\rm exp}_{12}$ (see Remark \ref{rem:echangepole}). Thus, we have that
	\begin{equation}
	\frac{R^{\rm geo}_{12}(i, u; j,v)}{(1-q)(1-\sqrt{q})^{n_i+m_j-n_j-m_i}}  \xrightarrow[q\to 1]{} R^{\rm exp}_{12}(i,x;jy),
	\label{eq:convergenceRgeo}
	\end{equation}
	and likewise for $\hat{R}^{\rm geo}_{12}$ and $\bar{R}^{\rm geo}_{12}$ (modulo the residue that is reincorporated in $I_{12}^{\rm exp}$).
	Convergence \eqref{eq:convergenceRgeo} needs some justification though: for $x>y$, we use the exact same arguments as above, since the factor $1/(z^{u-v})\approx e^{-(x-y)\tilde{z}}$ in \eqref{eq:R12geosimple} has exponential decay along the tails of the contour. For $x<y$, we make the change of variables $z=1/\hat{z}$ and we adapt the case $x>y$. This shows that
	$$\frac{K^{\rm geo}_{12}(i, u; j,v)}{(1-q)(1-\sqrt{q})^{n_i+m_j-n_j-m_i}}  \xrightarrow[q\to 1]{} K^{\rm exp}_{12}(i,x;jy). $$
	Using very similar arguments, we find that
	\begin{multline*}
	\frac{I^{\rm geo}_{22}(i, u;j, v)}{(1-q)^2 (1-\sqrt{q})^{m_i+m_j-n_i-n_j-2}}  \xrightarrow[q\to 1]{}  \int_{\mathcal{C}_{b_z}^{\pi/3}}\dd\tilde z\int_{\mathcal{C}_{b_z}^{\pi/3}}\dd\tilde w \frac{\tilde z-\tilde w}{\tilde z+\tilde w} e^{-x\tilde z-y\tilde w} \\ \times
	\frac{(1+2\tilde z)^{m_i}(1+2\tilde w)^{m_j}}{(1-2\tilde z)^{m_i}(1-2\tilde w)^{m_j}} \frac{1}{2\alpha -1 - 2\tilde z}\frac{1}{2\alpha -1 - 2\tilde w},
	\end{multline*}
	where $\mathcal{C}_{b_z}^{\pi/3}$ and $\mathcal{C}_{b_w}^{\pi/3}$ are chosen as in \eqref{eq:I22exp}, and
	$$\frac{R^{\rm geo}_{22}(i, u; j,v)}{(1-q)^2 (1-\sqrt{q})^{m_i+m_j-n_i-n_j-2}}  \xrightarrow[q\to 1]{} R^{\rm exp}_{22}(i,x;jy),$$
	and similar limits hold for $\hat{R}^{\rm geo}_{22}$ and $\bar{R}^{\rm geo}_{22}$.
	
	Note that when taking the Pfaffian, the factors $(1-\sqrt{q})^{n_i+n_j-m_i-m_j+2}$ and $(1-\sqrt{q})^{m_i+m_j-n_i-n_j-2}$ cancel out each other. This is because one can multiply the columns with odd index by $(1-\sqrt{q})^{m_i+m_j-n_i-n_j-2}$ and then the rows with even index by $(1-\sqrt{q})^{n_i+n_j-m_i-m_j+2}$, without changing the value of the Pfaffian\footnote{Otherwise said, we conjugate the kernel by a diagonal kernel with determinant $1$.}. In the end, we find that for all integer $n\geqslant 1$ , positive real variables $x_1, \dots, x_n$, and integers $1\leqslant i_1, \dots, i_n\leqslant k$,
	\begin{equation*}
	\frac{\Pf\Big[K^{\rm geo}\big( i_s, (1-q)^{-1}x_s ; i_t,  (1-q)^{-1}x_t \big) \Big]_{s, t=1}^n}{(1-q)^n} \xrightarrow[q\to 1]{} \Pf\Big[K^{\rm exp}\big(i_s, x_s ; i_t, x_t \big) \Big]_{s,t=1}^n
	\end{equation*}

	\paragraph{\textbf{Step (2):}} We need to prove that
	\begin{multline}
	\sum_{i_1, \dots, i_n=1 }^k \ \ \  \sum_{u_1 =\lfloor g_{i_1}(1-q)^{-1}\rfloor}^{+\infty} \dots \sum_{u_n =\lfloor g_{i_n}(1-q)^{-1}\rfloor}^{+\infty}\Pf\Big[K^{\rm geo}\big( i_s, u_s ; i_t,  u_t \big) \Big]_{s, t=1}^n\\ \xrightarrow[q\to1]{}
	\sum_{i_1, \dots, i_n=1 }^k \int_{g_{i_1}}^{+\infty} \dots  \int_{g_{i_n}}^{+\infty}  \Pf\Big[K^{\rm exp}\big(i_s, x_s ; i_t, x_t \big) \Big]_{s,t=1}^n\mathrm{d}x_1 \dots \mathrm{d}x_k .
	\label{eq:integralsPfaffianlimits}
	\end{multline}
	For that purpose we need to control the asymptotic behaviour of the kernel $K^{\rm geo}$.
	\begin{lemma}
		There exists positive constants $C>0$ and $\frac 1 2 >c_1>c_2>\frac{1-2\alpha}{2}$ such that uniformly for $q$ in a neighbourhood of $1$, for all $ x,y >0$,
		\begin{eqnarray}
		\bigg\vert \frac{K^{\rm geo}_{11}(i, (1-q)^{-1}x;j, (1-q)^{-1}y)}{(1-\sqrt{q})^{n_i+n_j-m_i-m_j}}  \bigg\vert \leqslant& C e^{-c_1 (x+y)}, \label{eq:boundexpgeomK11}\\
		\bigg\vert \frac{K^{\rm geo}_{12}(i, (1-q)^{-1}x;j, (1-q)^{-1}y)}{(1-\sqrt{q})^{n_i+n_j-m_i-m_j}}  \bigg\vert \leqslant& C   e^{-xc_1+yc_2}   , \label{eq:boundexpgeomK12}\\
		\bigg\vert \frac{K^{\rm geo}_{22}(i, (1-q)^{-1}x;j, (1-q)^{-1}y)}{(1-\sqrt{q})^{n_i+n_j-m_i-m_j}}  \bigg\vert \leqslant&  Ce^{xc_2+yc_2}.\label{eq:boundexpgeomK22}
		\end{eqnarray}
		\label{lem:expboundgeom}
	\end{lemma}
	\begin{proof}
		The arguments are similar for all integrals involved. It amounts to justify that we can restrict the integration on a region on the complex plane with real part greater than $c_1$ or $-c_2$. Let us first see how it works for the bound \eqref{eq:boundexpgeomK11}. We start with \eqref{eq:expressionforK11kernelexp} as in Step (1). The contribution of integration along the circular parts of $\mathcal{C}_{\pacman}$ can be bounded by a constant for $q$ in a neighbourhood of $1$. Next we perform the change of variables $z=1+\tilde{z}(1-q) $ and $w=1+\tilde{w}(1-q)$ as in Step (1) and observe that all factors in the integrand are bounded for $q$ in a neighbourhood of $1$, except the factor $1/(z^u w^v)$ which behaves as
		$$\frac{1}{z^{u}w^v} = \exp\big(-x(1-q)^{-1}\log(1+\tilde{z}(1-q))-y(1-q)^{-1}\log(1+\tilde{w}(1-q)) \big) \approx e^{-x\tilde{z}- y\tilde{w}}.$$
		Since we can deform the contours so that
		$\Re[\tilde{z}]$ and $ \Re[\tilde{w}]$ are greater than $c_1$ (for some $c_1<1/2$), we get the desired bound for $K_{11}$. Strictly speaking, Cauchy's theorem allows to deform the contour $\mathcal{C}_{\pacman}$ so that the real part stays above $c_1$ along the two rays, and then on cuts the circular part which has a negligible contribution.
		Similarly, we find that
		$$\bigg\vert \frac{I^{\rm geo}_{12}(i, (1-q)^{-1}x;j, (1-q)^{-1}y)}{(1-\sqrt{q})^{n_i+n_j-m_i-m_j}}  \bigg\vert \leqslant  C   e^{-xc_1+yc_2}$$
		by restricting the integration on contours such that  $\Re[\tilde{z}]>c_1$ and $ \Re[\tilde{w}]>-c_2$. When deforming the contour for $w$, we have to take into account the pole at $\frac{2\alpha-1}{2}$.
		Since for any $\alpha>0$, $\frac{1-2\alpha}{2}<\frac{1}{2}$, we can always find suitable constants $c_1,c_2>0$ such that  $\frac 1 2 >c_1>c_2>\frac{1-2\alpha}{2}$. By the same kind of arguments, 
		$$\bigg\vert \frac{\hat{R}^{\rm geo}_{12}(i, (1-q)^{-1}x;j, (1-q)^{-1}y)}{(1-\sqrt{q})^{n_i+n_j-m_i-m_j}}  \bigg\vert \leqslant  C   e^{-c\vert x-y\vert} <  C   e^{-xc_1+yc_2},$$
		and it is easier to check that the inequality also holds for $\bar{R}^{\rm geo}_{12}$ and $R^{\rm geo}_{12}$,
		so that \eqref{eq:boundexpgeomK12} holds.
		For $K_{22}$, one  does not need exponential decay. We have  that
		$$\bigg\vert \frac{I^{\rm geo}_{22}(i, (1-q)^{-1}x;j, (1-q)^{-1}y)}{(1-\sqrt{q})^{n_i+n_j-m_i-m_j}}  \bigg\vert \leqslant C $$
		and
		$$\bigg\vert \frac{\hat{R}^{\rm geo}_{22}(i, (1-q)^{-1}x;j, (1-q)^{-1}y)}{(1-\sqrt{q})^{n_i+n_j-m_i-m_j}}  \bigg\vert \leqslant  C \max\big\lbrace 1 ; e^{\vert x-y\vert \frac{1-2\alpha}{2}}\big\rbrace <Ce^{xc_2+yc_2}, $$
		so that \eqref{eq:boundexpgeomK22} holds.
	\end{proof}
	Hence using  Lemma \ref{lem:hadamard},   
	$$\bigg\vert \Pf\Big[K^{\rm geo}\big( i_s, (1-q)^{-1}x_s ; i_t,  (1-q)^{-1}x_t \big) \Big]_{s, t=1}^n  \bigg\vert \leqslant  (2n)^{n/2}C^n e^{-(c_1-c_2)(x_1+\dots + x_n)}.$$
	and \eqref{eq:integralsPfaffianlimits} follows by dominated convergence. Dominated convergence also proves the convergence of Fredholm Pfaffian series expansions, so that
	\begin{equation}
	\Pf\big(J-K^{\rm geo}\big)_{\ell^2\big(\tilde{\mathbb{D}}_k(h_1 (1-q)^{-1}, \dots, h_k(1-q)^{-1})\big)} \xrightarrow[q\to 1]{} \Pf\big(J-K^{\rm exp}\big)_{\mathbb{L}^2\big(\mathbb{D}_k(h_1, \dots, h_k)\big)}.
	\label{eq:limitFredholmPfaffiansdiscretetocontinuous}
	\end{equation}
	
	Finally, Proposition \ref{prop:limitgeomtoexp} implies that
	\begin{equation*} \PP\left(\bigcap_{i=1}^k \big\lbrace G(n_i, m_i)\leqslant \frac{x_i}{1-q} \big\rbrace\right) \xrightarrow[q\to 1]{}   \PP\left(  \bigcap_{i=1}^k \big\lbrace H(n_i,m_i)\leqslant x_i\big\rbrace \right).
	\end{equation*}
	which, together with  \eqref{eq:limitFredholmPfaffiansdiscretetocontinuous}, concludes the proof.
\end{proof}

\section{Convergence to the GSE Tracy-Widom distribution}
\label{sec:GSEasymptotics}
This section is devoted to the proof of the first part of Theorem \ref{theo:LPPdiagointro}, and thus we assume that $\alpha >1/2$. 
By Proposition \ref{prop:kernelexponential}, the probability $\PP\left( H(n,n)  < x \right)$ can be expressed as the Fredholm Pfaffian of the kernel $K^{\rm exp}$. The most natural would be to prove that under the scalings considered, $K^{\rm exp}$ converges pointwise to $K^{\rm GSE}$, and that the convergence of the Fredholm Pfaffian holds as well, using estimates to control the absolute convergence of the series expansion. Unfortunately, this approach cannot work here because $K^{\rm exp}$ does not converge to $K^{\rm GSE}$. The next remark provides a heuristic explanation. 

\begin{remark}
	\label{rem:explaincollisions} We expect the kernel $K^{\rm exp}(1,x;1,y)$ to be the Pfaffian correlation kernel of some simple Pfaffian point process, say 
	$ X^{(n)}:= \big\lbrace X^{(n)}_1> \dots > X^{(n)}_n \big\rbrace. $
	On the other hand, $K^{\rm GSE}(x, y)$ is the correlation kernel of a Pfaffian point process
	$\chi_1>  \chi_2> \dots $ having the same law as the first eigenvalues of a matrix from the GSE in the large size limit.

	In light of Remark \ref{rem:symmetricRSK}, the points  $X^{(n)}_i$ are such that  $X^{(n)}_1 + \dots + X^{(n)}_i$ has the same law as the maximal weight of a $i$-tuple of non-intersecting up-right paths in a symmetric exponential environment with weights $\mathcal{E}(\alpha)$ on the diagonal and $\mathcal{E}(1)$ off-diagonal (in particular $X^{(n)}_1$ has the same distribution as $H(n,n)$ since the half-space LPP and symmetrized LPP are equivalent).  When $\alpha$ is large enough, it is clear that the optimal paths will seldom visit the diagonal, so that $X_1^{(n)}$ and  $X_2^{(n)}$ will have (under $n^{1/3}$ scaling) the same scaling limit. More precisely, we expect that $X_{2i-1}^{(n)}$ and  $X_{2i}^{(n)}$ will both  converge to $\chi_i$, so that the point process $X^{(n)}$ converges as $n$ goes to infinity to a non-simple point process where each point has multiplicity 2.
	
	Let us denote by $\rho^{\rm GSE}$ the correlation function of the point process  $\chi_1, \chi_2, \dots $, and let $\rho^{(n)}$ be the correlation function of the rescaled point process $\frac{X^{(n)}-4n}{2^{4/3}n^{1/3}}$. By the characterization of correlation functions in  \eqref{eq:propertycorrelations}, we expect that for distinct points $x_1, \dots , x_k$,
	$$ \rho^{(n)}(x_1, \dots, x_k)\xrightarrow[n\to\infty]{} 2^{k} \rho^{\rm GSE}(x_1, \dots, x_k),$$
	and a singularity should appear in the limit of $\rho^{(n)}(x_1, \dots, x_k)$ for $k>1$ in the neighborhood of $x_i=x_j$. 
\end{remark}

\subsection{A formal representation of the GSE distribution}\label{formal}

We introduce a kernel which is a scaling limit of $K^{\rm exp}$ and whose Fredholm Pfaffian corresponds to the GSE distribution. However, it is not the Pfaffian correlation kernel of the GSE point process. This kernel is distribution valued, and we show in Proposition \ref{prop:equivalentpfaffians} that its Fredholm Pfaffian is the same as that of $K^{\rm GSE}$, by expanding and reordering terms in the Fredholm Pfaffian. In fact, to avoid dealing with distribution valued kernels, we later prove an approximate prelimiting version of this result as Proposition \ref{prop:approxequivalentpfaffians}.

Let us first recall that by Lemma \ref{def:GSEdistribution}, the kernel $K^{\rm GSE}$ has the form
$$ K^{\rm GSE}(x,y)  = \begin{pmatrix}
A(x,y) & -\partial_y A(x, y) \\
-\partial_x A(x,y) & \partial_x\partial_y A(x,y)
\end{pmatrix} $$
where $A(x,y)$ is the smooth and antisymmetric kernel $K^{\rm GSE}_{11}$.

We introduce the kernel
$$ K^{\infty}(x,y) = \begin{pmatrix}
A(x,y) & -2\partial_y A(x, y) \\
-2\partial_x A(x,y) & 4\partial_x\partial_y A(x,y) + \delta'(x,y)
\end{pmatrix} $$
where $\delta'$ is a distribution on $\R^2$ such that
\begin{equation} \int \int f(x,y)\delta'(x,y)\mathrm{d}x\mathrm{d}y =  \int  \left.\big(\partial_y f(x,y) - \partial_x f(x,y)\big)\right|_{y=x}\mathrm{d}x,\label{eq:defdeltaprime}
\end{equation}
for smooth and compactly supported test functions $f$.

The quantity
\begin{equation} \label{eq:kfold} \int_{\R^k} \Pf[K^{\infty}(x_i, x_j)]_{i,j=1}^k\mathrm{d}x_1\dots \mathrm{d}x_k
\end{equation}
makes perfect sense because all terms of the form $\delta'(x_i, x_j)$ are integrated against smooth functions of $x_i, x_j$ with sufficient decay. Expanding the Pfaffian and using the definition of $\delta'$, \eqref{eq:kfold} collapses into  a sum of $j$-fold integrals for $j=k$ down to $k-\lfloor k/2\rfloor$, where the integrands are smooth functions. We can then form the Fredholm Pfaffian of $K^{\infty}$, and we will show that the Fredholm Pfaffian series expansion is absolutely convergent so that we can reorder terms, and ultimately recognize the Fredholm Pfaffian of another kernel.
\begin{proposition}Let $\mathbf{A}:\R^2 \to \skeww(\R)$ be a kernel of the form
	$$ \mathbf{A}(x,y) =  \begin{pmatrix}
	A(x,y) & -\partial_y A(x, y) \\
	-\partial_x A(x,y) & \partial_x\partial_y A(x,y)
	\end{pmatrix},$$
	where $A$ is smooth, antisymmetric, and $\mathbf{A}$ satisfies the decay hypotheses of Lemma \ref{lem:hadamard}.
	Let $\mathbf{B}$ be the kernel
	$$ \mathbf{B}(x,y) =  \begin{pmatrix}
	A(x,y) & -2\partial_y A(x, y) \\
	-2\partial_x A(x,y) & 4\partial_x\partial_y A(x,y) +\delta'(x,y)
	\end{pmatrix}.$$
	Then for any $x\in \R$,
	\begin{equation}
	\Pf[J-\mathbf{A}]_{\mathbb{L}^2(x, +\infty)}= \Pf[J-\mathbf{B}]_{\mathbb{L}^2(x, +\infty)}.
	\label{eq:equivalentpfaffians}
	\end{equation}
	In particular,
	\begin{equation*}
	\Pf[J-K^{\rm GSE}]_{\mathbb{L}^2(x, +\infty)}= \Pf[J-K^{\infty}]_{\mathbb{L}^2(x, +\infty)}.
	\end{equation*}
	\label{prop:equivalentpfaffians}
\end{proposition}

\begin{proof}[Proof of Proposition \ref{prop:equivalentpfaffians}]
	In order to simplify the notations in this proof, we assume that all integrals are  performed over the domain $(x, +\infty)$.  
	Denoting
	$$b_m =  \int\mathrm{d}x_1\dots\int\mathrm{d}x_m  \ \Pf[\mathbf{B}(x_i, x_j)]_{i,j=1}^m, $$
	and 
	$$a_k =  \int \mathrm{d}x_1\dots\int \mathrm{d}x_k\  \Pf[\mathbf{A}(x_i, x_j)]_{i,j=1}^k,$$
	we have 
	$$ \Pf[J-\mathbf{A}] = \sum_{k=0}^{\infty} \frac{(-1)^k}{k!} a_k \ \ \text{ and }\ \   \Pf[J-\mathbf{B}] = \sum_{m=0}^{\infty} \frac{(-1)^m}{m!} b_m. $$
	In order to analyze the quantity $b_m$, we will use the following definition of the Pfaffian, valid for any antisymmetric matrix $M$ of size $2m$:
	\begin{equation}
	\Pf[M] = \sum_{\alpha = \lbrace (i_1, j_1), \dots, (i_m, j_m)\rbrace}\epsilon(\alpha) M_{i_1, j_1}\cdots M_{i_m, j_m},
	\label{eq:defPfaffian2}
	\end{equation}
	where the sum runs over partitions into pairs of the set $\lbrace 1, \dots, 2m\rbrace$, in the form $$\alpha = \big\lbrace (i_1, j_1), \dots, (i_m, j_m) \big\rbrace $$ with $i_1< \dots < i_m$ and $i_{\ell}<j_{\ell}$, and the signature $\varepsilon(\alpha)$ of the partition $\alpha$ is given by the signature of the permutation
	$$ \begin{pmatrix}
	1&2&3&\dots&2m \\
	i_1&j_1&i_2&\dots&j_{m}
	\end{pmatrix} .$$

	Expanding the Pfaffian $\Pf[\mathbf{B}(x_i, x_j)]_{i,j=1}^m$ and using the definition of $\delta'$, $b_m$ collapses into  a sum of $j$-fold integrals for $j\leqslant m$. In the kernel $\mathbf{B}$, the term $\delta'$ appears only in the component $\mathbf{B}_{22}$, so that when expanding  $\Pf[\mathbf{B}(x_i, x_j)]_{i,j=1}^m$ according to \eqref{eq:defPfaffian2} all terms have degree at most $\lfloor m/2\rfloor$ in $\delta'$. Hence we do not obtain  $j$-fold integral terms for $j<m-\lfloor m/2\rfloor$, and we can we decompose $b_m$ as 
	$$ \frac{(-1)^m}{m!} b_m
	= \sum_{k=\lceil m/2\rceil}^m b_m^k,  $$
	where $ b_m^j $ denotes the $j$-fold integral term in the expansion. 
	
	The key fact towards the proof of the proposition is that 
	\begin{equation} 
	\sum_{i=0}^k b_{k+i}^k=\frac{(-1)^k}{k!} a_k.
	\label{eq:keyidentity}
	\end{equation}
	This corresponds to the fact that $k$-fold integral terms match in the expansion of both sides of \eqref{eq:equivalentpfaffians}. 
	Then we can write 
	\begin{equation}
	\sum_{m=0}^{\infty} \frac{(-1)^m}{m!} b_m
	= \sum_{m=0}^{\infty} \sum_{k=\lceil m/2\rceil}^m b_m^k
	= \sum_{k=0}^{\infty} \sum_{i=0}^k b_{k+i}^k
	= \sum_{k=0}^{\infty} \frac{(-1)^k}{k!} a_k,
	\label{eq:fubinireorder}
	\end{equation}
	which is what we wanted to prove. In order to reorder terms in the third equality in \eqref{eq:fubinireorder} (using Fubini's theorem), we must check that the double series $b_{k+i}^k$ is absolutely convergent. Using Lemma \ref{lem:hadamard}, it is enough to check that
	$$ \sum_{k=0}^{\infty}\sum_{i=0}^k \frac{1}{k!} \binom k i 2^{k-1} (2k)^{k/2} C^k,$$
	is absolutely convergent, which is clearly the case ($C$ is some constant depending on  $A(x,y)$).

	The rest of this proof is devoted to the proof of \eqref{eq:keyidentity}. We will actually prove that 
	$$b_{k+i}^k  =  \frac{(-1)^k}{k!} \binom{k}{i} 2^{k-i}(-1)^i a_k.$$
	from which \eqref{eq:keyidentity} is deduced easily by summing over $i$:
	$$ \sum_{i=0}^k b_{k+i}^k = \frac{(-1)^k}{k!}  \sum_{i=0}^{k}\binom{k}{i} 2^{k-i}(-1)^i a_k = \frac{(-1)^k}{k!} a_k.$$

	The matrix $\mathbf{B}$ can be seen as the sum of a smooth matrix with good decay properties and another matrix whose entries are generalized functions of the type $\delta'(x_i, x_j)$. The following minor summation formula will be useful:
	\begin{lemma}[\cite{stembridge1990nonintersecting} Lemma 4.2] Let $M, N$ be skew symmetric matrices of size $n$ for some even $n$. Then,
		$$\Pf[M+N] = \sum_{I\subset \lbrace 1, \dots, n\rbrace; \#I \text{ even}} (-1)^{\vert I \vert -\#I/2} \Pf[M_I]\Pf[N_{I^c}],$$
		where $\#I$ denotes the cardinality of $I$, $\vert I\vert$ the sum of its elements, and $I^c$ its complement in $\lbrace 1, \dots, n\rbrace$. For $I\subset\lbrace 1, \dots, n\rbrace$, $M_I$ denotes the matrix $M$ restricted to rows and columns indexed by elements of $I$.
		\label{lem:minorsummation}
	\end{lemma}

	Let us fix some $k\geqslant 1$ and compute $b_{k+i}^{k}$ for $i=1$ to $k$.  It is instructive to consider first the case $i=1$. 
	\item{\textbf{Computation of $b_{k+1}^{k}$:}}
	By definition  $b_{k+1}^{k}$ corresponds to the terms of degree $1$ in $\delta'$ in the Pfaffian expansion of
	$$\frac{(-1)^{k+1}}{(k+1)!} \Pf\left[ \begin{pmatrix}
	\mathbf{A}_{11}(x_i,x_j) &  2\mathbf{A}_{12}(x_i, x_j) \\
	2\mathbf{A}_{21}(x_i,x_j) & 4\mathbf{A}_{22}(x_i,x_j)
	\end{pmatrix} + \begin{pmatrix}
	0&  0 \\
	0 & \delta'(x_i, x_j)
	\end{pmatrix} \right]_{i,j=1}^{k+1}$$
	after taking the integral over $x_1, \dots, x_{k+1}$. Using Lemma \ref{lem:minorsummation}, we can write 
	\begin{equation}
	b_{k+1}^{k} = \frac{(-1)^{k+1}}{(k+1)!}\int\mathrm{d}x_1\dots \int \mathrm{d}x_{k+1} \sum_{1\leqslant i<j\leqslant k+1}\delta'(x_i, x_j) (-1)^{2i+2j-1} \Pf\left[ \mathbf{A}_{k+1}^{\widehat{2i},\widehat{2j}} \right],
	\label{eq:termdegree1}
	\end{equation}
	where $\mathbf{A}_{n}^{\widehat{2i},\widehat{2j}}$ denotes the $2(n-1)\times 2(n-1)$ dimensional matrix which is the result of removing columns and rows $2i$ and $2j$ from
	$$\begin{pmatrix}
	\mathbf{A}_{11}(x_a,x_b) &  2\mathbf{A}_{12}(x_a, x_b) \\
	2\mathbf{A}_{21}(x_a,x_b) & 4\mathbf{A}_{22}(x_a,x_b)
	\end{pmatrix}_{a,b=1}^{n}.$$
	By using operations on rows and columns, we find that all terms in the summation in the R.H.S. of  \eqref{eq:termdegree1} are equal modulo a permutation of the $x_i$. Since we integrate the $x_i$ over a symmetric domain, these permutations are inconsequential so that 
	$$ b_{k+1}^{k} = -\frac{(-1)^{k+1}k(k+1)}{2(k+1)!} \int\mathrm{d}x_1\dots \int \mathrm{d}x_{k+1}\ \  \delta'(x_k, x_{k+1})  \Pf\left[ \mathbf{A}_{k+1}^{\widehat{2k},\widehat{2(k+1)}} \right].$$
	\begin{lemma}
		The action of $\delta'$ is such that
		\begin{equation*}
		\int\mathrm{d}x_1\dots\int  \mathrm{d}x_{k+1} \ \ \delta'(x_k, x_{k+1}) \ \  \Pf\left[ \mathbf{A}_{k+1}^{\widehat{2k},\widehat{2(k+1)}} \right]  =   -2^k a_k.
		\end{equation*}
		\label{lem:actionofdelta'}
	\end{lemma}
	\begin{proof}
		We will use the shorthand notation $\mathbf{M}_{k+1}$ for the matrix $\mathbf{A}_{k+1}^{\widehat{2k},\widehat{2(k+1)}}$, and 
		$ \mathbf{M}_{k}$ for the matrix formed from the matrix $\mathbf{A}_{k+1}^{\widehat{2k},\widehat{2(k+1)}}$ after swapping the column and row of indices $2k-1$ and $2k$. The notation $\mathbf{M}_{j}$ for $j=k, k+1$ means that the column and row where the variable $x_{j}$ appears is on the rightmost/bottommost position.
		We have $\Pf[\mathbf{M}_{k}] = -\Pf[\mathbf{A}_{k+1}^{\widehat{2k},\widehat{2(k+1)}}]$ and $\Pf[\mathbf{M}_{k+1}] = \Pf[\mathbf{A}_{k+1}^{\widehat{2k}\widehat{2(k+1)}}]$.
		Then,
		\begin{multline*}
		\int \mathrm{d}x_1\dots\int \mathrm{d}x_{k+1}\ \ \delta'(x_k, x_{k+1})  \Pf\left[ \mathbf{A}_{k+1}^{\widehat{2k},\widehat{2(k+1)}} \right]  \\=  \int\mathrm{d}x_1\dots\int \mathrm{d}x_{k} \ \ \Big(\partial_{x_k}\Pf\left[ \mathbf{M}_{k} \right] \big\vert_{x_{k+1}=x_k} +\partial_{x_{k+1}}\Pf\left[ \mathbf{M}_{k+1} \right]\big\vert_{x_{k+1}=x_k}\Big).
		\end{multline*}
		Moreover, expanding the Pfaffian as in \eqref{eq:defPfaffian1} or \eqref{eq:defPfaffian2}, we see that
		$$ \big(\partial_{x_k}\Pf\left[ \mathbf{M}_{k} \right]\big)\Big\vert_{x_{k+1}=x_k} = \Pf\left[\left(\partial_{x_k} \mathbf{M}_{k}\right)\big\vert_{x_{k+1}=x_k} \right] $$
		and
		$$ \big(\partial_{x_{k+1}}\Pf\left[ \mathbf{M}_{k+1} \right]\big)\Big\vert_{x_{k+1}=x_k} = \Pf\left[\left(\partial_{x_{k+1}} \mathbf{M}_{k+1}\right)\big\vert_{x_{k+1}=x_k} \right].$$
		
		Using the differential relations $\partial_{x_{k}} A_{11}(x_1, x_{k}) =  - A_{12}(x_1, x_{k})$, $\partial_{x_{k}} A_{11}( x_{k}, x_1) =  - A_{21}( x_{k}, x_1)$ and
		$\partial_{x_{k}} A_{21}(x_1, x_{k}) = -A_{22}(x_1, x_k)$, $\partial_{x_{k}} A_{12}( x_{k},x_1) = -A_{22}( x_k, x_1)$, we see that
		$$ \Pf\left[\left(\partial_{x_k} \mathbf{M}_{k}\right)\big\vert_{x_{k+1}=x_k} \right] = - 2^{k-1}\Pf[\mathbf{A}(x_i, x_j)]_{i,j=1}^{k},$$
		and  we obtain similarly that
		$$ \Pf\left[\left(\partial_{x_{k+1}} \mathbf{M}_{k+1}\right)\big\vert_{x_{k+1}=x_k} \right] = - 2^{k-1}\Pf[\mathbf{A}(x_i, x_j)]_{i,j=1}^{k}.$$
		Thus,
		\begin{multline*}
		\int\mathrm{d}x_1\dots \int \mathrm{d}x_{k+1} \ \  \delta'(x_k, x_{k+1})  \Pf\left[ \mathbf{A}_{k+1}^{\widehat{2k},\widehat{2(k+1)}} \right]  = \\ 
		-2^k \int\mathrm{d}x_1\dots \int \mathrm{d}x_{k} \ \ \Pf[\mathbf{A}(x_i, x_j)]_{i,j=1}^{k} = -2^k a_k. \qedhere
		\end{multline*}
	\end{proof}
	Thus, Lemma \ref{lem:actionofdelta'} shows that 
	$$  b_{k+1}^k = \frac{(-1)^{k+1}}{k!}k2^{k-1}  a_k.$$
	
	\item{\textbf{Computation of $b_{k+i}^{k}$ in general:}}
	By definition $b_{k+i}^{k}$ corresponds to  the terms of degree $i$ in $\delta'$ in the Pfaffian expansion of
	$$\frac{(-1)^{k+i}}{(k+i)!} \Pf\left[ \begin{pmatrix}
	\mathbf{A}_{11}(x_i,x_j) &  2\mathbf{A}_{12}(x_i, x_j) \\
	2\mathbf{A}_{21}(x_i,x_j) & 4\mathbf{A}_{22}(x_i,x_j)
	\end{pmatrix} + \begin{pmatrix}
	0&  0 \\
	0 & \delta'(x_i, x_j)
	\end{pmatrix} \right]_{i,j=1}^{k+i}$$
	after  taking the integral over $x_1, \dots, x_{k+i}$. Using Lemma \ref{lem:minorsummation}, 
	\begin{multline}
	b_{k+i}^{k} = \frac{(-1)^{k+i}}{(k+i)!}\int\mathrm{d}x_1\dots \int\mathrm{d}x_{k+i}  \\ 
	\sum_{1\leqslant j_1 < \dots <j_{2i}\leqslant k+i}\Pf[\delta'(x_{j_r}, x_{j_s})]_{r,s=1}^{2i} (-1)^{2j_1+\dots+2j_{2i}-i} \Pf\left[ \mathbf{A}_{k+i}^{\widehat{2j_1}, \dots, \widehat{2j_{2i}}} \right],
	\label{eq:termdegreei}
	\end{multline}
	where $\mathbf{A}_{n}^{\widehat{2j_1},\dots,  \widehat{2j_i}}$ is the $2(n-i)\times 2(n-i)$ dimensional matrix which is the result of removing the columns and rows indexed by $2j_1, 2j_2 \dots, 2j_{i}$ from
	$$\begin{pmatrix}
	\mathbf{A}_{11}(x_a,x_b) &  2\mathbf{A}_{12}(x_a, x_b) \\
	2\mathbf{A}_{21}(x_a,x_b) & 4\mathbf{A}_{22}(x_a,x_b)
	\end{pmatrix}_{a,b=1}^{n}.$$
	Again, all terms give an equal contribution after integrating over the $x_i$, so that 
	\begin{equation*} 
	b_{k+i}^{k} = (-1)^i \frac{(-1)^{k+i}}{(k+i)!}\binom{k+i}{2i} \int\mathrm{d}x_1\dots \int \mathrm{d}x_{k+i} 
	\Pf[\delta'(x_r, x_s)]_{r,s=k-i+1}^{k+i}  \Pf\left[ \mathbf{A}_{k+i}^{\widehat{2(k-i+1)}, \dots, \widehat{2(k+i)}} \right].
	\end{equation*}
	The Pfaffian $\Pf[\delta'(x_r, x_s)]_{r,s=k-i+1}^{k+i}$ can be expanded into products of $\delta'$ acting on disjoint pairs of variables. Each term in the expansion will give the same contribution after we have integrated over the $x_j$s, in the sense of the following.
	\begin{lemma}
		For any partition into pairs of the set $\lbrace k-i+1, \dots, k+i\rbrace$, in the form
		$\alpha = \big\lbrace (r_1, s_1), \dots, (r_i, s_i) \big\rbrace $ with $r_1< \dots < r_i$ and $r_j<s_j$, we have
		\begin{multline}
		\int \mathrm{d}x_1\dots \int \mathrm{d}x_{k+i}\ \ \delta'(x_{r_1}, x_{s_1}) \dots \delta'(x_{r_i}, x_{s_i})  \Pf\left[ \mathbf{A}_{k+i}^{\widehat{2(k-i+1)}, \dots, \widehat{2(k+i)}} \right] \\
		= \varepsilon(\alpha) \int\mathrm{d}x_1\dots \int\mathrm{d}x_{k+i}\ \  
		\prod_{j=1}^i   \delta'(x_{k-i+2j-1}, x_{k-i+2j})  \Pf\left[ \mathbf{A}_{k+i}^{\widehat{2(k-i+1)}, \dots, \widehat{2(k+i)}} \right].
		\label{eq:signatruresappear}
		\end{multline}
		\label{lem:equivterms}
	\end{lemma}
	\begin{proof}
		In order to match both sides of \eqref{eq:signatruresappear}, we make a change of variables so that the $\delta'$ factors exactly match. This does not change the value of the integral. Then we have  to exchange rows/columns in the Pfaffian. Each time that we exchange two adjacent rows/columns\footnote{Recall that in Pfaffian calculus, one always exchanges rows and columns simultaneously, to preserve the antisymmetry of the matrix.}, we factor out a $(-1)$,   hence the signature prefactor in the R.H.S of \eqref{eq:signatruresappear}.
	\end{proof}
	\begin{lemma}
		The action of $\delta'$ is such that
		\begin{equation}
		\int\mathrm{d}x_1\dots \int\mathrm{d}x_{k+i} \ \ \prod_{j=1}^i   \delta'(x_{k-i+2j-1}, x_{k-i+2j}) \Pf\left[ \mathbf{A}_{k+i}^{\widehat{2(k-i+1)},\dots, \widehat{2(k+i)}} \right]
		=(-1)^i 2^k  a_k.
		\end{equation}
		\label{lem:actionofseveraldelta'}
	\end{lemma}
	\begin{proof} 
		We adapt the proof of Lemma \ref{lem:actionofdelta'}. 
		For $j_1<\dots < j_i$ such that
		$$j_1\in\lbrace k-i+1, k-i+2 \rbrace, j_2\in\lbrace k-i+3, k-i+4 \rbrace, \ \dots\ , j_i\in\lbrace k+i-1, k+i \rbrace$$
		we denote by $\mathbf{M}_{j_1, \dots, j_i}$ the matrix obtained from $\mathbf{A}_{k+i}^{\widehat{2(k-i+1)},\dots, \widehat{2(k+i)}}$ where for all $r$, one has moved the column and row where $x_{j_r}$ appears to position $2(k-i+r)$. With this definition, if $j_r$ is even then the column and row where $x_{j_r}$ appears do not move, and if $j_r$ is odd then the column and row where $x_{j_r}$ appears is swapped with the right/bottom neighbor.
		Hence we have that
		$$ \Pf[\mathbf{A}_{k+i}^{\widehat{k-i+1},\dots, \widehat{k+i}}]  = (-1)^{j_1+\dots+ j_i}\Pf[\mathbf{M}_{j_1,\dots, j_i }].$$
		By the definition of $\delta'$,
		\begin{multline}
		\int \mathrm{d}x_1\dots\int \mathrm{d}x_{k+i}\ \ \prod_{j=1}^i   \delta'(x_{k-i+2j-1}, x_{k-i+2j}) \Pf\left[ \mathbf{A}_{k+i}^{\widehat{2(k-i+1)},\dots, \widehat{2(k+i)}} \right]\\
		= \int\mathrm{d}x_1\dots \int\mathrm{d}x_{k} \ \ \prod_{j=1}^i (\partial_{x_{k-i+2j}} - \partial_{x_{k-i+2j-1}})  \Pf\left[ \mathbf{A}_{k+i}^{\widehat{2(k-i+1)},\dots, \widehat{2(k+i)}} \right] \Big\vert_{x_{k+i}=\dots=x_{k+1}=x_k}.
		\label{eq:productsofpartial}
		\end{multline}
		We can develop the product of partial derivatives so that
		\begin{multline*}
		\eqref{eq:productsofpartial} = \int\mathrm{d}x_1\dots \int\mathrm{d}x_{k} \ \ \sum_{ j_1, \dots , j_i} \prod_{\ell=1}^i \big( (-1)^{j_{\ell}}  \partial_{x_{j_{\ell}}} \big) \Pf\left[ \mathbf{A}_{k+i}^{\widehat{2(k-i+1},\dots, \widehat{2(k+i)}} \right] \Big\vert_{x_{k+i}=\dots=x_{k+1}=x_k}\\
		=  \int \mathrm{d}x_1\dots \int \mathrm{d}x_{k}\ \ \sum_{j_1,  \dots,  j_i}  \partial_{x_{j_1}} \dots \partial_{x_{j_i}} \Pf\left[ \mathbf{M}_{j_1, \dots, j_i}\right] \Big\vert_{x_{k+i}=\dots=x_{k+1}=x_k},
		\end{multline*}
		where the sums runs over all sequences $j_1<\dots < j_i$ such that $j_1\in\lbrace k-i+1, k-i+2 \rbrace$,  $j_2\in\lbrace k-i+3, k-i+4 \rbrace$, etc.
		For any such sequence $j_1, \dots, j_i$, the relations between coefficients in the matrix $\mathbf{A}$ implies that
		$$ \partial_{x_{j_1}} \dots \partial_{x_{j_i}} \Pf\left[ \mathbf{M}_{j_1, \dots, j_i} \right] \Big\vert_{x_{k+i}=\dots=x_{k+1}=x_k} = (-1)^i 2^{k-i} \Pf[\mathbf{A}(x_r, x_s)]_{r,s=1}^{k}.$$
		Thus, we find that
		\begin{equation*}
		\eqref{eq:productsofpartial} =   (-1)^{i}  \int \mathrm{d}x_1\dots \int\mathrm{d}x_{k}\ \ 2^{k} \Pf[\mathbf{A}(x_r, x_s)]_{r,s=1}^{k}  =  (-1)^{i} 2^k a_k. \qedhere
		\end{equation*} 
	\end{proof}
	There are $ \frac{(2i)!}{2^i i!}$ terms in the expansion of the Pfaffian of the $2i\times 2i$ matrix $\big(\delta'(x_r, x_s)\big)_{r,s=k-i+1}^{k+i}$.  Combining  Lemmas \ref{lem:actionofseveraldelta'} and  \ref{lem:equivterms}, we have 
	\begin{equation*} b_{k+i}^k  = \frac{(-1)^{k+i}}{(k+i)!}\binom{k+i}{2i} \frac{(2i)!}{2^i i!}  2^k a_k = \\  \frac{(-1)^k}{k!} \binom{k}{i} 2^{k-i}(-1)^i a_k .\qedhere
	\end{equation*}
\end{proof}
Proposition \ref{prop:equivalentpfaffians} suggests that if the kernel $K^{\rm exp}$ converges to $K^{\infty}$ under some scalings, then the Fredholm Pfaffian of $K^{\rm exp}$ should converge to $F_{\rm GSE}$. However, the kernel $K^{\infty}$ is distribution valued, and we do not know an appropriate notion of convergence to distribution-valued kernels. The following provides an approximative version of Proposition \ref{prop:equivalentpfaffians} that provides  a notion of convergence sufficient for the Fredholm Pfaffian to converge.
\begin{proposition}
	Let $(A_n(x,y))_{n\in \mathbb{Z}_{\geqslant 0}}$ be a family of smooth and antisymmetric kernels such that for all $r,s\in \Z_{\geqslant 0}$,  $\partial_x^r\partial_y^sA_n(x,y)$ has exponential decay at infinity (in any direction of $\R^2$), uniformly for all $n$.
	Let
	$$  \mathbf{A}_n(x,y):= \begin{pmatrix} A_n(x,y)& -\partial_yA_n(x,y)\\
	-\partial_xA_n(x,y) &\partial_x\partial_y A_n(x,y)
	\end{pmatrix}$$
	and
	$$  \mathbf{B}_n(x,y):= \begin{pmatrix} A_n(x,y)& -2\partial_yA_n(x,y)\\
	-2\partial_xA_n(x,y) &4\partial_x\partial_y A_n(x,y) + \delta'_n(x,y)
	\end{pmatrix}$$
	be two families of kernels $(\R)^2 \to \skeww(\R)$ such that
	\begin{enumerate}
		\item The family of kernels
		$$ \begin{pmatrix} A_n(x,y)& -2\partial_yA_n(x,y)\\
		-2\partial_xA_n(x,y) &4\partial_x\partial_y A_n(x,y)
		\end{pmatrix} $$
		satisfy, uniformly in $n$, the hypotheses of Lemma \ref{lem:hadamard}.
		\item The kernel
		$\mathbf{A}_n(x,y)$ converges pointwise to
		$$\mathbf{A}:= \begin{pmatrix} A(x,y)& -\partial_yA(x,y)\\
		-\partial_xA(x,y) &\partial_x\partial_y A(x,y)
		\end{pmatrix},$$
		a smooth kernel  $\R^2 \to \skeww(\R)$ (necessarily satisfying the hypotheses of Lemma \ref{lem:hadamard}).
		\item For any smooth function $f:\mathbb{R}^2\to \R$ such that $\partial_x^r\partial_y^s f(x,y)$ have exponential decay at infinity for all $r,s$, we have
		\begin{equation*}
		\bigg\vert  \iint  \delta'_n(x,y)f(x,y)\mathrm{d}x\mathrm{d}y\ - \int \Big( \partial_y f(x,y) - \partial_x f(x,y)\Big)\Big\vert_{x=y} \mathrm{d}x\bigg\vert 
		=:  c_n(f) \xrightarrow[n\to\infty]{} 0.
		\end{equation*}
	\end{enumerate}
	Then for all $x\in \R$,
	$$ \Pf[J-\mathbf{B}_n]_{\mathbb{L}^2(x, \infty)} \xrightarrow[n\to\infty]{} \Pf[J-\mathbf{A}]_{\mathbb{L}^2(x, \infty)}.$$
	\label{prop:approxequivalentpfaffians}
\end{proposition}
\begin{proof}
	We will follow the proof of Proposition \ref{prop:equivalentpfaffians}, and make approximations when necessary. Let 
	$$b_m(n) =  \int\mathrm{d}x_1\dots\int \mathrm{d}x_m\ \   \Pf[\mathbf{B}_n(x_i, x_j)]_{i,j=1}^m, $$
	$$ a_m(n) =  \int \mathrm{d}x_1\dots\int\mathrm{d}x_m\ \  \Pf[\mathbf{A}_n(x_i, x_j)]_{i,j=1}^m,$$
	and denote by
	$ b_m^k(n) $ the part of the expansion of $\frac{(-1)^m}{m!} b_m(n)$ involving only terms of degree $m-k$ in $\delta'_n$ (as in \eqref{eq:termdegreei}).  We need to justify the approximation
	$$\sum_{i=0}^k b_{k+i}^k(n) \approx \frac{(-1)^k}{k!} a_k(n),$$ and show that the error is summable and goes to zero. Let us replace $\mathbf{A}, \mathbf{B}$ by $\mathbf{A}_n, \mathbf{B}_n$ in the proof of Proposition \ref{prop:equivalentpfaffians}. The only part of the proof which breaks down is when we apply the definition of $\delta'$. However, using the hypothesis (3), we see that replacing $\delta'$ by $\delta'_n$ in Lemma \ref{lem:actionofseveraldelta'} results in an error bounded by
	$$ 2^i (c_n)^i\vert a_k(n) \vert $$
	for $i\geqslant 1$ and the error is zero when $i=0$.
	Let  $\tilde{b}_{k+i}^k(n)$ be the quantity having the same expression as $b_{k+i}^k(n)$ where all occurrences of $\delta'_n$ are  replaced by $\delta'$. In other terms, $\tilde{b}_{k+i}^k(n)$  is given by \eqref{eq:termdegreei} where the kernel is $\mathbf{A}_n$ instead of $\mathbf{A}$. We have
	$$ \Big\vert  b_{k+i}^k(n) - \tilde{b}_{k+i}^k(n) \Big\vert \leqslant \frac{1}{k!} \binom{k}{i} 2^{k-i} 2^i (c_n)^i \vert a_k(n) \vert.$$
	On one hand,
	$$  \Pf[J-\mathbf{B_n}] = \sum_{m=0}^{\infty} \frac{(-1)^m}{m!} b_m(n)
	= \sum_{m=0}^{\infty} \sum_{k=\lfloor m/2\rfloor}^m b_m^k(n), $$
	while on the other hand,
	$$ \Pf[J-\mathbf{A_n}] =  \sum_{k=0}^{\infty} \frac{(-1)^k}{k!} a_k  = \sum_{k=0}^{\infty} \sum_{i=0}^k \tilde{b}_{k+i}^k = \sum_{m=0}^{\infty} \sum_{k=\lfloor m/2\rfloor}^m \tilde{b}_m^k(n),$$
	where we have used \eqref{eq:keyidentity} from the proof of Proposition \ref{prop:equivalentpfaffians} in the second equality and Fubini's Theorem in the third equality (which we can do by hypothesis (1) and Lemma \ref{lem:hadamard}).
	Thus,
	\begin{multline*} \Big\vert \Pf[J-\mathbf{B_n}] - \Pf[J-\mathbf{A_n}] \Big\vert \leqslant \sum_{m=0}^{\infty} \sum_{k=\lfloor m/2\rfloor}^{m} \frac{1}{k!} \binom{k}{m-k} 2^{k} (c_n)^{m-k} \vert a_k(n) \vert\\ = \sum_{k=0}^{\infty} \sum_{i=1}^k \frac{1}{k!} \binom{k}{i} 2^{k} (c_n)^i\vert a_k(n) \vert= \sum_{k=0}^{\infty} \frac{2^k}{k!} \vert a_k(n) \vert \left(\left(1+c_n  \right)^k - 1 \right) ,\end{multline*}
	Using the bound
	$$ \left(\left(1+c_n  \right)^k - 1 \right)\leqslant k \log(1+c_n) \left(1+c_n \right)^k \leqslant k c_n\left(1+c_n  \right)^k, $$
	we find that
	$$ \Big\vert \Pf[J-\mathbf{B_n}] - \Pf[J-\mathbf{A_n}] \Big\vert \leqslant \sum_{k=0}^{\infty} \frac{2^k}{k!} \vert a_k(n) \vert kc_n \left(1+c_n \right)^k \xrightarrow[n\to\infty]{} 0 $$
	
	Moreover,  using the pointwise convergence of hypothesis (2) and the decay hypothesis (1) along with Lemma \ref{lem:hadamard}, dominated convergence shows that
	\begin{equation*}
	\Pf[J-\mathbf{A_n}] \xrightarrow[n\to \infty]{} \Pf[J-\mathbf{A}]. \qedhere
	\end{equation*}
\end{proof}

\subsection{Asymptotic analysis of the kernel $ K^{\rm exp}$}
\label{sec:pointwiseGSE}
Recall that we have assumed that $\alpha>1/2$ (as in the first part of Theorem \ref{theo:LPPdiagointro}). Let us first focus on pointwise convergence of $K^{\rm exp}(1,x;1,y)$ as $n=m$ goes to infinity. In the following, we write simply $K^{\rm exp}(x;y)$ instead of $K^{\rm exp}(1,x;1,y)$.  We expect that as $n$ goes to infinity, $H(n,n)$ is asymptotically equivalent to $h n$ for some constant $h$, and fluctuates around $hn$ in the $n^{1/3}$ scale. Thus, it is natural to study the correlation kernel at a point 
\begin{equation}
\mathfrak{r}(x,y): = \big(n h + n^{1/3} \sigma x, n h + n^{1/3} \sigma y\big),
\label{eq:defr}
\end{equation}
where the constants $h, \sigma$ will be fixed later.  We introduce the rescaled kernel
\begin{equation*} K^{\rm{exp}, n}(x,y) :=  \begin{pmatrix}
\frac{1}{(2\alpha-1)^2} K^{\rm exp}_{11}\Big(\mathfrak{r}(x,y)\Big) & \sigma n^{1/3}K^{\rm exp}_{12}\Big(\mathfrak{r}(x,y)\Big)\\
\sigma n^{1/3}K^{\rm exp}_{21}\Big(\mathfrak{r}(x,y)\Big) & \sigma^2 n^{2/3}(2\alpha-1)^2 K^{\rm exp}_{22}\Big(\mathfrak{r}(x,y)\Big)
\end{pmatrix}.
\end{equation*}
By a change of variables in the Fredholm Pfaffian expansion of $K^{\rm exp}$ and a conjugation of the kernel preserving the Pfaffian, we get using Proposition \ref{prop:kernelexponential} that 
$$ \PP\left( \frac{H(n,n)-hn}{\sigma n^{1/3}} \leqslant y \right) = \Pf[J- K^{\rm{exp}, n}]_{\mathbb{L}^2(y, \infty)}.$$
Moreover, we write
$$ K^{\rm{exp}, n}(x,y) = I^{\rm{exp}, n}(x,y)+R^{\rm{exp}, n}(x,y), $$
according to the decomposition made in Section \ref{sec:kernelexp}.
The kernel $I^{\rm exp, n}_{11}$ can be rewritten as
\begin{multline}
I^{\rm exp, n}_{11}(x,y) = \int_{\mathcal{C}_{1/4}^{\pi/3}}\dd z \int_{\mathcal{C}_{1/4}^{\pi/3}} \dd w \frac{z-w}{4zw(z+w)} \frac{(2z+2\alpha -1)(2w+2\alpha -1)}{(2\alpha-1)^2}  \\ \times 
\exp\Big[n(f(z)+f(w)) - n^{1/3}\sigma (xz+yw)\Big],
\label{eq:kernelK11}
\end{multline}
where
\begin{equation}
f(z) = -h z + \log(1+2z) -\log(1-2z).
\label{eq:deff}
\end{equation}
Setting $h=4$, we find  that $f'(0)=f''(0)=0$.  Setting $\sigma=2^{4/3}$, Taylor expansion of $f$ around zero yields
\begin{equation}
f(z) = \frac{\sigma^3}{3} z^3 + \mathcal{O}(z^4).
\label{eq:Taylorapprox1}
\end{equation}
Similarly, the  kernel $I^{\rm exp, n}_{12}$ can be rewritten  as
\begin{multline}
I^{\rm exp, n}_{12}(x,y)= \sigma n^{1/3} \int_{\mathcal{C}_{1/4}^{\pi/3}} \dd z\int_{\mathcal{C}_{0}^{\pi/3}}\dd w \frac{z-w}{2z(z+w)} \frac{2\alpha -1 + 2z}{2\alpha -1-2w} \\ \times 
\exp\Big[n(f(z)+f(w)) - n^{1/3}\sigma (zx+wy)\Big].
\label{eq:kernelK12}
\end{multline}
Here the assumption that $\alpha>1/2$ matters in the choice of contours, so that the poles for $w=(2\alpha -1)/2$ lie on the right of the contour.
The kernel $I^{\rm exp, n}_{22}$ can be rewritten as
\begin{multline*}
I^{\rm exp, n}_{22}(x,y) = \sigma^2 n^{2/3} \int_{\mathcal{C}_{1/4}^{\pi/3}}\dd z\int_{\mathcal{C}_{1/4}^{\pi/3}}\dd w \frac{z-w}{z+w}  \frac{(2\alpha-1)}{2\alpha -1 - 2z}\frac{(2\alpha-1)}{2\alpha -1 - 2w} \\ \times 
\exp\Big[n(f(z)+f(w)) - n^{1/3}\sigma (zx+wy)\Big].
\end{multline*}
The kernel $R^{\rm{exp}, n}$ is such that $R^{\rm{exp}, n}_{11}= R^{\rm{exp}, n}_{12}=R^{\rm{exp}, n}_{21}=0$ and
$$ R^{\rm exp, n}_{22}(x,y)= \sgn(y-x) \left(\sigma n^{1/3}\frac{2\alpha-1}{2}\right)^2 e^{-\vert x-y\vert \sigma n^{1/3} \frac{2\alpha-1}{2}}.$$
\begin{proposition} For $\alpha >1/2$ and $\sigma=2^{4/3}$, we have
	\begin{align*}
	I^{\rm exp, n}_{11}(x,y) \xrightarrow[n\to\infty]{}& K_{11}^{\rm GSE}(x,y)=\int_{\mathcal{C}_{1}^{\pi/3}} \dd z\int_{\mathcal{C}_{1}^{\pi/3}}\dd w  \frac{(z-w)e^{z^3/3 + w^3/3 - xz - yw }}{4zw(z+w)}, \\
	I^{\rm exp, n}_{12}(x,y)  \xrightarrow[n\to\infty]{}& 2K_{12}^{\rm GSE}(x,y) =   \int_{\mathcal{C}_{1}^{\pi/3}} \dd z\int_{\mathcal{C}_{1}^{\pi/3}}\dd w  \frac{(z-w) e^{z^3/3 + w^3/3 - xz -  yw }}{2z(z+w)} ,\\
	I^{\rm exp, n}_{22}(x,y)  \xrightarrow[n\to\infty]{}& 4K_{22}^{\rm GSE}(x,y) =   \int_{\mathcal{C}_{1}^{\pi/3}}\dd z\int_{\mathcal{C}_{1}^{\pi/3}}\dd w \frac{(z-w)e^{z^3/3 + w^3/3 - xz -  yw }}{z+w} .
	\end{align*}
	\label{prop:GSEpointwise}
\end{proposition}
\begin{proof}
	Let us provide a detailed proof for the asymptotics of $I^{\rm exp, n}_{11}$. The arguments are almost identical for $I^{\rm exp, n}_{12}$ and $I^{\rm exp, n}_{22}$, the only difference being the presence of poles that are harmless for the contour deformation that we will perform (see the end of the proof).
	
	We use Laplace's method on the contour integrals in \eqref{eq:kernelK11}. It is clear that the asymptotic behaviour of the integrand is dominated by the term $\exp\big[n(f(z) + f(w))\big]$.
	Let us study the behaviour of $\Real [f(z)]$ along the contour $\mathcal{C}_{0}^{\pi/3}$.
	\begin{lemma} The contour $\mathcal{C}_{0}^{\pi/3}$ is steep descent\footnote{In general, we say that a contour $\mathcal{C}$ is steep descent for a real valued function $g$ if $g$ attains a unique maximum on the contour and is monotone along both parts of the contours separated by this maximum.} for $\Re[f]$ in the sense that
		the functions $t\mapsto \Real[f(te^{\I\pi/3})]$ and $t\mapsto \Real[f(te^{-\I\pi/3})]$ are strictly decreasing for $t>0$. Moreover, for $ t>1$, $\Real[f(te^{\pm\I\pi/3})]<-2t+2$.
		\label{lem:steepdescent}
	\end{lemma}
	\begin{proof} We have
		$$\Real[f(te^{\I\pi/3})] = \Real[f(te^{-\I\pi/3})] = -2 t + \tanh^{-1}\left(\frac{2t}{1+4t^2}\right), $$
		so that
		$$\frac{\mathrm{d}}{\mathrm{d}t}\Real[f(te^{\pm\I\pi/3})] = -\frac{16 (t^2 + 2 t^4)}{1 + 4 t^2 + 16 t^4)}, $$
		which is always negative, and less than $-2$ for $t>1$.
	\end{proof}
	We would like to deform the integration contour in \eqref{eq:kernelK11} to be the steep-descent contour  $\mathcal{C}_0^{\pi/3}$. This is not  possible due to the pole at $0$, and hence we modify the contour $\mathcal{C}_0^{\pi/3}$ in a $n^{-1/3}$--neighbourhood of $0$ in order to avoid the pole. We call $\mathcal{C}_{0\not\in}$ the resulting contour, which is depicted in Figure \ref{fig:modifiedcontour}. For a fixed but large enough $n$, the integrand in  \eqref{eq:kernelK11} has exponential decay along the tails of the integration contour: this is because for large $n$, the behaviour is governed by $e^{nf(z)+ nf(w)}$ and we have that
	$$ \frac{\Re[f(z)]}{\vert z \vert} \longrightarrow -\infty \text{ for }z\to\infty e^{\I \pi/3}. $$
	Hence, we are allowed to deform the contour to be $\mathcal{C}_{0\not\in}$ and  write
	\begin{multline}
	I^{\rm exp, n}_{11}(x,y) = \int_{\mathcal{C}_{0\not\in}}\dd z \int_{\mathcal{C}_{0\not\in}} \dd w \frac{z-w}{4zw(z+w)} \frac{(2z+2\alpha -1)(2w+2\alpha -1)}{(2\alpha-1)^2}  \\ \times
	\exp\Big[n(f(z)+f(w)) - n^{1/3}\sigma (xz+yw)\Big].
	\label{eq:kernelK11_modifiedcontour}
	\end{multline}
	We recall that $x$ and $y$ here can be negative. We will estimate \eqref{eq:kernelK11_modifiedcontour} by cutting it into two parts: The contribution of the integrations in a neighbourhood of $0$ will converge to the desired limit, and the integrations along the rest of the contour will go to $0$ as $n$ goes to infinity.
	
	Let us denote  $\mathcal{C}_{0\not\in}^{\rm local}$ the intersection of the contour $\mathcal{C}_{0\not\in}$ with a ball of radius $\epsilon$ centered at $0$, and set  $\mathcal{C}_{0\not\in}^{\rm tail} := \mathcal{C}_{0\not\in}\setminus \mathcal{C}_{0\not\in}^{\rm local}.$
	We first  show that the R.H.S. of \eqref{eq:kernelK11_modifiedcontour} with integrations over $\mathcal{C}_{0\not\in}^{\rm tail}$ goes to zero as $n\to \infty $  for any fixed $\epsilon>0$.
	\begin{figure}
		\begin{tikzpicture}[scale=0.85]
		\draw[->] (-2,0) -- (4,0);
		\draw[->] (0,-3.7) -- (0,3.7);
		\draw[thick] (60:1.5) -- (60:3.3);
		\draw[dashed] (60:0) -- (60:1.5);
		\draw[thick] (-60:1.5) -- (-60:3.3);
		\draw[dashed] (-60:0) -- (-60:1.5);
		\draw[thick] (-60:1.5 ) arc(-60:60:1.5);
		\draw[->, >=stealth'] (3,0) arc(0:59:3);
		\draw node at (3,1.5) {$\pi/3$};
		\draw node at (2,-2) {$\mathcal{C}_{0\not\in}^{\rm local}$};
		\draw[<->, >=stealth'] (45:0.05) -- (45:1.45);
		\draw node at (0.85, 0.4) {\tiny{$n^{-1/3}$}};
		\draw[<->, >=stealth'] (-0.1, 0.1) -- (62:3.3);
		\draw (0.6,2) node{$\epsilon$};
		\draw(-0.2, -0.3) node{$0$};
		
		\begin{scope}[xshift=8cm]
		\draw[->] (-2,0) -- (4,0);
		\draw[->] (0,-3.7) -- (0,3.7);
		\draw[thick] (60:1.5) -- (60:3.5);
		\draw[dashed] (60:0) -- (60:1.5);
		\draw[thick] (-60:1.5) -- (-60:3.5);
		\draw[dashed] (-60:0) -- (-60:1.5);
		\draw[thick] (-60:1.5 ) arc(-60:60:1.5);
		\draw[->, >=stealth'] (3,0) arc(0:59:3);
		\draw node at (3,1.5) {$\pi/3$};
		\draw node at (2,-2) {$\mathcal{C}_{0\not\in}$};
		\draw[<->, >=stealth'] (45:0.05) -- (45:1.45);
		\draw node at (0.85, 0.4) {\tiny{$n^{-1/3}$}};
		\draw[thick, dotted] (60:3.5) -- (60:4.5);
		\draw[thick, dotted] (-60:3.5) -- (-60:4.5);
		\draw(-0.2, -0.3) node{$0$};
		\end{scope}
		\end{tikzpicture}
		\caption{Contours used in the proof of Proposition  \ref{prop:GSEpointwise}. Left: the contour $\mathcal{C}_{0\not\in}^{\rm local}$. Right: the contour $\mathcal{C}_{0\not\in}$.}
		\label{fig:modifiedcontour}
	\end{figure}
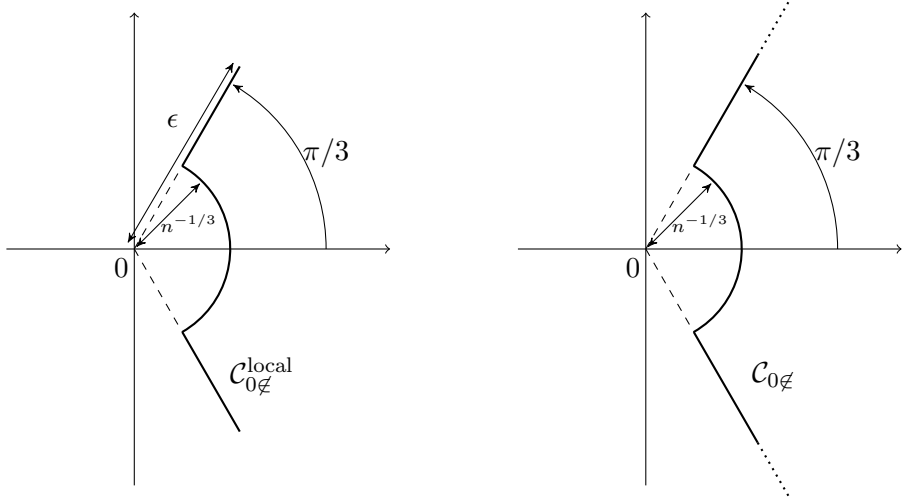
	Using Lemma \ref{lem:steepdescent}, we know that for any fixed $\epsilon >0$, there exists a constant $c>0$ such that for $t >\epsilon$,  $\Real[f(te^{\pm \I\pi/3})]<-ct$. Moreover, using both Taylor approximation \eqref{eq:Taylorapprox1} and Lemma \ref{lem:steepdescent}, we know that for any $z\in \mathcal{C}_{0\not\in}$, $\Real[nf(z)]$ is bounded uniformly in $n$ on the circular part of $\mathcal{C}_{0\not\in}$.
	This implies that there exists constants $n_0, C_0>0$ such that for $n>n_0$ and  $z, w \in \mathcal{C}_{0\not\in}^{\rm tail}$ we have
	$$\bigg\vert \exp\Big[n(f(z)+f(w)) - n^{1/3}\sigma (xz+yw)\Big]\bigg\vert <  C_0 e^{-cn \vert z\vert}.$$
	Using this estimate inside the integrand in \eqref{eq:kernelK11_modifiedcontour} and integrating over $\mathcal{C}_{0\not\in}^{\rm tail}$ yields a bound of the form  $Ce^{-cn\epsilon /2}$ which goes to $0$ as $n\to \infty$.
	
	Now  we estimate the contribution of the integration over $\mathcal{C}_{0\not\in}^{\rm local}$. We make the change of variables $z=n^{-1/3}\tilde z, w=n^{-1/3}\tilde w$ and approximate the integrand in \eqref{eq:kernelK11_modifiedcontour} by
	\begin{equation}
	\int_{n^{1/3}\mathcal{C}_{0\not\in}^{\rm local}} \int_{n^{1/3}\mathcal{C}_{0\not\in}^{\rm local}}\frac{\tilde z-\tilde w}{4\tilde z\tilde w(\tilde z+\tilde w)} e^{\frac{\sigma^3}{3}\tilde z^3 +\frac{\sigma^3}{3} \tilde w^3 -\sigma x\tilde z - \sigma y\tilde w }\mathrm{d}\tilde z\mathrm{d}\tilde w.
	\label{eq:maincontrib}
	\end{equation}
	We are making two approximations:
	(1) Replacing $nf(z)$ by  $\frac{\sigma^3}{3}\tilde{z}^3$, and likewise for $w$. Note that by Taylor approximation,  $nf(z) - \frac{\sigma^3}{3}\tilde{z}^3 = n\ \mathcal{O}(z^4) = n^{-1/3}\mathcal{O}(\tilde{z}^4)$.
	(2) Replacing $(2\tilde z+2\alpha -1)$ by $(2\alpha - 1)$ and likewise for $\tilde w$.
	
	In order to control the error made, we use the inequality $\vert e^x -1 \vert \leqslant\vert x\vert e^{\vert x \vert}$, so that the error can be bounded by $ E_1 + E_2 $ where
	\begin{multline}
	E_1 =  \int_{n^{1/3}\mathcal{C}_{0\not\in}^{\rm local}}\int_{n^{1/3}\mathcal{C}_{0\not\in}^{\rm local}} \vert n^{-1/3}\mathcal{O}(\tilde{z}^4) + n^{-1/3}\mathcal{O}(\tilde{w}^4) \vert    e^{\vert n^{-1/3}\mathcal{O}(\tilde{z}^4) + n^{-1/3}\mathcal{O}(\tilde{w}^4) \vert } \\ \times  \frac{z-w}{4zw(z+w)} e^{\frac{\sigma^3}{3}\tilde z^3 + \frac{\sigma^3}{3}\tilde w^3 -\sigma x\tilde z - \sigma y\tilde w } \frac{(2\tilde z+2\alpha -1)(2\tilde w+2\alpha -1)}{(2\alpha-1)^2} \mathrm{d}\tilde z\mathrm{d}\tilde w
	\label{eq:decoupage1}
	\end{multline}
	and
	\begin{equation}
	E_2 =  n^{-1/3}\int_{n^{1/3}\mathcal{C}_{0\not\in}^{\rm local}} \int_{n^{1/3}\mathcal{C}_{0\not\in}^{\rm local}} \frac{(\tilde z-\tilde w)e^{\frac{\sigma^3}{3}\tilde z^3 + \frac{\sigma^3}{3}\tilde w^3 -\sigma x\tilde z -\sigma  y\tilde w }}{4\tilde z\tilde w(\tilde z+\tilde w)}  \frac{4(\tilde z +\tilde w)}{2\alpha-1}\mathrm{d}\tilde z\mathrm{d}\tilde w.
	\label{eq:decoupage2}
	\end{equation}
	Notice that in the first integral in \eqref{eq:decoupage1},
	$$ e^{\vert n^{-1/3}\mathcal{O}(\tilde{z}^4) n^{-1/3}\mathcal{O}(\tilde{w}^4) \vert} =  e^{\vert n^{-1/3} \tilde{z} \mathcal{O}(\tilde{z}^3) + n^{-1/3} \tilde{w} \mathcal{O}(\tilde{w}^3)\vert } < e^{\epsilon \vert\mathcal{O}(\tilde{z}^3)+ \mathcal{O}(\tilde{w}^3)\vert}.$$
	Hence, choosing $\epsilon$ smaller than $\sigma^3/3$, the integrand in \eqref{eq:decoupage1} has exponential decay along the tails of the contour  $n^{1/3}\mathcal{C}_{0\not\in}^{\rm local}$, so that by dominated convergence, $E_1$ goes to $0$ as $n$ goes to infinity. It is also clear that $E_2$ goes to $0$ as $n$ goes to infinity.
	
	Finally, let us explain why the contour $n^{1/3}\mathcal{C}_{0\not\in}^{\rm local}$ in \eqref{eq:maincontrib} can be replaced by $\mathcal{C}_{1}^{\pi/3}$. Using the exponential decay of the integrand for $z, w$ in direction $ \pm\pi/3 $, we can extend  $n^{1/3}\mathcal{C}_{0\not\in}^{\rm local}$ to an infinite contour. Then, using Cauchy's theorem and again the exponential decay, this infinite contour can be freely deformed (without crossing any pole) to $\mathcal{C}_{1}^{\pi/3}$. We obtain,  after a change of variables $\sigma \tilde{z}= z $ and likewise for $w$,
	$$I^{\rm exp, n}_{11}(x,y) \xrightarrow[n\to\infty]{}  \frac{1}{(2\I\pi)^2} \int_{\mathcal{C}_{1}^{\pi/3}} \dd z\int_{\mathcal{C}_{1}^{\pi/3}} \dd w \frac{z-w}{4zw(z+w)} e^{z^3/3 + w^3/3 - xz -  yw } .$$
	
	When adapting the same method for $I^{\rm exp, n}_{12}$ and $I^{\rm exp, n}_{22}$, we could in principle encounter poles of $1/(2\alpha-1-2z)$ or $1/(2\alpha-1-2w)$. However, assuming that $\alpha>1/2$, these poles are in the positive real part region, and given the definitions of contours used in Section \ref{sec:deformedcontours}, we can always perform the above contour deformations.
\end{proof}

Now, we check that the kernel $R^{\rm exp, n}_{22}$ satisfies the hypothesis (3) of Proposition \ref{prop:approxequivalentpfaffians}.
\begin{proposition} Assume that $f(x,y)$ is a smooth function such that all its partial derivatives have exponential decay at infinity. Then,
	$$ \bigg\vert  \iint R^{\rm exp, n}_{22} f(x, y)\mathrm{d}x\mathrm{d}y - \iint \delta'(x, y) f(x, y)\mathrm{d}x\mathrm{d}y  \bigg\vert \xrightarrow[n\to\infty]{} 0 . $$
	\label{prop:asymptoticsRGSE}
\end{proposition}
\begin{proof}
	It is enough to prove the Proposition replacing $R^{\rm exp, n}_{22}$ by
	$$ n^2 \sgn(y-x)e^{\vert x-y\vert n}.$$
	We can write
	\begin{multline*}
	\iint n^2 \sgn(y-x)e^{\vert x-y\vert n}f(x, y)\mathrm{d}x\mathrm{d}y  = \\ \iint_{y>x}\big(f(x,y)-f(x,x)\big) n^2 e^{\vert x-y\vert n} - \iint_{y>x}\big(f(y,x)-f(x,x)\big) n^2 e^{\vert x-y\vert n}.\end{multline*}
	We make the change of variables $y=x+z/n$, and use the Taylor approximations
	\begin{align*}
	f(y,x)-f(x,x)&= \frac{z}{n}\partial_x f(x,x) + \frac{z^2}{2n^2}\partial_x^2 f(x,x) + \frac{z^3}{6n^{3}}\partial_x^3 f(x,x) + o(z^3n^{-3}),\\
	f(x,y)-f(x,x)&= \frac{z}{n}\partial_y f(x,x) +\frac{z^2}{2n^2}\partial_y^2 f(x,x) + \frac{z^3}{6n^{3}}\partial_y^3 f(x,x) + o(z^3n^{-3}).
	\end{align*}
	The quadratic terms will cancel, so that
	\begin{multline*} \iint n^2 \sgn(y-x)e^{\vert x-y\vert n}f(x, y)\mathrm{d}x\mathrm{d}y  =  \int \big(\partial_yf(x,x)-\partial_xf(x,x)\big)\mathrm{d}x\int ze^{-z}\dd z + \\\frac{1}{n^2}\int \big(\partial_y^3f(x,x)-\partial_x^3f(x,x)\big)\mathrm{d}x\int z^3e^{-z}\dd z + o(n^{-2}).
	\end{multline*}
	so that
	$$ \iint n^2 \sgn(y-x)e^{\vert x-y\vert n}f(x, y)\mathrm{d}x\mathrm{d}y =\int \big(\partial_yf(x,x)-\partial_xf(x,x)\big)\mathrm{d}x + \frac{C}{n^2} + o(n^{-2}).$$
\end{proof}

\subsection{Conclusion}
We have to prove that when $\alpha>1/2$,
\begin{equation}
\lim_{n\to\infty} \PP\left( H(n,n) <4n + 2^{4/3}n^{1/3} z \right)  = \Pf\big( J- K^{\rm GSE}\big)_{\mathbb{L}^2(z, \infty)}.
\label{eq:limitPfaffianFredholmGSE}
\end{equation}
Since
$$\PP\left( H(n,n) <4n + 2^{4/3}n^{1/3} z \right)=\Pf[J-K^{\rm exp, n}]_{\mathbb{L}^2(z, \infty)}$$
we are left with checking that the kernel $K^{\rm exp, n}$ satisfies the assumptions in Proposition \ref{prop:approxequivalentpfaffians}: Hypothesis (3) is proved by Proposition \ref{prop:asymptoticsRGSE}. Hypothesis (2) is proved by Proposition \ref{prop:GSEpointwise} (with $\mathbf{B}_n$ being $K^{\rm exp}$). Note that $K^{\rm GSE}(x,y)$ (defined in Lemma \ref{def:GSEdistribution}) does not have exponentail decay in all directions, but since we compute the Fredholm Pfaffian on $\mathbb{L}^2(z, \infty)$, we can replace $K^{\rm GSE}(x,y)$ by $\mathds{1}_{x,y\geqslant z}K^{\rm GSE}(x,y)$ so that exponential decay holds in all directions. Hypothesis (1) is proved by the following
\begin{lemma} Assume $\alpha>1/2$ and
	let $z\in \R$ be fixed. There exist a constant $c_0>0$ and for all $r,s\in \Z_{\geqslant 0}$, there exist positive constants $C_{r,s}, n_0$ such that for $n>n_0$ and $x, y>z$,
	$$ \Big\vert \partial_x^r\partial_y^s I^{\rm exp, n}_{11}(x,y)\Big\vert < C_{r,s} e^{-c_0 x - c_0 y}.$$
	In particular, there exist $C>0$ such that 
	$$\Big\vert I^{\rm exp, n}_{11}(x,y)\Big\vert < C e^{-c_0 x - c_0 y}, \ \ \Big\vert  I^{\rm exp, n}_{12}(x,y)\Big\vert < C  e^{-c_0 x - c_0 y},\ \
	\Big\vert  I^{\rm exp, n}_{22}(x,y)\Big\vert < C  e^{-c_0 x - c_0 y}. $$
	\label{lem:exponentialbound}
\end{lemma}
\begin{proof}
	We start with the expression for $I^{\rm exp, n}_{11}$ in  \eqref{eq:kernelK11_modifiedcontour}. We have already seen in the proof of Proposition \ref{prop:GSEpointwise} that using both Taylor approximation \eqref{eq:Taylorapprox1} and Lemma \ref{lem:steepdescent}, $\Re[nf(z)]$ is bounded by some constant $C_f$ uniformly in $n$ for $z\in \mathcal{C}_{0\not\in}$. Hence, there exists a constant $C_f$ such that
	$$ \vert \eqref{eq:kernelK11_modifiedcontour} \vert < \int_{\mathcal{C}_{0\not\in}}\dd z \int_{\mathcal{C}_{0\not\in}} \dd w \frac{(z-w)(2z+2\alpha -1)(2w+2\alpha -1)}{4zw(z+w)} 
	e^{2 C_f - n^{1/3}\sigma (xz+yw)}. $$
	Now make the change of variables $\tilde{z} = n^{1/3}z$ and likewise for $w$, so that
	\begin{equation*}
	\vert \eqref{eq:kernelK11_modifiedcontour} \vert < \int_{n^{1/3}\mathcal{C}_{0\not\in}}\mathrm{d}\tilde{z} \int_{n^{1/3}\mathcal{C}_{0\not\in}} \mathrm{d}\tilde{w} \frac{\tilde{z}-\tilde{w}}{4\tilde{z}\tilde{w}(\tilde{z}+\tilde{w})}  (2n^{-1/3}\tilde{z}+2\alpha -1)(2n^{-1/3}\tilde{w}+2\alpha -1)
	e^{2 C_f - \sigma (x\tilde{z}+y\tilde{w})}.
	\end{equation*}
	Since the real part of $z$ and $ w$ stays above 1/2 along the contour $n^{1/3}\mathcal{C}_{0\not\in}$, we get that for $x, y>0$
	$$ \vert \eqref{eq:kernelK11_modifiedcontour} \vert < C \exp\big(-c_0 x - c_0 y \big) $$
	for some constants $c_0, C>0$. Moreover, the estimates used in Proposition \ref{prop:GSEpointwise} show that for $x$ and $y$ in a compact subset, the kernel is bounded uniformly in $n$, so that there exists $C_{00}>0$ such that for all $x, y>z$,
	$$\Big\vert I^{\rm exp, n}_{11}(x,y)\Big\vert  < C_{00} \exp\big(-c_0 x - c_0 y \big).$$

	Applying derivatives in $x$ or $y$ to $I^{\rm exp, n}_{11}(x,y)$ corresponds to multiply the integrand by powers of $z$ and $w$. Hence, the above proof also shows that for $r,s\in \Z_{\geqslant 0}$, there exists $C_{rs}>0$ such that
	\begin{equation*}\Big\vert \partial_x^r \partial_y^s I^{\rm exp, n}_{11}(x,y)\Big\vert  < C_{rs} \exp\big(-c_0 x - c_0 y \big).\qedhere\end{equation*}
\end{proof}

\section{Asymptotic analysis in other regimes: GUE, GOE, crossovers}
\label{sec:asymptoticsvarious}

In this Section,  we prove GUE and GOE and Gaussian  asympotics and discuss the phase transition. In these regimes, the phenomenon of coalescence of points in the limiting point process --  which introduces subtleties in the asymptotic analysis of Section \ref{sec:GSEasymptotics} -- is not present, and the proofs follow a more standard approach.

\subsection{Phase transition}
\label{sec:mainasymptoticresults}

When we make $\alpha$ vary in $(0, \infty)$, the fluctuations of $H(n,n)$ obey  a phase transition already observed in \cite{baik2001asymptotics} in the geometric case. Let us explain why such a transition occurs. When $\alpha$ goes to infinity, the weights on the diagonal go to $0$.  It is then  clear that the last passage path to $(n,n)$  will avoid the diagonal. The total passage time will be the passage time from $(1,0)$ to $(n, n-1)$, which has the same law as $H(n-1, n-1)$ when $\alpha=1$.  Thus, the fluctuations should be the same when $\alpha$ equal $1$ or infinity, and by interpolation,  for any $\alpha\in(1, +\infty)$. The first part of Theorem \ref{theo:LPPawaydiagointro} shows that this is in fact the case for $\alpha\in (1/2, \infty)$. When $\alpha$ is very small however,  we expect that the last passage path will stay close to the diagonal in order to collect most of the large weights on the diagonal. Hence, the last passage time will be the sum of $\mathcal{O}(n)$ random variables, and hence fluctuate on the $n^{1/2}$ scale with Gaussian fluctuations. The critical value of $\alpha$ separating these two very different fluctuation regimes occurs  at  $1/2$.

Let us give an argument showing why the critical value is $1/2$, using a symmetry of the Pfaffian Schur processes.
\begin{lemma}
	The half-space last passage time ${H}(n,n)$ when the weights are $\mathcal{E}(\alpha)$ on the diagonal and $\mathcal{E}(1)$ away from the diagonal,   has the same distribution as the half-space last passage time $\tilde{H}(n,n)$ when the weights are $\mathcal{E}(\alpha)$ on the first row and $\mathcal{E}(1)$ everywhere else.
\end{lemma}
\begin{proof}
	Apply Proposition \ref{prop:diagorow} to the Pfaffian Schur process with single variable specializations considered in Proposition \ref{prop:SchurLPPcorrelations}, and take $q\to 1$.
\end{proof}
Let us study where the transition should occur in the latter $\tilde H$  model.
We have the equality in law
\begin{equation}
\tilde{H}(n,n)  \overset{(d)}{=} \max_{x\in\Z_{>0}} \left\lbrace \sum_{i=1}^x w_{1i} + \bar{H}( n-1, n-x)\right\rbrace ,
\label{eq:maximization1}
\end{equation}
where $\bar{H}(n-1, n-x)$ is the half-space last passage time from point $(n,n)$ to point  $(x, 2)$. It is independent of the $w_{1i}$, and has the same law as $H(n-1, n-x)$ when $\alpha=1$. For $n$ going to infinity with  $x=(1-\kappa) n$, $\kappa\in[0,1]$, we know from Theorem \ref{theo:LPPawaydiagointro} (which is independent from the present discussion) that $\bar{H}(n-1, n-x)/n$ goes to $(1+\sqrt{\kappa})^2$. Hence,
\begin{equation}
\lim_{n\to\infty} \frac{\tilde{H}(n,n)}{n} =  \max_{\kappa\in[0,1]} \left\lbrace \frac{1-\kappa}{\alpha} + (1+\sqrt{\kappa})^2 \right\rbrace =
\begin{cases} 4 &\mbox{ if }\alpha\in(1/2, 1),\\
\frac{1}{\alpha(1-\alpha)} &\mbox{ if }\alpha \leqslant 1/2.
\end{cases}
\label{eq:maximization2}
\end{equation}
This suggests that the transition happens when $\alpha=1/2$. The fact that  the maximum in \eqref{eq:maximization2} is attained for $\kappa=(\alpha/(1-\alpha))^2$ when $\alpha < 1/2$ even suggests that the fluctuations are Gaussian with variance $\frac{1-2\alpha}{\alpha^2 (1-\alpha)^2}$ on the scale $n^{1/2}$. The fluctuations of $\sum_{i=1}^x w_{1i}$ are dominant compared to those of $\bar{H}(n-x, n-1)$. We do not attempt to make this argument complete -- it would require proving some concentration bounds for $\bar{H}(n-1,n-x)$ when $x=(1-\kappa) n+ o(n)$ --  but we provide a computational derivation of the Gaussian asymptotics in Section \ref{sec:Gaussianasymptotics}. See also \cite{baik2005phase, arous2011current} for more details on similar probabilistic arguments.

In the case $\alpha=1/2$, the second part of Theorem \ref{theo:LPPdiagointro} shows that the fluctuations of $H(n,n)$ are on the scale $n^{1/3}$ and Tracy-Widom GOE distributed. It suggests that the value of $x$ that maximizes \eqref{eq:maximization1} is close to $0$  on the scale at most  $n^{2/3}$.

\subsection{GOE asymptotics}
\label{sec:GOEasymptotics}
In this Section, we prove the second part of Theorem \ref{theo:LPPdiagointro}. 
We consider a scaling of the kernel slightly different than what we used in Section \ref{sec:pointwiseGSE}. We set $\sigma=2^{4/3}$ as before, $\mathfrak{r}(x,y)$ as in \eqref{eq:defr}, and here 
\begin{equation*} K^{\rm{exp}, n}(x,y) := 
\begin{pmatrix}
\sigma^2 n^{2/3} K^{\rm exp}_{11}\Big(\mathfrak{r}(x,y)\Big) & \sigma n^{1/3}K^{\rm exp}_{12}\Big(\mathfrak{r}(x,y)\Big)\\
\sigma n^{1/3}K^{\rm exp}_{21}\Big(\mathfrak{r}(x,y)\Big) &  K^{\rm exp}_{22}\Big(\mathfrak{r}(x,y)\Big)
\end{pmatrix}.
\end{equation*}
We write
$ K^{\rm{exp}, n}(x,y) = I^{\rm{exp}, n}(x,y)+R^{\rm{exp}, n}(x,y), $
according to the decomposition made in Section \ref{sec:kernelexp}.
The pointwise asymptotics of Proposition \ref{prop:GSEpointwise} can be readily adapted when $\alpha=1/2$.
\begin{proposition} For $\alpha =1/2$, we have
	\begin{align*}
	I^{\rm exp, n}_{11}(x,y) \xrightarrow[n\to\infty]{}&  \int_{\mathcal{C}_{1}^{\pi/3}} \dd z\int_{\mathcal{C}_{1}^{\pi/3}}\dd w  \frac{z-w}{z+w} e^{z^3/3 + w^3/3 - xz - yw },\\
	I^{\rm exp, n}_{12}(x,y)  \xrightarrow[n\to\infty]{}&   \int_{\mathcal{C}_{1}^{\pi/3}} \dd z\int_{\mathcal{C}_{-1/2}^{\pi/3}}\dd w  \frac{w-z}{2w(z+w)} e^{z^3/3 + w^3/3 - xz -  yw },\\
	I^{\rm exp, n}_{22}(x,y)  \xrightarrow[n\to\infty]{}&   \int_{\mathcal{C}_{1}^{\pi/3}}\dd z\int_{\mathcal{C}_{1}^{\pi/3}}\dd w \frac{z-w}{4zw(z+w)} e^{z^3/3 + w^3/3 - xz -  yw }.
	\end{align*}
	\label{prop:GOEpointwise}
\end{proposition}
\begin{proof}
	We adapt the asymptotic analysis from Proposition \ref{prop:GSEpointwise}.
	For $I^{\rm exp, n}_{11}$ the arguments are strictly identical.  For  $I^{\rm exp, n}_{12}$, we have a pole for the variable $w$ in $0$ (coming from the factor $ \frac{1}{2\alpha-1-2w}$ in \eqref{eq:kernelK12}) which must stay on the right of the contour. Hence, we cannot use the contour $\mathcal{C}_{0\not\in} $ as for $I^{\rm exp, n}_{11}$ but another modification of the contour $ \mathcal{C}_0^{\pi/3}$ in a $n^{-1/3}$--neighbourhood of $0$ such that the contour passes to the left of $0$. We denote by $\mathcal{C}_{0\in}$ this new contour -- see Figure \ref{fig:modifiedcontour2}.
	Instead of working with a formula similar to \eqref{eq:kernelK11_modifiedcontour}, we start with
	\begin{multline}
	I^{\rm exp, n}_{12}(x,y)= \sigma n^{1/3} \int_{\mathcal{C}_{0\in}}\dd z\int_{\mathcal{C}_{0\in}}\dd w \frac{z-w}{2z(z+w)}\frac{2\alpha -1 + 2z}{2\alpha -1 - 2w} \\ \times
	\exp\Big[n(f(z)+f(w)) - n^{1/3}\sigma (zx+wy)\Big],
	\label{eq:kernelK12modifiedcontour}
	\end{multline}
	where $f$ is defined in \eqref{eq:deff}. 
	The arguments used to perform the asymptotic analysis of \eqref{eq:kernelK11_modifiedcontour} can be readily adapted to  \eqref{eq:kernelK12modifiedcontour}.
	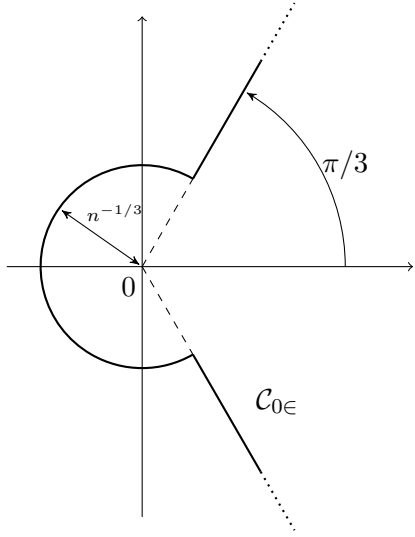
\begin{figure}
		\begin{tikzpicture}[scale=0.9]
		\draw[->] (-2,0) -- (4,0);
		\draw[->] (0,-3.7) -- (0,3.7);
		\draw[thick] (60:1.5) -- (60:3.5);
		\draw[dashed] (60:0) -- (60:1.5);
		\draw[thick] (-60:1.5) -- (-60:3.5);
		\draw[dashed] (-60:0) -- (-60:1.5);
		\draw[thick] (60:1.5 ) arc(60:300:1.5);
		\draw[->, >=stealth'] (3,0) arc(0:59:3);
		\draw node at (3,1.5) {$\pi/3$};
		\draw node at (2,-2) {$\mathcal{C}_{0\in}$};
		\draw[<->, >=stealth'] (145:0.05) -- (145:1.45);
		\draw node at (-0.4, 0.8) {\tiny{$n^{-1/3}$}};
		\draw[thick, dotted] (60:3.5) -- (60:4.5);
		\draw[thick, dotted] (-60:3.5) -- (-60:4.5);
		\draw(-0.2, -0.3) node{$0$};
		\end{tikzpicture}
		\caption{The contour $\mathcal{C}_{0\in}$ used in the end of the proof of Proposition  \ref{prop:GOEpointwise}, for the asymptotics of $K^{\rm exp}_{12}$.}
		\label{fig:modifiedcontour2}
	\end{figure}
	This explains the choice of $\mathcal{C}_{-1/2}^{\pi/3}$ as the contour for $w$ in $I_{12}$ in the limit. For $I^{\rm exp, n}_{22}$ however, thanks to the deformation of contours performed in Section \ref{sec:deformedcontours}, the presence of poles for $z$ and $w$ at $0$ has no influence in the computation of the limit, and the asymptotic analysis is very similar as in Proposition \ref{prop:GSEpointwise}.
\end{proof}
\begin{proposition}
	Assume $\alpha=1/2$. Then, we have $\bar{R}^{\rm exp, n}_{11}(x,y) = \bar{R}^{\rm exp, n}_{12}(x,y)=\bar{R}^{\rm exp, n}_{21}(x,y)=0$, and for any $x, y\in\R$,
	$$ \bar{R}^{\rm exp, n}_{22}(x,y) \xrightarrow[n\to\infty]{} -\int_{\mathcal{C}_{1}^{\pi/3}} e^{z^3/3-zx}\frac{\dd z}{4z}+  \int_{\mathcal{C}_{1}^{\pi/3}} e^{z^3/3-zy}\frac{\dd z}{4z} -\frac{\sgn{(x-y)}}{4}.$$
\end{proposition}
\begin{proof}
	The fact that $\bar{R}^{\rm exp, n}_{11}(x,y) = \bar{R}^{\rm exp, n}_{12}(x,y)=\bar{R}^{\rm exp, n}_{21}(x,y)=0$ is just by definition of $\bar{R}^{\rm exp, n}$ in Section \ref{sec:kernelexp}.  Regarding $\bar{R}^{\rm exp, n}_{22}$, when $x>y$,
	\begin{equation}
	\bar{R}_{22}^{\rm exp, n}(x ; y)  = -\int_{\mathcal{C}_{1/4}^{\pi/3}}  \frac{e^{n f(z)-x\sigma n^{1/3}z}}{4z}\dd z
	+   \int_{\mathcal{C}_{1/4}^{\pi/3}}  \frac{e^{nf(z)-y\sigma n^{1/3}z}}{4z}\dd z 
	+ \int_{\mathcal{C}_{1/4}^{\pi/3}} \frac{1}{2z} e^{-\sigma n^{1/3}\vert x-y\vert z}\dd z\ \  -\frac{1}{4}, 
	\end{equation}
	where $f$ is defined in the proof of Proposition \ref{prop:GSEpointwise}, and $\bar{R}_{22}^{\rm exp}(i,x ; j,y)$ is antisymmetric. Similar arguments as in Proposition  \ref{prop:GSEpointwise} show that
	$$
	\int_{\mathcal{C}_{1/4}^{\pi/3}}  \frac{e^{n f(z)-x\sigma n^{1/3}z}}{4z}\dd z  \xrightarrow[n\to\infty]{} \int_{\mathcal{C}_{1}^{\pi/3}} e^{z^3/3-zx}\frac{\dd z}{4z}, $$
	so that the Proposition is proved, using the antisymmetry of $\bar{R}_{22}$.
\end{proof}
At this point, we have shown that when $\alpha=1/2$,
\begin{equation*}
\Pf\Big(K^{\rm exp, n}\big( x_i, x_j \big) \Big)_{i,j=1}^k \xrightarrow[q\to 1]{} \Pf\Big({K}^{\rm GOE}\big( x_i, x_j \big) \Big)_{i,j=1}^k.
\end{equation*}
In order to conclude that the Fredholm Pfaffian likewise has the desired limit, we need a control on the entries of the kernel $K^{\rm{exp}, n}$, in order to apply dominated convergence.
\begin{lemma}
	Assume $\alpha=1/2$ and
	let $a\in \R$ be fixed. There exist positive constants $C, c_1, m$ for $n>m$ and $x, y>a$,
	\begin{equation*}
	\Big\vert n^{2/3} K^{\rm exp, n}_{11}(x,y)\Big\vert < C e^{-c_1 x - c_1 y},
	\Big\vert n^{1/3} K^{\rm exp, n}_{12}(x,y)\Big\vert < C e^{-c_1 x },
	\Big\vert  K^{\rm exp, n}_{22}(x,y)\Big\vert < C.
	\end{equation*}
	\label{lem:GOEbound}
\end{lemma}
\begin{proof}
	The proof is very similar to that of Lemma \ref{lem:exponentialbound}. The contour for $w$ in $K^{\rm exp}_{12}$ used in the asymptotic analysis is now $\mathcal{C}_{0\in}$  instead of $\mathcal{C}_{0\not\in}$, resulting in an absence of decay as $y \to\infty$. The kernel $I^{\rm exp, n}_{22}$ has exponential decay in $x$ and $y$ using arguments similar with Lemma \ref{lem:exponentialbound}. However, the bound on $K^{\rm exp, n}_{22}$ is only constant, due to the term $\sgn(x-y)$.
\end{proof}
The bounds from Lemma \ref{lem:GOEbound} are such that the hypotheses in Lemma \ref{lem:hadamard} are satisfied, and we conclude, applying dominated convergence in the Pfaffian series expansion,  that
$$ \lim_{n\to\infty} \PP\left( H(n,n) <4n + 2^{4/3}n^{1/3}x \right)  = \Pf\big( J- K^{\rm GOE}\big)_{\mathbb{L}^2(x, \infty)}. $$

\subsection{GUE asymptotics}
\label{sec:GUEasymptotics} In this Section, we prove Theorem \ref{theo:LPPawaydiagointro}. 
Fix a parameter $\kappa\in (0,1)$ and assume that  $m = \kappa n$. Again we expect that $H(n,m)/n$ converges to some constant $h$ depending on $\kappa$, with fluctuations on the $n^{1/3}$ scale.

The kernel $K^{\rm exp}_{11}$ can be rewritten as
\begin{multline}
K^{\rm exp, n}_{11}\Big(\mathfrak{r}(x,y)\Big) =\int_{\mathcal{C}_{1/4}^{\pi/3}}\dd z\int_{\mathcal{C}_{1/4}^{\pi/3}}\dd w \frac{z-w}{4zw(z+w)}\\(2z+2\alpha -1)(2w+2\alpha -1)
\exp\Big[n(f_{\kappa}(z)+f_{\kappa}(w)) - n^{1/3}\sigma (xz+yw)\Big],
\label{eq:kernelK11:N<n}
\end{multline}
where
\begin{equation}
f_{\kappa}(z) = -h z + \log(1+2z) -\kappa\log(1-2z).
\label{eq:deffGUE}
\end{equation}

It is convenient to parametrize the constant $\kappa$ by a constant $\theta \in (0,1/2)$, and set the values of $h$ and $\sigma$ such that
\begin{equation*}
\kappa = \left(\frac{1-2\theta}{1+2\theta}\right)^2, \ \ 
h = \frac{4}{(1+2\theta)^2}, \ \ 
\sigma =  \left(\frac{8}{(1+2\theta)^3} + \frac{8\kappa}{(1-2\theta)^3}\right)^{1/3}.
\end{equation*}
With this choice,  $f_{\kappa}'(\theta) = f_{\kappa}''(\theta)=0$ and by Taylor expansion,
\begin{equation}
f_{\kappa}(z) =  f_{\kappa}(\theta) + \frac{\sigma^3}{3}(z-\theta)^3 + \mathcal{O}\big((z-\theta)^4\big).
\label{eq:Taylorapprox:N<T1}
\end{equation}
The kernel $K^{\rm exp}_{22}$ can be rewritten as
$$K^{\rm exp}_{22}= I^{\rm exp}_{22}+\begin{cases}R^{\rm exp}_{22} &\mbox{ when }\alpha>1/2,\\ \hat{R}^{\rm exp}_{22} &\mbox{ when }\alpha<1/2,\\\bar{R}^{\rm exp}_{22}&\mbox{ when }\alpha=1/2,\end{cases}$$
where
\begin{multline*}
I^{\rm exp}_{22}\Big(\mathfrak{r}(x,y)\Big) =  \int_{\mathcal{C}_{1/4}^{\pi/3}}\dd z\int_{\mathcal{C}_{1/4}^{\pi/3}}\dd w \frac{z-w}{z+w} \frac{1}{2\alpha -1-2z} \frac{1}{2\alpha -1-2w} \\
\exp\Big[(n(g_{\kappa}(z)+g_{\kappa}(w)) - n^{1/3}\sigma (xz+yw)\Big] ,
\end{multline*}
with $ g_{\kappa}(z) = -f_{\kappa}(-z)$.
When $\alpha>1/2$,
\begin{equation*}
R^{\rm exp}_{22} \Big(\mathfrak{r}(x,y)\Big) = -\sgn(x-y)\int_{\mathcal{C}_{0}^{\pi/3}} \frac{e^{ n(g_{\kappa}(z)-f_{\kappa}(z))-\vert x-y\vert n^{1/3}\sigma z}}{(2\alpha-1-2z)(2\alpha-1+2z)} 2z\dd z.
\end{equation*}
When $\alpha<1/2$,
\begin{multline*}
\hat{R}^{\rm exp}_{22} \Big(\mathfrak{r}(x,y)\Big) =
-\sgn(x-y)\int_{\mathcal{C}_{0}^{\pi/3}} \frac{e^{ n(g_{\kappa}(z)-f_{\kappa}(z))-\vert x-y\vert n^{1/3}\sigma z}}{(2\alpha-1-2z)(2\alpha-1+2z)}2z\dd z\\
-e^{\frac{1-2\alpha}{2}\sigma n^{1/3}y} \int_{\mathcal{C}_{0}^{\pi/3}}e^{n\left(g_{\kappa}(z)+g_{\kappa}\left(\frac{1-2\alpha}{2}\right)\right)-x\sigma n^{1/3}z}\frac{\dd z}{2\alpha-1-2z} \\
+e^{\frac{1-2\alpha}{2}\sigma n^{1/3}x} \int_{\mathcal{C}_{0}^{\pi/3}}e^{n\left(g_{\kappa}(z)+g_{\kappa}\left(\frac{1-2\alpha}{2}\right)\right)-y\sigma n^{1/3}z}\frac{\dd z}{2\alpha-1-2z}.
\end{multline*}
when $\alpha=1/2$,
\begin{multline*}
\bar{R}^{\rm exp}_{22} \Big(\mathfrak{r}(x,y)\Big) =
\sgn(x-y)\int_{\mathcal{C}_{1/4}^{\pi/3}} e^{ n(g_{\kappa}(z)-f_{\kappa}(z))-\vert x-y\vert n^{1/3}\sigma z}\frac{\dd z}{2z}\\
+ \int_{\mathcal{C}_{1/4}^{\pi/3}}e^{n\left(g_{\kappa}(z)+g_{\kappa}\left(\frac{1-2\alpha}{2}\right)\right)-x\sigma n^{1/3}z}\frac{\dd z}{2z} 
- \int_{\mathcal{C}_{1/4}^{\pi/3}}e^{n\left(g_{\kappa}(z)+g_{\kappa}\left(\frac{1-2\alpha}{2}\right)\right)-y\sigma n^{1/3}z}\frac{\dd z}{2z} - \frac{\sgn(x-y)}{4}.
\end{multline*}
Since $g_{\kappa}'(-\theta)= g_{\kappa}''(-\theta)=0$, and $g_{\kappa}'''(-\theta) = f_{\kappa}'''(\theta)$,  Taylor expansion around $-\theta$ yields
$$ g_{\kappa}(z) = g_{\kappa}(-\theta) + \frac{\sigma^3}{3} (z+\theta)^3 + \mathcal{O}\big((z+\theta)^4\big). $$

Scale the kernel $K^{\rm exp}$ as
\begin{equation*} K^{\rm{exp}, n}(x,y) := 
\begin{pmatrix}
\dfrac{\sigma n^{1/3} K^{\rm exp}_{11}\Big(\mathfrak{r}(x,y)\Big)}{e^{ 2nf_{\kappa}(\theta) -\sigma n^{1/3} (x+y)\theta }}
& \sigma n^{1/3}K^{\rm exp}_{12}\Big(\mathfrak{r}(x,y)\Big)\\
\sigma n^{1/3}K^{\rm exp}_{21}\Big(\mathfrak{r}(x,y)\Big)
&\dfrac{ \sigma n^{1/3}  K^{\rm exp}_{22}\Big(\mathfrak{r}(x,y)\Big)}{e^{2ng_{\kappa}(-\theta) +\sigma n^{1/3} (x+y)\theta }}
\end{pmatrix}.
\end{equation*}
By a conjugation of the kernel preserving the Pfaffian,
the factor $e^{2nf_{\kappa}(\theta)}$ in $K^{\rm exp}_{11}$ cancels with $e^{2ng_{\kappa}(-\theta)}$ in $K^{\rm exp}_{22}$ (since $f_{\kappa}(\theta)= - g_{\kappa}(-\theta)$), and the factor $e^{-\sigma n^{1/3}(x\theta + y\theta) }$ in $K^{\rm exp}_{11}$ cancels with $e^{-\sigma n^{1/3}(-x\theta - y\theta)} $ in $K^{\rm exp}_{22}$.
This  implies that by a change of variables in the Pfaffian series expansion, 
$$ \PP\left(\frac{H(n,\kappa n) - hn}{\sigma n^{1/3}} \leqslant y\right)  = \Pf\big(J - K^{\rm exp, n} \big)_{\mathbb{L}^2(y, \infty)}.$$
\begin{proposition}
	We have that
	\begin{equation*}
	\sigma^2 n^{2/3}K^{\rm exp, n}_{11}(x,y) \xrightarrow[]{n\to\infty}    \int_{\mathcal{C}_{1}^{\pi/4}}\dd z \int_{\mathcal{C}_{1}^{\pi/4}}\dd w  \frac{z-w}{8\theta^3}  (2\theta+2\alpha-1)^2
	e^{\frac{z^3}{3}+ \frac{w^3}{3}-x z - y  w)},
	\end{equation*}
	so that in particular $K^{\rm exp, n}_{11}(x,y)\xrightarrow[n\to \infty]{} 0$.
	\label{prop:GUEpointwise11}
\end{proposition}
\begin{proof}
	We start with \eqref{eq:kernelK11:N<n}. Since for fixed $n$, the integrand has exponential decay  in the direction $e^{i\phi}$ for any $\phi\in (-\pi/2, \pi/2)$, we are allowed to deform the contour from $\mathcal{C}_{1/4}^{\pi/3}$ to $\mathcal{C}_{1/4}^{\pi/4}$. Then we deform the contour from $\mathcal{C}_{1/4}^{\pi/4}$ to  $\mathcal{C}_{\theta}^{\pi/4}$. This is valid as soon as we do not cross any pole during the deformation, which is the case when  $\theta\in (0, 1/2)$. The following shows that the contour $\mathcal{C}_{\theta}^{\pi/4}$ is steep-descent.
	\begin{lemma}For any $\theta\in (0,1/2)$, the functions $t\mapsto \Real[f_{\kappa}(\theta + t(1+i))]$ and $t\mapsto \Real[f_{\kappa}(\theta + t(1-i))]$ are strictly decreasing for $t>0$. Moreover, for $t>1$, $\frac{\mathrm{d}}{\mathrm{d}t} \Real[f_{\kappa}(\theta + t(1+i))]$ is uniformly bounded away from zero.
		\label{lem:steepdescentGUE}
	\end{lemma}
	\begin{proof}
		Straightforward calculations show that 
		\begin{multline*}
		\frac{\mathrm{d}}{\mathrm{d}t} \Real[f_{\kappa}(\theta + t(1+i))] =   -h(\theta + t)+ \frac 1 2 \log\big((1+2\theta + 2t)^2+(2t)^2 \big) - \frac{\kappa}{2}\log\big( (1-2\theta-2t)^2+(2t)^2\big).
		\end{multline*}
		is strictly negative for $t>0$ and $\theta\in (0, 1/2)$. Calculations are the same for the other branch (in direction $1-i$) of the contour.
	\end{proof}
	Using similar arguments as in Section \ref{sec:pointwiseGSE}, we estimate the integral in two parts.  We call $\mathcal{C}_{\theta}^{\rm local}$ the intersection between $\mathcal{C}_{\theta}^{\pi/4}$ with a ball of radius $\epsilon$ centred at $\theta$, and we write
	$$\mathcal{C}_{\theta}^{\pi/4}  = \mathcal{C}_{\theta}^{\rm local} \sqcup \mathcal{C}_{\theta}^{\rm tail}  .$$
	Using the steep-descent properties of Lemma \ref{lem:steepdescentGUE}, we can show that the integration over  $\mathcal{C}_{\theta}^{\rm tail}$ goes to $0$ exponentially fast as $n$ goes to infinity. We  are left with finding the asymptotic behavior of
	\begin{equation*}
	\int_{\mathcal{C}_{\theta}^{\rm local}}\int_{\mathcal{C}_{\theta}^{\rm local}} \frac{(z-w)(2z+2\alpha -1)(2w+2\alpha -1)}{4zw(z+w)}
	e^{n(f_{\kappa}(z)+f_{\kappa}(w)) - n^{1/3}\sigma (xz+yw)} \dd z\dd w.
	\end{equation*}
	We make the change of variables
	$ z=\theta+\tilde z n^{-1/3}$ and $w=\theta+\tilde w n^{-1/3}, $
	and find
	\begin{multline*}
	n^{-2/3} \int_{n^{1/3}\mathcal{C}_{\theta}^{\rm local}}\int_{n^{1/3}\mathcal{C}_{\theta}^{\rm local}} \frac{n^{-1/3}(\tilde z-\tilde w)}{8\theta^3}(2\theta+2\alpha -1)(2\theta+2\alpha -1)\\
	\exp\Big[ 2nf(\theta) -\sigma n^{1/3} (x+y)\Big]\exp\Big[\frac{\sigma^3}{3}\tilde z^3- \sigma (x \tilde z+y\tilde w)\Big] \mathrm{d}\tilde z\mathrm{d}\tilde w,
	\end{multline*}
	where the $ n^{-2/3}$ factor is due to the Jacobian of the change of variables. Finally, we can extend $n^{1/3}\mathcal{C}_{\theta}^{\rm local}$ to an infinite contour making a negligible error and then freely deform the contour to  $\mathcal{C}_{1}^{\pi/4}$.
\end{proof}
For $K^{\rm exp, n}_{22}$, the presence of the pole in $1/(2\alpha -1-2z)$ imposes the condition $ \alpha >1/2 - \theta$, or equivalently, $\alpha >\frac{\sqrt{\kappa}}{1+\sqrt{\kappa}}$.
\begin{proposition}
	For $\alpha >\frac{\sqrt{\kappa}}{1+\sqrt{\kappa}}$, we have that
	\begin{equation*}
	\sigma^2 n^{2/3}K^{\rm exp, n}_{22}(x,y) \xrightarrow[n\to\infty]{}  \int_{\mathcal{C}_{0}^{\pi/4}}\dd z \int_{\mathcal{C}_{0}^{\pi/4}}\dd w  \frac{z-w}{-2{\theta}}   \frac{e^{ \frac{z^3}{3}  + \frac{w^3}{3}  - xz-yw }}{(2\alpha-1+2{\theta})^2}
	,
	\end{equation*}
	so that in particular $K^{\rm exp, n}_{22}(x,y)\xrightarrow[n\to \infty]{} 0$.
	\label{prop:GUEpointwise22}
\end{proposition}
\begin{proof}
	Since $n>m=\kappa n$, all contours can be deformed to the vertical contours $\mathcal{C}_{0}^{\pi/2}$ or  $\mathcal{C}_{1/4}^{\pi/2}$ in the expressions for $I_{22}, R_{22}, \hat{R}_{22}$ and $\bar{R}_{22}$. This is because the quantity $e^{g_{\kappa}(z)} $ has enough decay along vertical contours as long as $\kappa<1$ and $n$ is large enough. By reversing the deformations of contours performed in Section \ref{sec:deformedcontours}, 
	\begin{equation*}
	K^{\rm exp}_{22}\Big(\mathfrak{r}(x,y)\Big) =  \int_{\mathcal{C}_{a_z}^{\pi/2}}\dd z\int_{\mathcal{C}_{a_w}^{\pi/2}}\dd w \frac{z-w}{z+w}   \frac{e^{n(g_{\kappa}(z)+g_{\kappa}(w)) - n^{1/3}\sigma (xz+yw)}}{(2\alpha -1-2z)(2\alpha -1-2w)},
	\end{equation*}
	where $a_z, a_w$ are chosen as any value in $\mathbb{R}_{<0}$ when $\alpha\geqslant 1/2$ and 
	any value in $(\frac{2\alpha-1}{2}, 0)$ when $\alpha<1/2$.
	Since we have assumed that $\theta\in(\frac{2\alpha-1}{2}, 0)$, we can freely deform the contours to the contour $\mathcal{C}_{\vert\eta}$ depicted on Figure \ref{fig:contourK22}, where the constant $\eta>0$ can be chosen as small as we want.
	\begin{figure}
		\begin{tikzpicture}[scale=0.8]
		\draw[->] (-4,0) -- (2,0);
		\draw[->] (0,-3.7) -- (0,3.7);
		\draw (-2, -0.1) -- (-2, 0.1);
		\draw(-2, 0.3) node{$\theta$};
		\draw(0.2, -0.4) node{$0$};
		\draw[<->, >=stealth'] (-0.45, 1) -- (0.05,1);
		\draw node at (-0.2, 0.8) {\tiny{$\eta$}};
		\draw[thick]  (-0.4, -3.6) -- (-0.4, -1.6) --(-2,0) -- (-0.4, 1.6) -- (-0.4, 3.6);
		\draw node at (-1,2) {$\mathcal{C}_{\vert \eta}$};
		\end{tikzpicture}
		\caption{The contour $\mathcal{C}_{\vert\eta}$ used in the proof of Proposition \ref{prop:GUEpointwise22}.}
		\label{fig:contourK22}
	\end{figure}
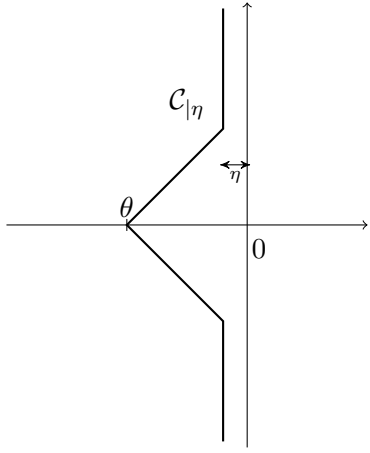
	Now, we estimate the behaviour of $\Re[g_{\kappa}]$ on the contour $\mathcal{C}_{\vert \eta}$.
	\begin{lemma}
		For any $\theta\in (0,1/2)$, the functions $t\mapsto \Real[f_{\kappa}(\theta - t(1+i))]$ and $t\mapsto \Real[f_{\kappa}(\theta - t(1-i))]$ are strictly increasing for $t>0$. Moreover, for $t>1$, $\frac{\mathrm{d}}{\mathrm{d}t} \Real[f_{\kappa}(\theta - t(1+i))]$ is uniformly bounded away from zero. In particular, the contour $\mathcal{C}_{-\theta}^{\pi/4}$ is steep descent for $\Re[g_{\kappa}]$.
		\label{lem:steepdescentGUE2}
	\end{lemma}
	\begin{proof}
		The calculations are analogous to the proof of Lemma \ref{lem:steepdescentGUE}.
	\end{proof}
	\begin{lemma}
		For $\eta>0$ small enough, the function $\Real[g_{\kappa}(-\eta+iy)]$ is decreasing for $y\in (c(\eta), +\infty)$ where $0<c(\eta)<\theta-\eta$, and increasing for $y\in (-c(\eta), -\infty)$.
		\label{lem:steepdescentGUEvert}
	\end{lemma}
	\begin{proof}
		We have 
		$$\frac{\mathrm{d}}{\mathrm{d}y}\Real\big[g_{\kappa}(-\eta + iy)\big] = \frac{(\kappa-1)\big(32y^3+ (1+2\eta)^2\big) + 8\eta}{2\big((1-2\eta)^2+4y^2\big)\big((1+2\eta)^2+4y^2\big)},$$
		from which the result follows readily because $\kappa<1$.
	\end{proof}
	Lemmas \ref{lem:steepdescentGUE2} and \ref{lem:steepdescentGUEvert} together imply that for $\eta$ small enough, the contour $\mathcal{C}_{\vert \eta}$ is steep descent for the function $\Real[g_{\kappa}]$. Then, an adaptation of the proof of Proposition \ref{prop:GUEpointwise11} concludes the proof.
\end{proof}
Turning to the kernel $K^{\rm exp}_{12}$, after the change of variables $ w=-\tilde{w}$, we have 
\begin{multline}
K^{\rm exp, n}_{12}(x,y) = \frac{\sigma^{-1} n^{-1/3}}{(2\I\pi)^2} \int_{\mathcal{C}_{1/4}^{\pi/3}}\dd z\int_{\mathcal{C}_{-1/4}^{2\pi/3}}\dd \tilde w \frac{z+\tilde w}{2z(z-\tilde w)} \frac{2\alpha-1+2z}{2\alpha-1+2\tilde w} \\ \times 
\exp\Big[n(f_{\kappa}(z)-f_{\kappa}(\tilde w)) - n^{1/3}\sigma (z-\tilde w)\Big],
\label{eq:kernelK12GUE}
\end{multline}
where $f_{\kappa}$ is defined in \eqref{eq:deffGUE}.
\begin{proposition}
	For $\alpha >\frac{\sqrt{\kappa}}{1+\sqrt{\kappa}}$,
	$$K^{\rm exp, n}_{12}(x,y) \xrightarrow[n\to\infty]{}  K_{\rm Ai}(x, y),$$
	where $K_{\rm Ai}$ is the Airy kernel defined in \eqref{eq:defairykernel}.
	\label{prop:GUEpointwise12}
\end{proposition}
\begin{proof}
	We have already seen in Lemma \ref{lem:steepdescentGUE} that the contour $\mathcal{C}_{\theta}^{\pi/4}$ is steep-descent for $\Real[f_{\kappa}(z)]$, and in  Lemma \ref{lem:steepdescentGUE2}  that $\mathcal{C}_{\theta}^{3\pi/4}$ is steep descent for $-\Real[f_{\kappa}(z)]$.
	Due to the term $1/(z-w)$ in \eqref{eq:kernelK12GUE}, we  cannot in principle use simultaneously the contour $\mathcal{C}_{\theta}^{\pi/4}$ for $z$ and $\mathcal{C}_{\theta}^{3\pi/4}$ for $w$. Hence, we deform the contours in a $n^{-1/3}$-neighbourhood of $\theta$. Let us denote $\mathcal{C}_{\theta\not\in}$ and $\mathcal{D}_{\theta\not\in}$ the modified contours, see Figure \ref{fig:modifiedcontoursGUE}.
	As in the proof of Proposition \ref{prop:GSEpointwise},    we  call $\mathcal{C}_{\theta\not\in}^{\rm local}$ and $\mathcal{D}_{\theta\not\in}^{\rm local}$ the intersection of $\mathcal{C}_{\theta\not\in}$ and $\mathcal{D}_{\theta\not\in}$ with  a ball of radius $\epsilon$ around $\theta$, as in Section \ref{sec:pointwiseGSE}.
	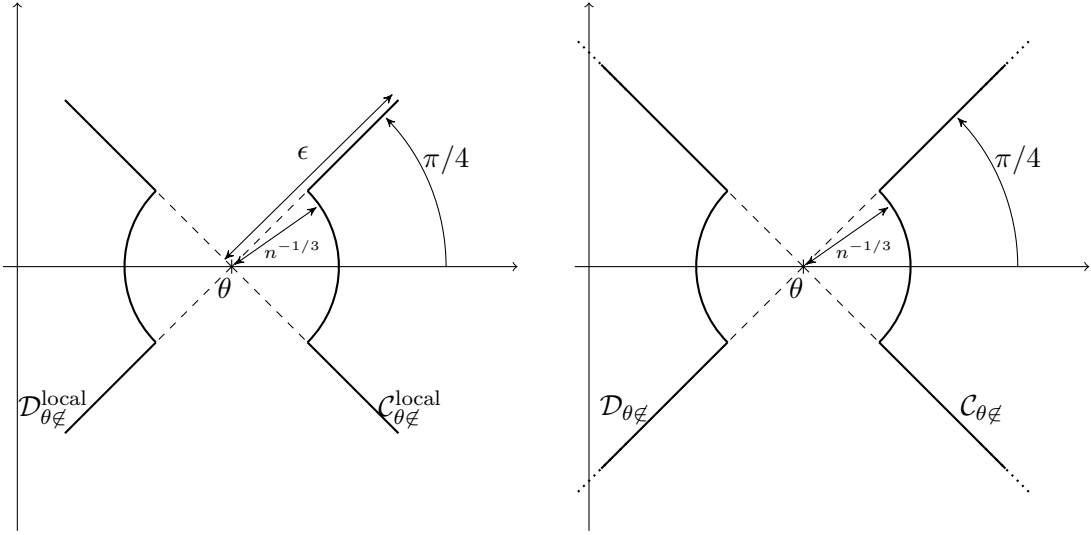
\begin{figure}
		\begin{tikzpicture}[scale=0.95]
		\draw (0, -0.1) -- (0, 0.1);
		\draw[->] (-3.2,0) -- (4,0);
		\draw[->] (-3,-3.7) -- (-3,3.7);
		\draw[thick] (45:1.5) -- (45:3.3);
		\draw[dashed] (45:0) -- (45:1.5);
		\draw[thick] (-45:1.5) -- (-45:3.3);
		\draw[dashed] (-45:0) -- (-45:1.5);
		\draw[thick] (-45:1.5 ) arc(-45:45:1.5);
		\draw[->, >=stealth'] (3,0) arc(0:44:3);
		\draw node at (3,1.5) {$\pi/4$};
		\draw node at (2.5,-2) {$\mathcal{C}_{\theta\not\in}^{\rm local}$};
		\draw[<->, >=stealth'] (35:0.05) -- (35:1.45);
		\draw node at (0.85, 0.25) {\tiny{$n^{-1/3}$}};
		\draw[<->, >=stealth'] (-0.1, 0.1) -- (47:3.3);
		\draw (1,1.6) node{$\epsilon$};
		\draw(-0.1, -0.3) node{$\theta$};

		\draw[thick] (135:1.5) -- (135:3.3);
		\draw[dashed] (135:0) -- (135:1.5);
		\draw[thick] (-135:1.5) -- (-135:3.3);
		\draw[dashed] (-135:0) -- (-135:1.5);
		\draw[thick] (135:1.5 ) arc(135:225:1.5);
		\draw node at (-2.5,-2) {$\mathcal{D}_{\theta\not\in}^{\rm local}$};

		\begin{scope}[xshift=8cm]
		\draw (0, -0.1) -- (0, 0.1);
		\draw[->] (-3.2,0) -- (4,0);
		\draw[->] (-3,-3.7) -- (-3,3.7);
		\draw[thick] (45:1.5) -- (45:4);
		\draw[dashed] (45:0) -- (45:1.5);
		\draw[thick] (-45:1.5) -- (-45:4);
		\draw[dashed] (-45:0) -- (-45:1.5);
		\draw[thick] (-45:1.5 ) arc(-45:45:1.5);
		\draw[->, >=stealth'] (3,0) arc(0:44:3);
		\draw node at (3,1.5) {$\pi/4$};
		\draw node at (2.5,-2) {$\mathcal{C}_{\theta\not\in}$};
		\draw[<->, >=stealth'] (35:0.05) -- (35:1.45);
		\draw node at (0.85, 0.25) {\tiny{$n^{-1/3}$}};
		\draw(-0.1, -0.3) node{$\theta$};

		\draw[thick] (135:1.5) -- (135:4);
		\draw[dashed] (135:0) -- (135:1.5);
		\draw[thick] (-135:1.5) -- (-135:4);
		\draw[dashed] (-135:0) -- (-135:1.5);
		\draw[thick] (135:1.5 ) arc(135:225:1.5);
		\draw node at (-2.5,-2) {$\mathcal{D}_{\theta\not\in}$};

		\draw[thick, dotted] (45:3.5) -- (45:4.5);
		\draw[thick, dotted] (-45:3.5) -- (-45:4.5);
		\draw[thick, dotted] (135:3.5) -- (135:4.5);
		\draw[thick, dotted] (-135:3.5) -- (-135:4.5);
		\end{scope}
		\end{tikzpicture}
		\caption{Contours used in the proof of Proposition  \ref{prop:GSEpointwise}. Left: the contours $\mathcal{C}_{\theta\not\in}^{\rm local}$ and $\mathcal{D}_{\theta\not\in}^{\rm local}$. Right: the contours $\mathcal{C}_{\theta\not\in}$ and $\mathcal{D}_{\theta\not\in}$.}
		\label{fig:modifiedcontoursGUE}
	\end{figure}
	Using the same arguments as in Section \ref{sec:pointwiseGSE}, we can first deform the contours in \eqref{eq:kernelK12GUE} to $\mathcal{C}_{\theta\not\in}$ and $\mathcal{D}_{\theta\not\in}$, and then approximate the integrals by integrations over $\mathcal{C}_{\theta\not\in}^{\rm local}$ and $\mathcal{D}_{\theta\not\in}^{\rm local}$. By making the change of variables
	$ z=\theta+\tilde z n^{-1/3}$ and $ w=\theta+\tilde w n^{-1/3}, $
	and approximating the integrand using the Taylor approximation \eqref{eq:Taylorapprox:N<T1} and straightforward pointwise limits, we are  left with
	\begin{equation*}
	n^{-2/3} \int_{\mathcal{C}_{1/4}^{\pi/4}}\dd\tilde z\int_{\mathcal{C}_{-1/4}^{3\pi/4}}\dd \tilde w \frac{e^{\frac{\sigma^3}{3} \tilde z^3 -\frac{\sigma^3}{3} \tilde w^3 -x\sigma \tilde z + y\sigma \tilde w}}{n^{-1/3}(\tilde z-\tilde w)}
	,
	\end{equation*}
	as desired. Note that the integration contours should be $n^{1/3}\big(\mathcal{C}_{0\not\in}^{\rm local} - \theta\big)$ and $n^{1/3}\big(\mathcal{D}_{0\not\in}^{\rm local} - \theta\big)$, but in the large $n$ limit, we can use $\mathcal{C}_{1}^{\pi/4}$ and $\mathcal{C}_{-1}^{3\pi/4}$ instead for the same reasons as in Proposition \ref{prop:GSEpointwise}, so that we recognize the Airy kernel.
\end{proof}
Let us denote the pointwise limit of $K^{\rm exp, n}$ by
$$ K^{\rm GUE}(x,y):=\begin{pmatrix}
0 & K_{\Ai}(x,y)\\
-K_{\Ai}(y,x) & 0
\end{pmatrix}.  $$
The pointwise convergence of the kernel (Propositions \ref{prop:GUEpointwise11},  \ref{prop:GUEpointwise22} and  \ref{prop:GUEpointwise12}) along with the fact that $K^{\rm exp, n}$ satisfies uniformly the hypotheses of Lemma \ref{lem:hadamard} (this is clear for $K^{\rm exp, n}_{11}$ and $K^{\rm exp, n}_{22}$ and it is proved  as in Lemma  \ref{lem:GOEbound} for $K^{\rm exp, n}_{12}$), shows by dominated convergence that
$$ \lim_{n\to\infty} \PP\left( H(n,\kappa n) <(1+\sqrt{\kappa})^2 n+\sigma n^{1/3} x \right)  = \Pf\big( J- K^{\rm GUE}\big)_{\mathbb{L}^2(x, \infty)}.$$
Finally,
\begin{align*}
\Pf\big( J- K^{\rm GUE}\big)_{\mathbb{L}^2(x, \infty)} &= \sqrt{\det\big(I +JK^{\rm GUE} \big)_{\mathbb{L}^2(x, \infty)}} \\
&=\sqrt{ \det\left( \begin{pmatrix}
	I & 0\\
	0 & I
	\end{pmatrix}  - \begin{pmatrix}
	K_{ \Ai} & 0 \\
	0 & K_{ \Ai}
	\end{pmatrix} \right)_{\mathbb{L}^2(x, \infty)}}\\
&= \det\big(I-K_{\Ai}\big)_{\mathbb{L}^2(x, \infty)} = F_{\rm GUE}(x).
\end{align*}
\begin{remark}
	We have assumed in the statement of Theorem \ref{theo:LPPawaydiagointro} that $m=\kappa n$ for some fixed $\kappa$. A stronger statement holds. When $m=\kappa n+sn^{2/3-\epsilon}$ for any $s\in \R$ and $\epsilon >0$, 
	$$\lim_{n\to\infty} \PP\left( \frac{H(n,\kappa n) -(1+\sqrt{\kappa})^2 n-2sn^{2/3-\epsilon}}{\sigma n^{1/3}} < x \right) = F_{\rm GUE}(x).$$
	This change  results in an additional factor $ e^{2sn^{2/3-\epsilon} w - sn^{2/3-\epsilon}\log(1-2w) }$
	in \eqref{eq:kernelK11:N<n}. After the change of variables $w=n^{-1/3}\tilde{w}$ in the proof of Propositions \ref{prop:GUEpointwise11}, \ref{prop:GUEpointwise22} and \ref{prop:GUEpointwise12}, the additional factor becomes $e^{2n^{-\epsilon} \tilde{w}}$ and does not contribute to the limit.
	\label{rem:varyingkappa}
\end{remark}

\subsection{Gaussian asymptotics}
\label{sec:Gaussianasymptotics}
In this Section, we prove the third part of Theorem \ref{theo:LPPdiagointro}. 
Assume that $\alpha<1/2$.  The factors $\frac{1}{2\alpha-1-2w}$ and $\frac{1}{2\alpha-1-2z}$ in the expressions of $K^{\rm exp}_{12}$ and $K^{\rm exp}_{22}$ in Proposition \ref{prop:kernelexponential} prevent us from deforming the contours to go through the critical point at zero as in Section \ref{sec:pointwiseGSE}. We already know that the LLN of $H(n,n)$ will be different, and so too will  the critical point (its position will coincide with the poles above-mentioned). We have already argued in Section \ref{sec:mainasymptoticresults} that
$$ \frac{H(n,n)}{n} \xrightarrow[n\to\infty]{} \frac{1}{\alpha(1-\alpha)} =:h(\alpha), $$
and we expect Gaussian fluctuations on the $n^{1/2}$ scale. Let $\mathfrak{r}_{\alpha}(x,y) = (n h_{\alpha} + n^{1/2} \sigma_{\alpha} x,n h_{\alpha} + n^{1/2} \sigma_{\alpha} y)$ for some constants $h_{\alpha}, \sigma_{\alpha}$ depending on $\alpha$. 
The kernel $K^{\rm exp}_{11}$ can be rewritten as
\begin{multline}
K^{\rm exp}_{11}\Big(\mathfrak{r}_{\alpha}(x,y)\Big) =\int_{\mathcal{C}_{1/4}^{\pi/3}}\dd z \int_{\mathcal{C}_{1/4}^{\pi/3}} \dd w \frac{z-w}{4zw(z+w)}\\ \times (2z+2\alpha -1)(2w+2\alpha -1)
e^{n(f_{\alpha}(z)+f_{\alpha}(w)) - n^{1/2}\sigma (xz+yw)},
\label{eq:kernelK11:Gaussian}
\end{multline}
where
$$ f_{\alpha}(z) = -h_{\alpha} z + \log(1+2z) -\log(1-2z).$$
There are  two critical points, $f_{\alpha}'\left(\frac{2\alpha-1}{2}\right)= f_{\alpha}'\left(\frac{1-2\alpha}{2}\right)=0$.  Setting $\theta=\frac{1-2\alpha}{2}$,  and $\sigma_{\alpha}>0$ such that $\sigma_{\alpha}^2=f_{\alpha}''\left(\theta \right) = \frac{1-2\alpha}{\alpha^2 (1-\alpha)^2},$
Taylor expansions of $f_{\alpha}$ around $\theta$ and $-\theta$ yield
\begin{align*}
f_{\alpha}(z) = f_{\alpha}\left(\theta \right) + \frac{\sigma_{\alpha}^2}{2} (z-\theta)^2 + \mathcal{O}\left((z-\theta)^3\right), \label{eq:TaylorapproxGaussian1}\\
f_{\alpha}(z) = f_{\alpha}\left(-\theta\right) - \frac{\sigma_{\alpha}^2}{2} (z+\theta)^2 + \mathcal{O}\left((z+\theta)^3\right).
\end{align*}
Let us rescale the kernel by setting
\begin{equation*} K^{\rm{exp}, n}(x,y) := 
\begin{pmatrix}
\frac{\sigma_{\alpha}^2 n K^{\rm exp}_{11}\Big(\mathfrak{r}_{\alpha}(x,y)\Big)}{e^{ 2nf_{\alpha}(\theta) -\sigma n^{1/2} (x+y)\theta }}
& \sigma_{\alpha} n^{1/2}K^{\rm exp}_{12}\Big(\mathfrak{r}_{\alpha}(x,y)\Big)\\
\sigma_{\alpha} n^{1/2}K^{\rm exp}_{21}\Big(\mathfrak{r}_{\alpha}(x,y)\Big)
&\frac{   K^{\rm exp}_{22}\Big(\mathfrak{r}_{\alpha}(x,y)\Big)}{e^{ 2nf_{\alpha}(-\theta) +\sigma_{\alpha} n^{1/2} (x+y)\theta }}
\end{pmatrix}.
\end{equation*}
We can conjugate the kernel without changing the Pfaffian so that factors $\exp \big[ 2 n f_{\alpha}(-\theta)-\sigma n^{1/3}(-x\theta -y\theta) \big]$ and $\exp\Big[ 2n f_{\alpha}(\theta) -\sigma n^{1/2} (x+y)\theta \Big]$ cancels out each other.
Moreover, a factor $\sigma_{\alpha} n^{1/2}$ will be absorbed by the Jacobian of the change of variables in the Fredholm Pfaffian expansion of $K^{\rm exp, n}$, so that
$$ \PP\left(\frac{H(n,\kappa n) - h(\alpha)n}{\sigma_{\alpha} n^{1/2}} \leqslant y\right)  = \Pf\big(J - K^{\rm exp, n} \big)_{\mathbb{L}^2(y, \infty)}.$$

Saddle-point analysis of $K^{\rm exp}_{11}$ around the critical point $ \theta = \frac{1-2\alpha}{2}$ shows
\begin{equation}
\sigma_{\alpha}^3 n^{1/2} K^{\rm exp, n}_{11}(x,y) \xrightarrow[n \to \infty]{} - \int_{\mathcal{C}_{1}^{\pi/3}}\dd z \int_{\mathcal{C}_{1}^{\pi/3}}\dd w  \frac{z-w}{8\theta^3} zw e^{z^2/2+ w^3/3- x  z-y w},
\label{eq:limitGaussian1}
\end{equation} 
Notice that since $\alpha<1/2$, $\theta>0$. Since the contours for $K^{\rm exp}_{11}$ in Proposition \ref{prop:kernelexponential} must stay in the positive real part region, we choose the critical point $\theta$.

The derivation of \eqref{eq:limitGaussian1} follows the same steps as in Proposition \ref{prop:GUEpointwise11}. The most important step is to find a steep-descent contour. The contour $\mathcal{C}_{\theta}^{\pi/4}$ used in in Proposition \ref{prop:GUEpointwise11} is not a good choice here, but it is easy to check by direct computation that  $\mathcal{C}_{\theta}^{\pi/3}$ is steep descent for $z\mapsto \Re[f_{\alpha}]$ for any value of $\alpha\in(0,1/2)$. 

Similarly for $K^{\rm exp}_{22}$, by a saddle-point analysis around the critical point $-\theta = \frac{2\alpha-1}{2}$ (the restriction on the contours of $K^{\rm exp}_{22}$ in Proposition \ref{prop:kernelexponential} impose now to chose the negative critical point), we get that
\begin{equation*}
\sigma_{\alpha} n^{1/2} K^{\rm exp, n}_{22}(x,y) \xrightarrow[n \to \infty]{}  \int_{\mathcal{C}_{1}^{\pi/6}}\dd z \int_{\mathcal{C}_{1}^{\pi/6}}\dd w  \frac{z-w}{-4{\theta^2}z w}  
e^{ -z^2/2 - w^2/2 - xz-yw }.
\end{equation*}
The angle is chosen as $\pi/6$ because the contour  $\mathcal{C}_{-\theta}^{\pi/6}$ is steep descent for the function $z\mapsto \Re[f_{\alpha}]$, which can be checked again by direct computation. Note that it is less obvious here that we can use vertical contours as in the proof of Proposition \ref{prop:GUEpointwise22}, but it is still possible because the integrand is oscillatory with $1/z$ decay.

In the case of $K^{\rm exp}_{12}$, we use the critical point $\theta$ for the variable $z$ and $-\theta$ for the variable $w$ and obtain 
\begin{equation*}
K^{\rm exp, n}_{12}(x,y) \xrightarrow[n \to \infty]{}  \int_{\mathcal{C}_{2}^{\pi/3}}\dd z \int_{\mathcal{C}_{-1}^{\pi/6}}\dd w \frac{1}{z+w}\frac{z}{-w} e^{ z^2/2 -w^2/2 -x z - y w}.
\end{equation*}
We recognize the Hermite kernel in the R.H.S. (see \cite{bleher2005integral}), which yields $K^{\rm exp, n}_{12}(x,y) \xrightarrow[n \to \infty]{} \frac{1}{\sqrt{2\pi}}e^{-\frac{x^2+y^2}{4}}.$

Finally, using exponential bounds for the kernel as in Lemmas \ref{lem:exponentialbound} and \ref{lem:GOEbound}, we obtain that
$$ \lim_{n\to\infty} \PP\left( \frac{H(n, n) -h(\alpha) n}{\sigma n^{1/2}} < x \right)  = \Pf\big( J- K^{\rm G}\big)_{\mathbb{L}^2(x, \infty)} ,$$
where
$K^{G}$ is defined by the matrix kernel
\begin{equation*}
K^{G}_{11}(x,y) = K^{G}_{22}(x,y) = 0, \ \ 
K^{G}_{12}(x,y) = -K^{G}_{21}(x,y) =  \frac{1}{\sqrt{2\pi}} e^{-\frac{x^2+y^2}{4}}.
\end{equation*}
It is clear that $K^G$ is of rank one and its Fredholm Pfaffian is the cumulative distribution function of the standard Gaussian.
\subsection{Convergence of the Symplectic-Unitary transition to the GSE distribution}
\label{sec:rigorousSUtoGSE}
We have seen in Section \ref{sec:GSEasymptotics} that when $\alpha>1/2$, $K^{\rm exp}$  is the  correlation kernel of a simple point process which converges to a point process where each point has multiplicity two in the limit. The  symplectic-unitary kernel $K^{\rm  SU}$ is another example of  such kernel in the limit $\eta\to 0$. The framework that we introduced in Section \ref{sec:GSEasymptotics} -- in particular Proposition \ref{prop:approxequivalentpfaffians} -- can be used to prove the following. 
\begin{proposition}
	Let $F_{\rm SU}(x) = \Pf\big( J- K^{\rm  SU}\big)_{\mathbb{L}^2(x, \infty)}$ where $K^{\rm  SU}(x,y):=K^{\rm  SU}(1,x;1,y)$ be the cumulative distribution function of the largest eigenvalue of the symplectic-unitary transition ensemble, depending on a parameter $\eta:=\eta_1\in (0,+\infty)$.  We have that for all $x\in\R$,
	\begin{equation*}
	F_{\rm SU}(x)\xrightarrow[\eta\to 0]{} F_{\rm GSE}(x) \ \ \ \text{ and }\ \  \ F_{\rm SU}(x)\xrightarrow[\eta\to +\infty]{} F_{\rm GUE}(x).
	\end{equation*}
	\label{prop:limitcrossoverSU}
\end{proposition}
\begin{proof}
	The limit $F_{\rm SU}(x)\xrightarrow[\eta\to +\infty]{} F_{\rm GUE}(x)$ is straightforward as $K^{\rm SU}$ converges pointwise to $K^{GUE}$ and we readily check that we can apply dominated convergence in the Fredholm Pfaffian expansion.
	
	Regarding the convergence $F_{\rm SU}(x)\xrightarrow[\eta\to 0]{} F_{\rm GSE}(x)$, the limit of $I^{\rm SU}(x,y)$ when $\eta\to 0$ is straightforward as well, and we compute that
	$$ R_{22}^{\rm SU}(x,y)=   \frac{(y-x) \exp\left( \frac{(x-y)^2}{8\eta}+\frac{2\eta^3}{3}-\eta(x+y)\right)}{4\sqrt{2\pi}\eta^{3/2}}\underset{\eta\to 0}{\sim}   \frac{(y-x) \exp\left( \frac{(x-y)^2}{8\eta}\right)}{4\sqrt{2\pi} \eta^{3/2}}.$$
	Then we readily  verify that $R_{22}^{\rm SU}$ converges to the distribution $\delta'$ in the same sense as in Proposition \ref{prop:asymptoticsRGSE}. Hence, $K^{\rm SU}$ converges as $\eta\to 0$ to the kernel $K^{\infty}$ in the sense of Proposition \ref{prop:approxequivalentpfaffians}, and we conclude that
	\begin{equation*} \Pf(J-K^{\rm SU})_{\mathbb{L}^2(x, \infty)} \xrightarrow[\eta \to 0]{}\Pf(J-K^{\rm GSE})_{\mathbb{L}^2(x, \infty)}.\qedhere
	\end{equation*}
\end{proof}
\bibliographystyle{amsplain}
\bibliography{facilitated.bib}

\end{document}